\documentclass{scrartcl}

\usepackage{imakeidx}
\makeindex[name=general]

\usepackage[]{algorithm2e}
\usepackage{bbm} 
\usepackage{dsfont} 

\usepackage{commath} 

\usepackage{amsmath, amsthm, amssymb}
\usepackage{mathabx} 
\usepackage{mathrsfs}
\usepackage{url}
\usepackage{amssymb,amsfonts}
\usepackage{color}
\usepackage{shuffle}
\usepackage[usenames,dvipsnames,table]{xcolor}

\usepackage[linecolor=white,backgroundcolor=white,bordercolor=white,textsize=tiny]{todonotes}
\usepackage{breqn}

\usepackage{silence}
\usepackage[colorlinks,backref]{hyperref}
\usepackage{cleveref}
\usepackage{graphicx}
\usepackage[font={small,it}]{caption}
\usepackage[flushleft]{threeparttable}
\usepackage{longtable}
\usepackage{enumitem}

\newtheorem{theorem}{Theorem}[section]
\theoremstyle{plain}

\newtheorem{corollary}[theorem]{Corollary}

\newtheorem{example}[theorem]{Example}

\newtheorem{lemma}[theorem]{Lemma}

\newtheorem{remark}[theorem]{Remark}

\theoremstyle{definition} 

\newtheorem{definition}[theorem]{Definition}
\newtheorem*{tata}{Generalization}
  {\begin{mdframed}[backgroundcolor=lightgray]\begin{tata}}%
  {\end{tata}\end{mdframed}}

\newcommand{\R}{\mathbb{R}}

\newcommand{\C}{\mathbb{C}}
\newcommand{\N}{\mathbb{N}}
\newcommand{\Z}{\mathbb{Z}}

\newcommand{\Q}{\mathbb{Q}}

\newcommand{\eps}{\epsilon}

\newcommand{\vareps}{\varepsilon}
\newcommand{\id}{\operatorname{id}}

\newcommand{\mute}[1]{}
\let\todon\todo
\renewcommand{\todo}[1]{\todon[inline,color=green!40]{\color{NavyBlue}{#1}}}

\usepackage{stmaryrd}
\usepackage{faktor}
\usepackage{subcaption}

\newcommand\diag{\operatorname{diag}}

\newcommand\adiag{\operatorname{adiag}}

\newcommand\QSym{{\operatorname{QSym}^{(2)}}}
\newcommand\qSh{\operatorname{qSh}}
\newcommand\Zero{\mathsf{Zero}}
\newcommand\zero{\mathsf{zero}}

\newcommand{\qShuffle}{\stackrel{\scalebox{0.6}{$\mathsf{qs}$}}{\shuffle}}

\renewcommand\O{{\mathcal O}}
\DeclareMathOperator*{\multicirc}{\bigcirc}

\newcommand\splitt{\mathfrak{S}}

\newcommand\A{{\mathfrak A}} 

\newcommand\ISS{\operatorname{ISS}}

\newcommand\Endom{\operatorname{End}}
\newcommand\weight{\mathsf{weight}}
\newcommand\Span{\operatorname{span}}

\newcommand\CS{\operatorname{cumsum}}

\definecolor{myGreen}{rgb}{0.18039216 0.49803922 0.09411765}
\definecolor{oldlace}{rgb}{0.99, 0.96, 0.9}
\definecolor{moccasin}{rgb}
{0.93333333 0.49019608 0.04313725}
\definecolor{palegreen}{rgb}{0.6, 0.98, 0.6}
\definecolor{green(html/cssgreen)}{rgb}{0.0, 0.5, 0.0}
\definecolor{harlequin}{rgb}{0.25, 1.0, 0.0}
\definecolor{lasallegreen}{rgb}{0.03, 0.47, 0.19}
\definecolor{kellygreen}{rgb}{0.3, 0.73, 0.09}
\definecolor{forestgreen(web)}{rgb}{0.2, 0.8, 0.2}
\definecolor{darktangerine}{rgb}{1.0, 0.66, 0.07}
\definecolor{fashionfuchsia}{rgb}{0.96, 0.0, 0.63}
\definecolor{darkred}{rgb}{0.65, 0.0, 0.0}
\definecolor{darkRed}{rgb}{0.65, 0.0, 0.0}
\definecolor{coquelicot}{rgb}{0.9, 0.22, 0.0}
\definecolor{ballblue}{rgb}
{0.392,0.513,0.725}
\definecolor{blue-violet}{rgb}{0.54, 0.17, 0.89}
\definecolor{darkorange}{rgb}{0.894, 0.866, 0.125}
\definecolor{airforceblue}{rgb}{0.36, 0.44, 0.46}
\definecolor{tealBlue}{rgb}{0.21, 0.46, 0.53}
\definecolor{indigoWeb}{rgb}{0.29, 0.0, 0.51}
\definecolor{CeruleanBlue}{rgb}{0.16, 0.32, 0.75}
\definecolor{darkBlue}{rgb}{0.16, 0.32, 0.75}
\definecolor{darkGrey}{rgb}{0.66, 0.66, 0.66}
\definecolor{babypink}{rgb}{0.878, 0.345, 0.12549}
\definecolor{brown}{rgb}{0.45882353 0.27058824 0.00784314}

\newcommand\w[1]{{\color{cyan}\mathbf{#1}}}

\newcommand\tuIn[1]{\boldsymbol{#1}}

\newcommand\DEF[1]{\textbf{#1}\index[general]{#1}}

\newcommand\Stutter{\mathsf{warp}}

\newcommand\composition{\mathsf{Cmp}}
\newcommand\groundRing{{\mathbb{K}}}
\newcommand\monoidComp{\mathfrak{M}}
\newcommand\monCompNeutrElem{\vareps}
\newcommand\ec{\mathsf{e}} 
\newcommand\formalPorerSeriesOfFiniteDegree{\groundRing\llbracket x\rrbracket_{<\infty}}

\begin{document}

\newcommand\closureDiag{{\overline{\monoidComp_d}^{\chain}}} %
\newcommand\closureADiag{{\overline{\A^{1\times 1}}^{\adiag}}}
\newcommand\closureDiagADiag{{\overline{\A^{1\times 1}}^{\diag,\adiag}}}
\newcommand\sh{\operatorname{sh}}
\newcommand\SH{\operatorname{SH}}
\newcommand\QSH{\operatorname{QSH}}

\newcommand\e[1]{{e}_{#1}} 

\newcommand\NFzero{{\operatorname{NF}_{\mathsf{Zero}}}}
\newcommand\NFstut{{\operatorname{NF}_{\mathsf{warp}}}}
\newcommand\NFsim{{\operatorname{NF}_{\sim}}}
\newcommand\NFconst{{\operatorname{NF}_{\ker(\delta)}}}
\newcommand\EQzero{\sim_{\operatorname{Zero}}}
\newcommand\EQstut{\sim_{\operatorname{warp}}}

\newcommand\compositionConnected{\composition_{\mathsf{con}}}

\newcommand\funcFinSupport{\groundRing^{(\N\times\N)}}

\newcommand\evZ{\mathsf{evZ}}
\newcommand\evC{\mathsf{evC}}

\newcommand\row{\mathsf{rows}}
\newcommand\col{\mathsf{cols}}
\newcommand\size{\mathsf{size}}

\newcommand\domainAlphaBeta{\mathfrak{s}}

\makeatletter
\newsavebox{\@brx}
\newcommand{\llangle}[1][]{\savebox{\@brx}{\(\m@th{#1\langle}\)}%
  \mathopen{\copy\@brx\kern-0.5\wd\@brx\usebox{\@brx}}}
\newcommand{\rrangle}[1][]{\savebox{\@brx}{\(\m@th{#1\rangle}\)}%
  \mathclose{\copy\@brx\kern-0.5\wd\@brx\usebox{\@brx}}}
\makeatother

\title{Two-parameter sums signatures and corresponding quasisymmetric functions}
\author{Joscha Diehl \and Leonard Schmitz}
\date{\emph{University of Greifswald}}
\maketitle

\begin{abstract}
  Quasisymmetric functions have recently been used in time series analysis as
  polynomial features that are invariant under, so-called, dynamic time warping.
  We extend this notion to data indexed by two parameters and thus provide
  warping invariants for images.
  We show that two-parameter quasisymmetric functions are complete in a certain
  sense, and provide a two-parameter quasi-shuffle identity. A compatible coproduct is
  based on diagonal concatenation of the input data, leading to a (weak) form of
  Chen’s identity. 
\end{abstract}


\tableofcontents

\section{Introduction}

Certain forms of \emph{signatures} have proven beneficial as features in time series analysis.
The \emph{iterated-integrals signature} was introduced by Chen in the 1950s \cite{chen1954iterated}
for homological considerations on loop space.
After applications in control theory, starting in the 1970s, \cite{fliess1976outil,fliess1981fonctionnelles} and rough path theory, starting in the 1990s, \cite{lyons1998differential},
it has in the last decade been successfully applied in machine learning tasks on time series \cite{BKA19,kiraly2019kernels,diehl2019invariants,KMFL20,TBO20}.
As the name suggests, it applies to \emph{continuous} objects, namely (smooth enough) curves in Euclidean space.
For discrete time series to fit in the machinery,
they have to undergo a (simple) interpolation step.

The \emph{iterated-sums signature}, introduced in \cite{Diehl_2020},
forgoes this intermediate step and immediately works on the discrete-time object.
This discrete perspective brings additional benefits:
a broader class of features (even for one-dimensional time series, whose \emph{integral} signature is trivial),
flexibility in the underlying ground field \cite{diehl2020tropical},
and a tight, new-found, connection to the theory of quasisymmetric functions \cite{malvenuto1995duality} and \emph{dynamic time warping} \cite{SC1978,berndt1994using,keogh2005exact}.

In the present work we will take the latter perspective and apply it to data indexed by \emph{two parameters},
the canonical example being image data.

\bigskip

\textbf{Related work}

In data science, two
recent works have very successfully
applied iterated integrals to images.
In \cite{ibrahim2022imagesig}
the classical, one-parameter, iterated-integrals signature is used for images (by working ``row-by-row''), whereas
certain multi-parameter iterated integrals are used in \cite{ZLT22}.
A principled extension of Chen's iterated integrals,
based on their original use in topology,
is presented in \cite{GLNO2022}.

More generally, the use of ``signature-like''
feature-maps has recently been extended to
graphs \cite{toth2022capturing,caudillo2022graph}
and trees  \cite{cochrane2021sk}.

\bigskip

\index[general]{n@$\N$, strictly positive integers}
\index[general]{n0@$\N_0$, non-negative integers}
\index[general]{composition@$\circ$, composition of functions}
\index[general]{row@$\row$}
\index[general]{col@$\col$}
\index[general]{size@$\size$}
\index[general]{C@$\C$, complex number field}

\textbf{Notation}

Throughout,
$\N = \{1,2,\dots\}$ denotes the strictly positive integers and
$\N_0 = \{0\}\cup\N$ denotes the non-negative integers.
Let $(\N^2,\leq)$ denote the product poset (partially ordered set),
i.e.
(here, and throughout, we denote tuples with bold letters)
\begin{align*}
    \tuIn{i}\leq\tuIn{j}  \Leftrightarrow \tuIn{i}_1\leq\tuIn{j}_1 \text{ and } \tuIn{i}_2\leq\tuIn{j}_2.
\end{align*}
For every matrix $\mathbf{A}\in M^{m\times n}$ with entries from an arbitrary set $M$ let $\size(\mathbf{A}):=(\row(\mathbf{A}),\col(\mathbf{A})):=(m,n)\in \N^2$ denote the number of rows and columns in $\mathbf{A}$ respectively. 
 Let $g\circ f:M\rightarrow P,\;m\mapsto g(f(m))$ be the set-theoretic composition of  functions $f:M\rightarrow N$ and $g:N\rightarrow P$. 
Let $\C$ denote the complex number field.

  \subsection{Warping invariants motivate the signature}\label{sec:inroInvariants}
 
We briefly recall the notion of  (classical, one-parameter) time warping invariants, as covered in \cite{Diehl_2020}. 
For simplicity, we consider eventually-constant, $\C$-valued \DEF{time series} in discrete time, 
\begin{align*}\evC(\N,\C)&:=
\left\{x:\N\rightarrow\C\mid \exists n\in\N\,:x_i\not=x_{j}\implies i\leq n\right\}.
\end{align*}

Intuitively one might think of complex numbers as colored pixels, becoming especially valuable for visualization in the two-parameter case.
Later in this paper, the complex numbers are actually replaced by the module $\groundRing^d$ over some arbitrary integral domain, covering the classical encoding of colors via $\R^3$. 
A single \DEF{time warping operation} is formalized by the mapping 
  $$\Stutter_{k}:\evC(\N,\C)\rightarrow\evC(\N,\C),\;{(\Stutter_{k}x)}_i
    :=
    \begin{cases}
      x_i       & i \le k\\
      x_{i-1} & i > k,
    \end{cases}$$
    which leaves all entries until the $k$-th unchanged, copies this value once, attaches it at position $k+1$, and shifts all remaining successors by one. 

  For example, consider  
  \begin{align*}
\Stutter_{2}\!\left(\,\begin{tabular}{|llll}
    \hline
    \cellcolor{ballblue}2
   &\cellcolor{moccasin}1
   &\cellcolor{myGreen}3
   &\cellcolor{moccasin}1\\
    \hline
\end{tabular}\cdots\right)
&=
\begin{tabular}{|>{\columncolor{red}}lllll}
    \hline
    \cellcolor{ballblue}2
   &\cellcolor{moccasin}1
   &\cellcolor{moccasin}1
   &\cellcolor{myGreen}3
   &\cellcolor{moccasin}1\\
    \hline
\end{tabular}\cdots
\end{align*}
where the dots on the right hand side indicate that all relevant information is provided, i.e., that the series has reached a  constant and will not change again. 
A \DEF{time warping invariant} is a function from the set of time series to the complex numbers which remains unchanged under warping.
An example of such an invariant is 
\begin{equation}\label{eq:exOneDimInv}
 \varphi:\evC(\N,\C)\rightarrow\C,\;x\mapsto x_{1}-\lim\limits_{t\rightarrow\infty}x_{t}
 \end{equation}
where the limit exists, since $x$ was assumed to be eventually constant. 
This invariant does not ``see'', whether certain entries of a time series are repeated over and over again.  
Indeed,
neither the first entry nor the limt can be changed by any warping. 
In the numerical example
\begin{align*}
\!\!\!\!\!\!\varphi\!\left(\,\begin{tabular}{|llllll}
    \hline
    \cellcolor{ballblue}2
   &\cellcolor{myGreen}3
   &\cellcolor{moccasin}1
   &\cellcolor{darkorange}5
   &\cellcolor{moccasin}1
   &\cellcolor{moccasin}1\\
    \hline
\end{tabular}\cdots\right)
&=\varphi\!\left(\,
\begin{tabular}{|>{\columncolor{red}}llllllll}
    \hline
    \cellcolor{ballblue}2
   &\cellcolor{myGreen}3
   &\cellcolor{myGreen}3
   &\cellcolor{moccasin}1
   &\cellcolor{moccasin}1
   &\cellcolor{moccasin}1
   &\cellcolor{darkorange}5
   &\cellcolor{moccasin}1\\
    \hline
\end{tabular}\cdots\right)\\
&=\varphi\!\left(\,
\begin{tabular}{|>{\columncolor{red}}lllllll}
    \hline
    \cellcolor{ballblue}2
   &\cellcolor{ballblue}2
   &\cellcolor{ballblue}2
   &\cellcolor{myGreen}3
   &\cellcolor{moccasin}1
   &\cellcolor{darkorange}5
   &\cellcolor{moccasin}1\\
    \hline
\end{tabular}\cdots\right)\\
&=\varphi\!\left(\,
\begin{tabular}{|>{\columncolor{red}}llllllll}
    \hline
    \cellcolor{ballblue}2
   &\cellcolor{myGreen}3
   &\cellcolor{myGreen}3
   &\cellcolor{moccasin}1
   &\cellcolor{darkorange}5
   &\cellcolor{darkorange}5
   &\cellcolor{darkorange}5
   &\cellcolor{moccasin}1\\
    \hline
\end{tabular}\cdots\right)\\
&=
\begin{tabular}{|l|}
    \hline
    \cellcolor{ballblue}2\\
    \hline
\end{tabular}
-
\begin{tabular}{|l|}
    \hline
    \cellcolor{moccasin}1\\
    \hline
\end{tabular}
\end{align*}
the initial time series is warped to three different representatives and yet still yields the same value under $\varphi$. 

Next, we move to the two-parameter case, which is the focus of this paper.   
  
We denote by
\begin{align*}\evC(\N^2,\C)&:=
\left\{X:\N^2\rightarrow\C\mid \exists \tuIn{n}\in\N^2\,:X_{\tuIn{i}}\not=X_{\tuIn{j}}\implies \tuIn{i}\leq \tuIn{n}\right\}
\end{align*}
the set of two-parameter functions which are eventually constant. 
A function from this set is a two-parameter analog to a (classical, one-parameter) time series that is eventually constant
and can be thought of as an image of arbitrary size, with its pixels being encoded by $\C$.  

We define a single \DEF{warping operation}
$\Stutter_{a,k}$ similar to the one-parameter case,
except that we add a second index $a\in\{1,2\}$ indicating on which axis the warping takes place. 
For the axis $a=1$ we obtain an operation on rows, i.e., at position $k$ we copy a row and shift all remaining rows by one. 
Illustratively, for $k=2$ we have
\begin{align*}
\Stutter_{1,2}\!\left(\;
{\begin{tabular}{lllll}
\hline
\multicolumn{1}{|l}{
\cellcolor{ballblue}2
}
   &\cellcolor{moccasin}1
   &\cellcolor{myGreen}3
   &\cellcolor{moccasin}1
   &\cellcolor{moccasin}1\\
\multicolumn{1}{|l}{
\cellcolor{myGreen}3}
   &\cellcolor{ballblue}2
   &\cellcolor{darkorange}5
   &\cellcolor{moccasin}1
   &\cellcolor{moccasin}1\\
\multicolumn{1}{|l}{
   \cellcolor{moccasin}1}
   &\cellcolor{moccasin}1
   &\cellcolor{moccasin}1
   &\cellcolor{moccasin}1
   &\cellcolor{moccasin}1\\
   &
\end{tabular}}_{\;\ddots\;}
\right)
&=\;
{\begin{tabular}{lllll}
\hline
\multicolumn{1}{|l}{
\cellcolor{ballblue}2
}
   &\cellcolor{moccasin}1
   &\cellcolor{myGreen}3
   &\cellcolor{moccasin}1
   &\cellcolor{moccasin}1\\
\multicolumn{1}{|l}{
\cellcolor{myGreen}3}
   &\cellcolor{ballblue}2
   &\cellcolor{darkorange}5
   &\cellcolor{moccasin}1
   &\cellcolor{moccasin}1\\
\multicolumn{1}{|l}{
\cellcolor{myGreen}3}
   &\cellcolor{ballblue}2
   &\cellcolor{darkorange}5
   &\cellcolor{moccasin}1
   &\cellcolor{moccasin}1\\
\multicolumn{1}{|l}{
   \cellcolor{moccasin}1}
   &\cellcolor{moccasin}1
   &\cellcolor{moccasin}1
   &\cellcolor{moccasin}1
   &\cellcolor{moccasin}1\\
   &
\end{tabular}}_{\;\ddots\,}
\end{align*}
whereas for the axis $a=2$ we get the warping of columns, illustrated by  
\begin{align*}
\Stutter_{2,2}\!\left(\;
{\begin{tabular}{lllll}
\multicolumn{1}{|l}{
\cellcolor{ballblue}2
}
   &\cellcolor{moccasin}1
   &\cellcolor{myGreen}3
   &\cellcolor{ballblue}2
   &\cellcolor{ballblue}2\\
\multicolumn{1}{|l}{
\cellcolor{myGreen}3}
   &\cellcolor{gray}0
   &\cellcolor{darkorange}5
   &\cellcolor{ballblue}2
   &\cellcolor{ballblue}2\\
\multicolumn{1}{|l}{
   \cellcolor{ballblue}2}
   &\cellcolor{ballblue}2
   &\cellcolor{ballblue}2
   &\cellcolor{ballblue}2
   &\cellcolor{ballblue}2\\
   &
\end{tabular}}_{\;\ddots\;}
\right)
&=\;
{\begin{tabular}{llllll}
\multicolumn{1}{|l}{
\cellcolor{ballblue}2
}
   &\cellcolor{moccasin}1
   &\cellcolor{moccasin}1
   &\cellcolor{myGreen}3
   &\cellcolor{ballblue}2
   &\cellcolor{ballblue}2\\
\multicolumn{1}{|l}{
\cellcolor{myGreen}3}
   &\cellcolor{gray}0
   &\cellcolor{gray}0
   &\cellcolor{darkorange}5
   &\cellcolor{ballblue}2
   &\cellcolor{ballblue}2\\
\multicolumn{1}{|l}{
   \cellcolor{ballblue}2}
   &\cellcolor{ballblue}2
   &\cellcolor{ballblue}2
   &\cellcolor{ballblue}2
   &\cellcolor{ballblue}2
   &\cellcolor{ballblue}2\\
   &
\end{tabular}}_{\;\ddots\,}
\end{align*}

also at position $k=2$.
A formal definition of $\Stutter_{a,k}$ is provided in \Cref{sec:invariantsWarping}.

We call a function $\psi:\evC(\N^2,\C)\rightarrow\C$ an \DEF{invariant to warping} (in both directions independently), if it remains unchanged under warping, i.e., 
$$\psi\circ\Stutter_{a,k}=\psi \quad\forall (a,k)\in\{1,2\}\times\N.$$
\Cref{fig:trees}
illustrates eight two-parameter functions (i.e. images),
obtained from three initial functions by repeatedly applying warping.%
\footnote{In the sense of \Cref{def:StuttInv} each of the eight images is \emph{equivalent} to one of the initial three.}
Hence, the value of any warping invariant will coincide on pairwise equivalent inputs.

\begin{figure}
\centering
     \begin{subfigure}[b]{0.23\textwidth}
         \centering
         \includegraphics[height=0.95\textwidth]{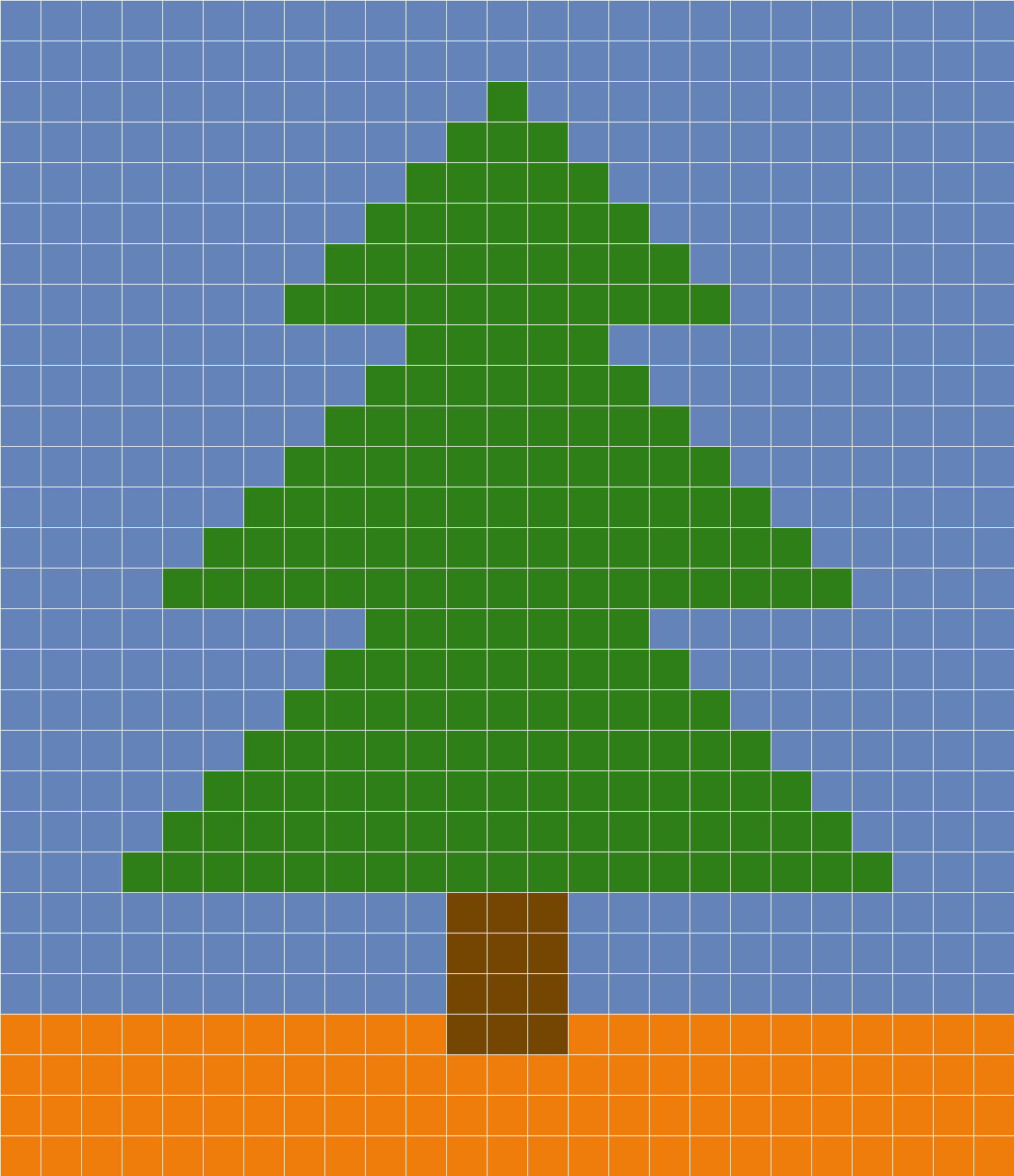}
         \caption{}
        \label{fig:pica}
     \end{subfigure}
     \hfill
     \begin{subfigure}[b]{0.23\textwidth}
         \centering
         \includegraphics[height=0.95\textwidth]{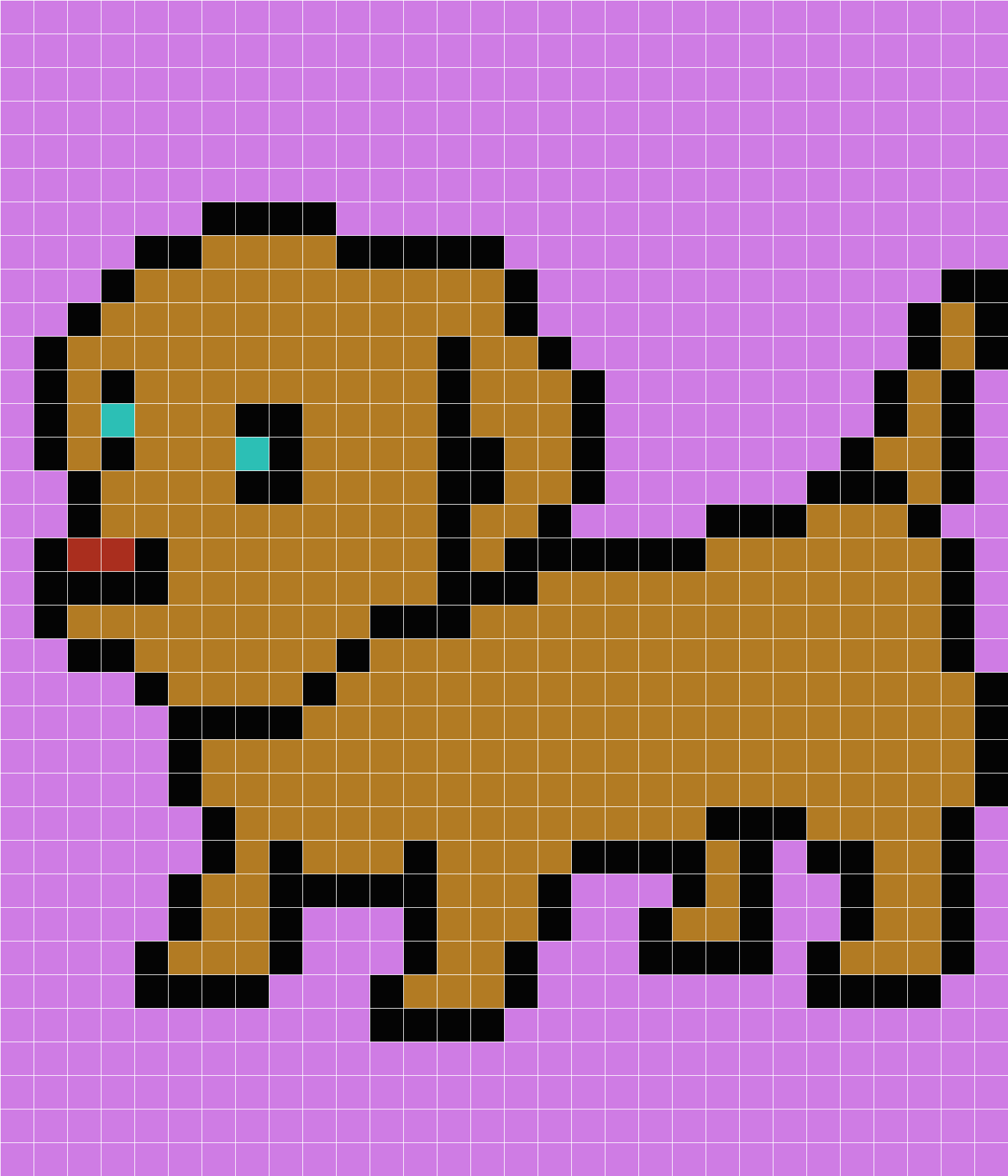}
         \caption{}
        \label{fig:picb}
     \end{subfigure}
     \hfill
     \begin{subfigure}[b]{0.23\textwidth}
         \centering
         \includegraphics[height=0.95\textwidth]{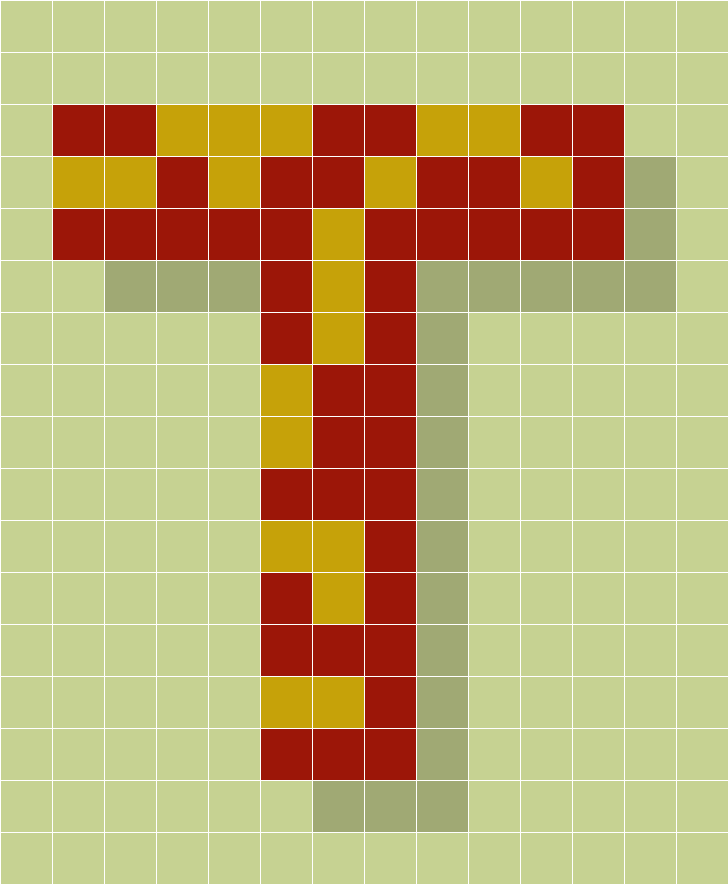}
         \caption{}
     \end{subfigure}
     \hfill
     \begin{subfigure}[b]{0.23\textwidth}
         \centering
         \includegraphics[height=0.95\textwidth]{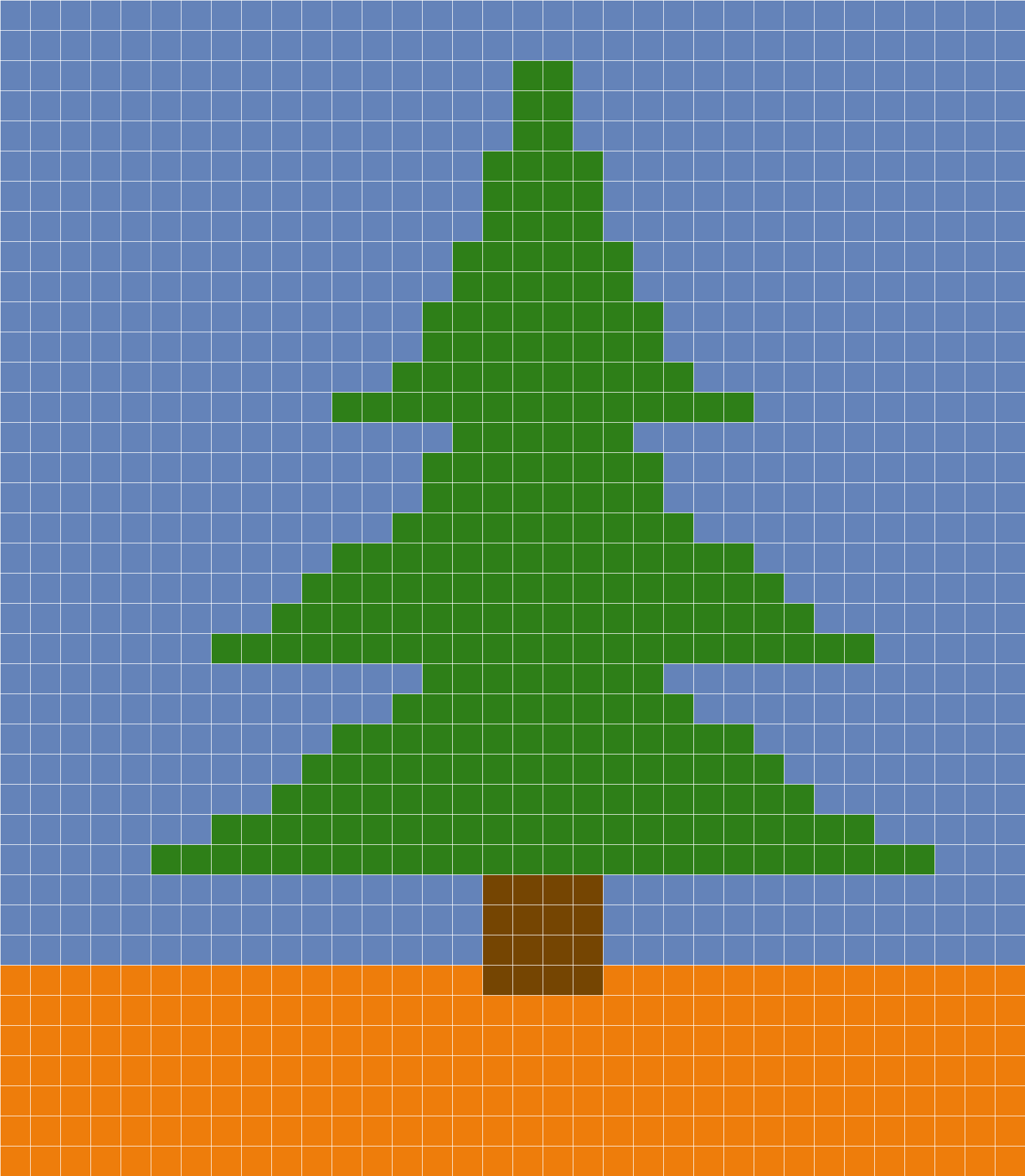}
         \caption{}
     \end{subfigure}
     \newline\\
     \centering
     \begin{subfigure}[b]{0.23\textwidth}
         \centering
         \includegraphics[height=0.95\textwidth]{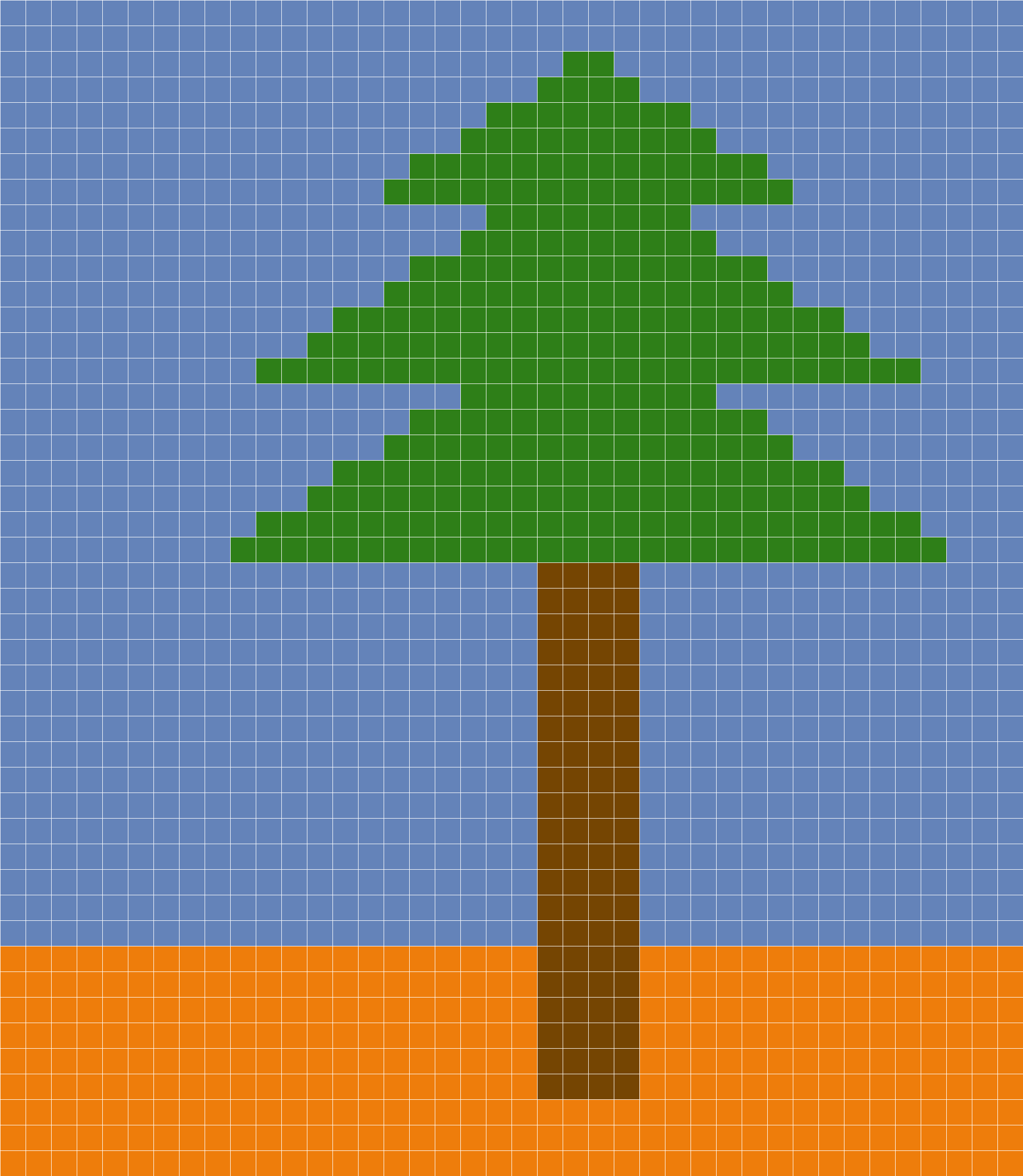}
         \caption{}
     \end{subfigure}
     \hfill
     \begin{subfigure}[b]{0.23\textwidth}
         \centering
         \includegraphics[height=0.95\textwidth]{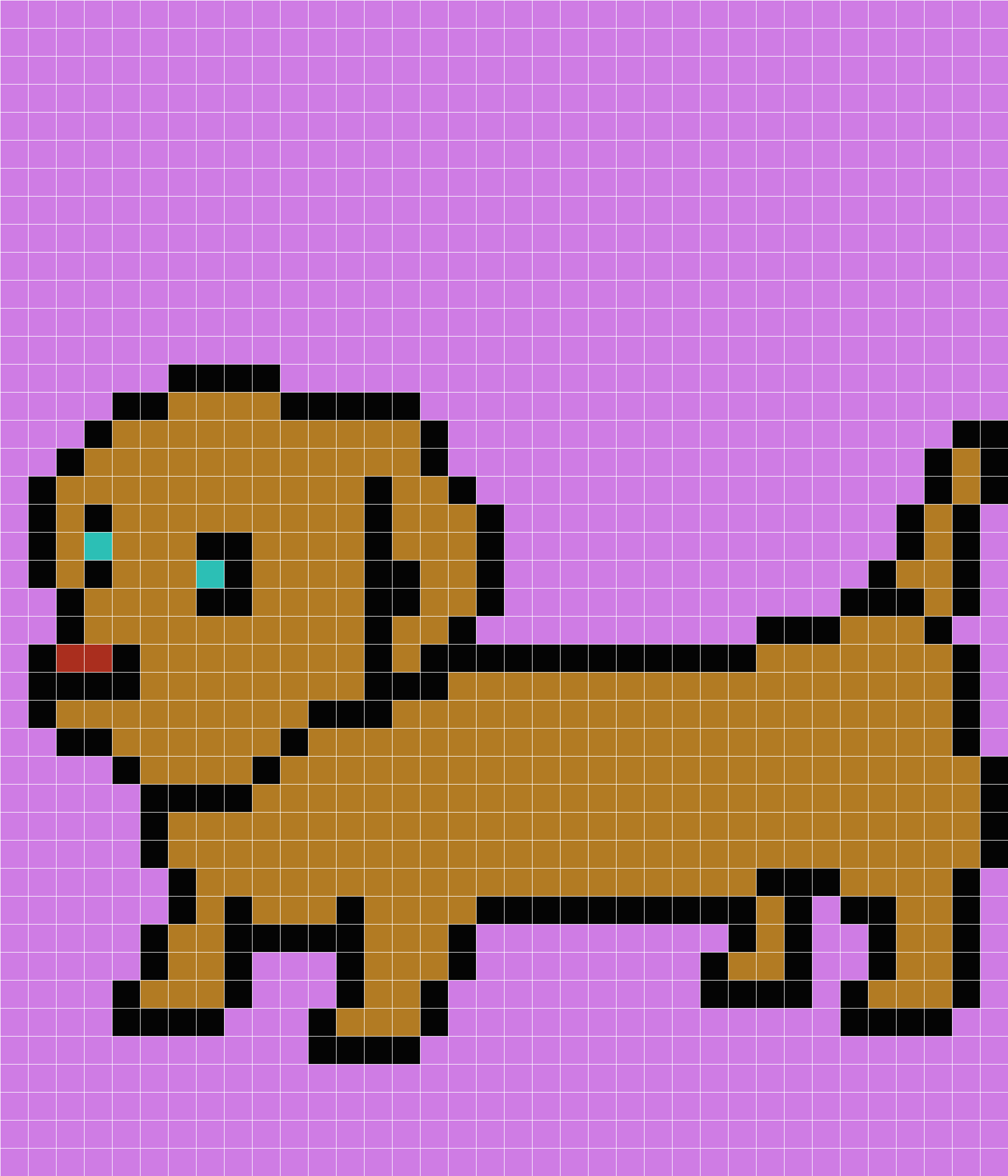}
         \caption{}
     \end{subfigure}
     \hfill
     \begin{subfigure}[b]{0.23\textwidth}
         \centering
         \includegraphics[height=0.95\textwidth]{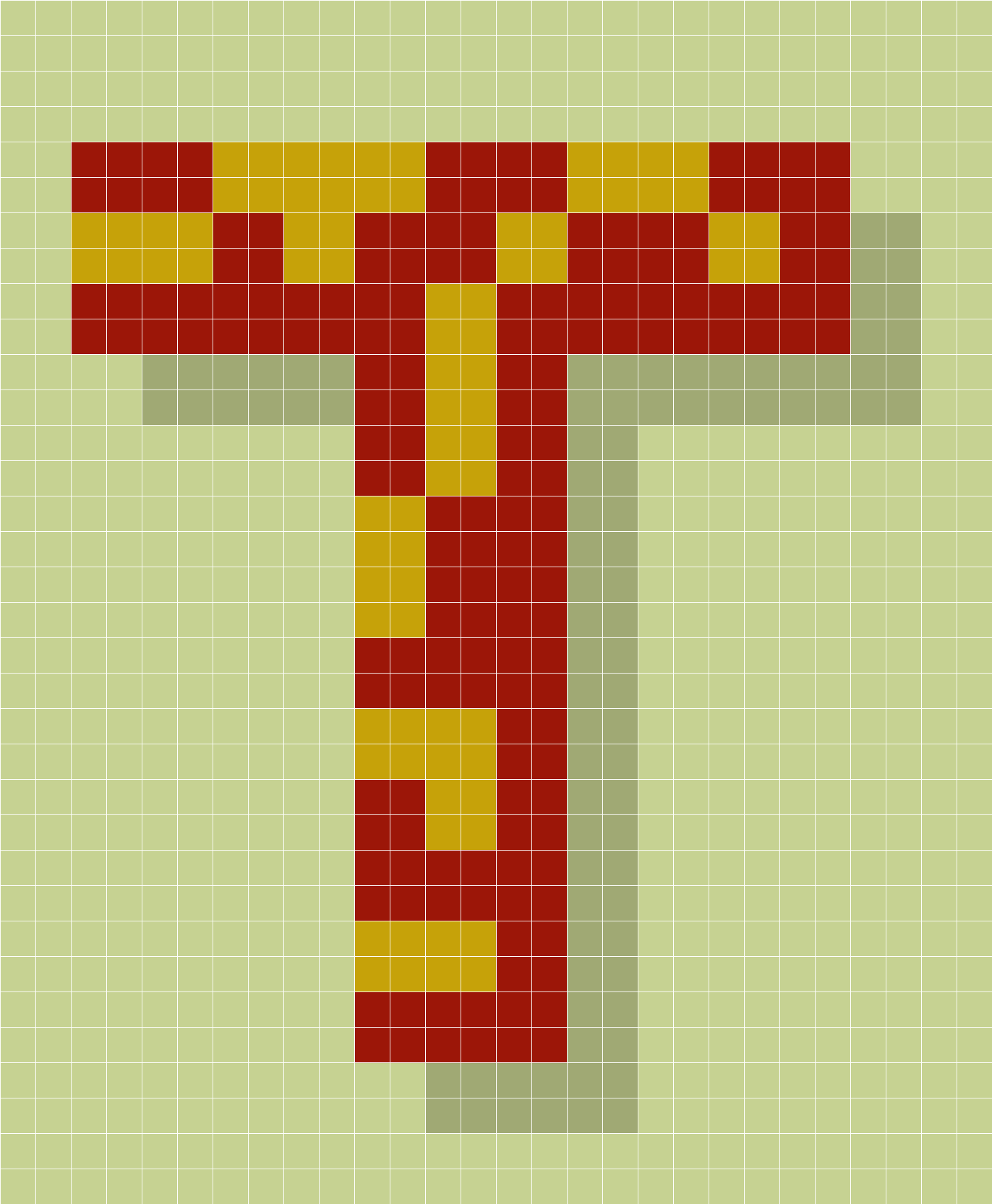}
         \caption{}
     \end{subfigure}
     \hfill
     \begin{subfigure}[b]{0.23\textwidth}
         \centering
         \includegraphics[height=0.95\textwidth]{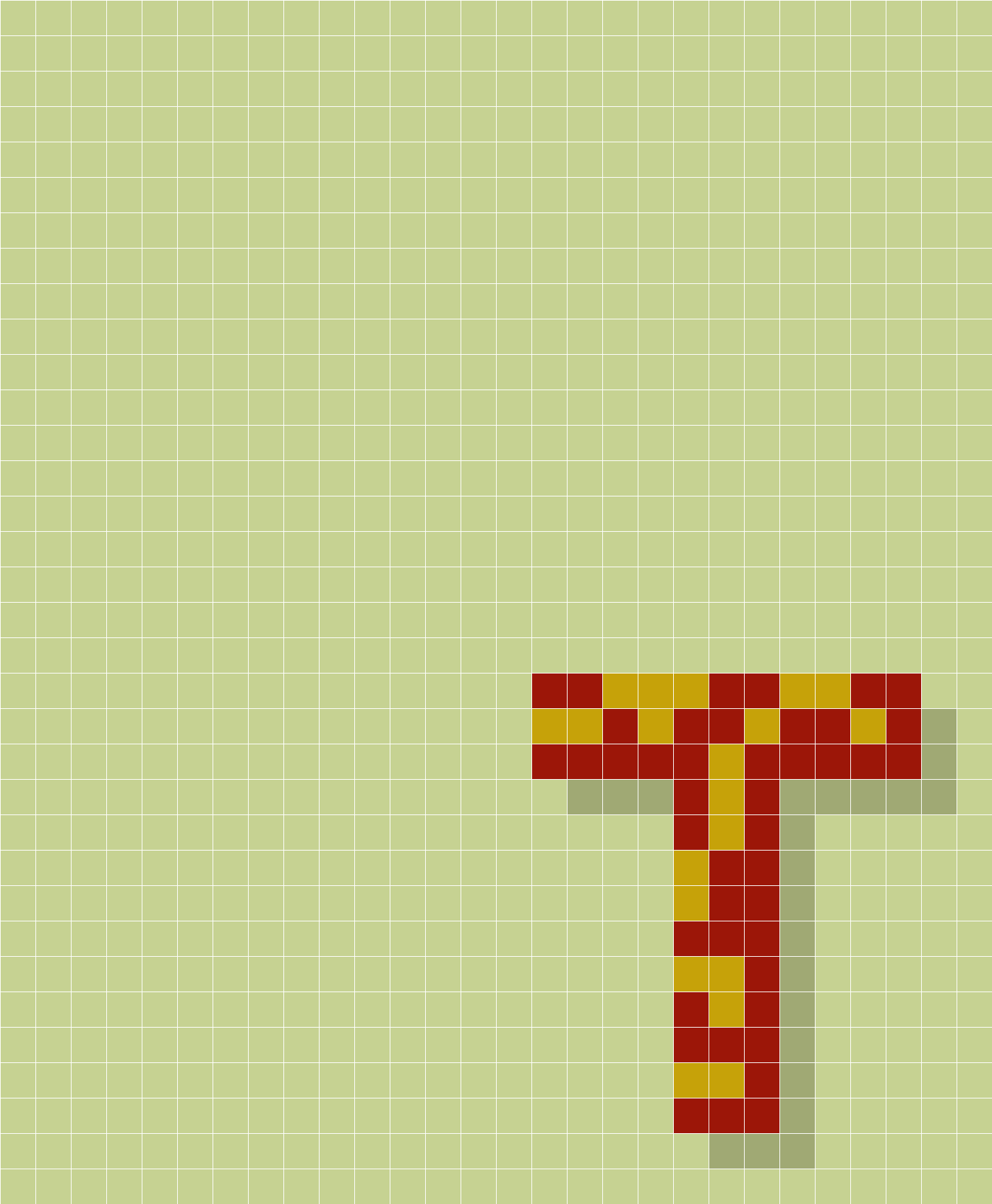}
         \caption{}
     \end{subfigure}
\caption{Functions in $\evC(\N^2,\C)$.
The images (a), (d) and (e) are equivalent under warping,
so are (b) and (f), as well as (c), (g) and (h).}
\label{fig:trees}
\end{figure}

As an example, consider the two-parameter analogon of \Cref{eq:exOneDimInv}, 
\begin{equation}\label{eq:exampleWarpingInvariant2dims}
\Psi_{\begin{scriptsize}\begin{bmatrix}1\end{bmatrix}\end{scriptsize}}:\evC(\N^2,\C)\rightarrow\C,\;X\mapsto X_{1,1}-\lim\limits_{s,t\rightarrow\infty}X_{s,t}
\end{equation}
which is a warping invariant.
In fact it belongs to an entire class $\Psi$ of invariants constructed as follows.

An $\N_0$-valued matrix is called a \DEF{matrix composition}, if it has no  zero lines or zero columns. 
\renewcommand\SS{\mathsf{SS}}
For every eventually-zero $Z\in\evC(\N^2,\C)$, that is $\lim_{s,t\rightarrow\infty}Z_{s,t}=0$, we define the \DEF{two-parameter sums signature} coefficient of $Z$ at the matrix composition $\mathbf{a}$ via 
\begin{align*}
\langle\SS(Z),\mathbf{a}\rangle:=\sum_{\substack{\iota_1<\cdots<\iota_{\row(\mathbf{a})}\\\kappa_1<\cdots<\kappa_{\col(\mathbf{a})}}}\;\prod_{s=1}^{\row(\mathbf{a})}\prod_{t=1}^{\col(\mathbf{a})}{ Z}_{\iota_s,\kappa_t}^{\mathbf{a}_{s,t}}\in\C.
\end{align*}
Note that this sum over strictly increasing chains $\iota$ and $\kappa$ is always finite since $Z$ is zero almost everywhere.  
A numerical illustration is provided in \Cref{ex:numericalExSS}.

We collect the second ingredient for warping invariants. 
For $X\in\evC(\N^2,\C)$ let $\delta X\in\evC(\N^2,\C)$ be defined via the forward \DEF{difference operator}
\begin{equation*}
{(\delta X)}_{i,j}:=X_{i+1,j+1}-X_{i+1,j}-X_{i,j+1}+X_{i,j}. 
\end{equation*}
Note that $\delta X$ is eventually zero, since each index $\tuIn{i}$ for which $X_{\tuIn{i}}$ has no changing successors yields ${(\delta X)}_{\tuIn{i}}=0$. 
Putting all together, we consider a matrix composition $\mathbf{a}$ as an index and obtain an  $\mathbf{a}$-indexed family of functions
\begin{align*}
\Psi_{\mathbf{a}}:\evC(\N^2,\C)\rightarrow\C,X\mapsto \langle\SS(\delta X),\mathbf{a}\rangle
\end{align*}
which are invariant to warping in both directions independently (\Cref{theorem:invariants}). 
For example, we have 
\begin{align*}
\Psi_{\mathbf{a}}\!\left(
\begin{scriptsize}
\;{\begin{tabular}{lllll}
\hline
\multicolumn{1}{|l}{
\cellcolor{ballblue}2}
   &\cellcolor{ballblue}2
   &\cellcolor{ballblue}2
   &\cellcolor{ballblue}2
   &\cellcolor{gray}0\\
\multicolumn{1}{|l}{
\cellcolor{ballblue}2}
   &\cellcolor{myGreen}3
   &\cellcolor{myGreen}3
   &\cellcolor{ballblue}2
   &\cellcolor{gray}0\\
\multicolumn{1}{|l}{
\cellcolor{ballblue}2}
   &\cellcolor{myGreen}3
   &\cellcolor{myGreen}3
   &\cellcolor{ballblue}2
   &\cellcolor{gray}0\\
   \multicolumn{1}{|l}{
\cellcolor{ballblue}2}
   &\cellcolor{myGreen}3
   &\cellcolor{myGreen}3
   &\cellcolor{myGreen}3
   &\cellcolor{gray}0\\
   \multicolumn{1}{|l}{
\cellcolor{ballblue}2}
   &\cellcolor{ballblue}2
      &\cellcolor{brown}4
   &\cellcolor{ballblue}2
   &\cellcolor{gray}0\\
\multicolumn{1}{|l}{
   \cellcolor{moccasin}1}
   &\cellcolor{moccasin}1
      &\cellcolor{brown}4
   &\cellcolor{moccasin}1
   &\cellcolor{gray}0\\
   \multicolumn{1}{|l}{
   \cellcolor{gray}0}
        &\cellcolor{gray}0
     &\cellcolor{gray}0
      &\cellcolor{gray}0
   &\cellcolor{gray}0\\
   &
\end{tabular}}_{\;\ddots\,}
\end{scriptsize}
\right)
=
\Psi_{\mathbf{a}}\!\left(
\begin{scriptsize}
\;{\begin{tabular}{llllll}
\hline
\multicolumn{1}{|l}{
\cellcolor{ballblue}2}
   &\cellcolor{ballblue}2
   &\cellcolor{ballblue}2
   &\cellcolor{ballblue}2
   &\cellcolor{ballblue}2
   &\cellcolor{gray}0\\
\multicolumn{1}{|l}{
\cellcolor{ballblue}2}
   &\cellcolor{myGreen}3
   &\cellcolor{myGreen}3
   &\cellcolor{ballblue}2
   &\cellcolor{ballblue}2
   &\cellcolor{gray}0\\
   \multicolumn{1}{|l}{
\cellcolor{ballblue}2}
   &\cellcolor{myGreen}3
   &\cellcolor{myGreen}3
   &\cellcolor{myGreen}3
   &\cellcolor{myGreen}3
   &\cellcolor{gray}0\\
      \multicolumn{1}{|l}{
\cellcolor{ballblue}2}
   &\cellcolor{ballblue}2
      &\cellcolor{brown}4
   &\cellcolor{ballblue}2
   &\cellcolor{ballblue}2
   &\cellcolor{gray}0\\
   \multicolumn{1}{|l}{
\cellcolor{ballblue}2}
   &\cellcolor{ballblue}2
      &\cellcolor{brown}4
   &\cellcolor{ballblue}2
   &\cellcolor{ballblue}2
   &\cellcolor{gray}0\\
      \multicolumn{1}{|l}{
\cellcolor{ballblue}2}
   &\cellcolor{ballblue}2
      &\cellcolor{brown}4
   &\cellcolor{ballblue}2
   &\cellcolor{ballblue}2
   &\cellcolor{gray}0\\
\multicolumn{1}{|l}{
   \cellcolor{moccasin}1}
   &\cellcolor{moccasin}1
      &\cellcolor{brown}4
   &\cellcolor{moccasin}1
   &\cellcolor{moccasin}1
   &\cellcolor{gray}0\\
   \multicolumn{1}{|l}{
   \cellcolor{gray}0}
        &\cellcolor{gray}0
     &\cellcolor{gray}0
      &\cellcolor{gray}0
      &\cellcolor{gray}0
   &\cellcolor{gray}0\\
   &
\end{tabular}}_{\;\ddots\,}
\end{scriptsize}
\right)
\end{align*}
for every matrix composition $\mathbf{a}$. 
 Note that the invariant from  \Cref{eq:exampleWarpingInvariant2dims} is indeed ``signature-induced'', verified by \Cref{ex:polynomial_invariant}.

In \Cref{corr:suffLarge} and \Cref{theorem:ssIffStuttInv}
we argue that the family $\Psi$ is sufficiently large. 
In one dimension, they are precisely given by the polynomial invariants (\Cref{the:identificationPolynomailInvariants}). 
In \Cref{subsection:IteratedSScoeff} we provide a certain subfamily of $\Psi$ which can be computed iteratively in linear time.

\subsection{Illustrations of quasi-shuffle relations  and of Chen's identity}\label{subsection:introQshAndChen}

The sums signature introduced in the previous section, and properly defined in \Cref{def:ss},
satisfies two crucial properties. 

First, polynomial expressions in its entries can be rewritten
as \emph{linear} expressions of other entries.
For instance, the invariant \eqref{eq:exampleWarpingInvariant2dims}, when squared, can be rewritten as the following linear combination
 \begin{align*}
\Psi_{\begin{scriptsize}
\begin{bmatrix}1\end{bmatrix}
\end{scriptsize}}
\cdot
\Psi_{\begin{scriptsize}
\begin{bmatrix}1\end{bmatrix}
\end{scriptsize}}
 =
\Psi_{\begin{scriptsize}
\begin{bmatrix}2\end{bmatrix}
\end{scriptsize}}
 +2\,
\Psi_{\begin{scriptsize}
 \begin{bmatrix}1\\1\end{bmatrix}
\end{scriptsize}}
 +2\,
\Psi_{\begin{scriptsize}
 \begin{bmatrix}1&1\end{bmatrix}
\end{scriptsize}}
 +2\,
\Psi_{\begin{scriptsize}
 \begin{bmatrix}1&0\\0&1\end{bmatrix}
\end{scriptsize}}
 +2\,
\Psi_{\begin{scriptsize}
 \begin{bmatrix}0&1\\1&0\end{bmatrix}
\end{scriptsize}}.
 \end{align*}
The two-parameter \emph{quasi-shuffle} for matrix compositions,
which appears on the right-hand side here,
is defined in \Cref{def:twodim_qshuffle}.
This allows us to formulate a \emph{quasi-shuffle identity} for two-parameter signatures, as stated in \Cref{lem:quasiShuffleRel}. 
We put our results into context to the (classical, one-parameter) quasi-shuffle of words, and we suggest an algorithm (\Cref{sec:algo2dim}) for an efficient computation.

Second, the sums signature satisfies a kind of
\emph{Chen's identity.}
For two-parameter data, there is no single choice of concatenation,
and we focus on the case of \emph{diagonal} concatenation.
The general statement is formulated in \Cref{thm:chen} and \Cref{chen:diff},
after properly introducing diagonal concatenation for functions with domain
$\N^2$. 
Concerning the invariants $\Psi_{\mathbf{a}}$ from above, we obtain a formula
to compute $\Psi_{\mathbf{a}}(X)$ with $X\in\evC(\N^2,\C)$ via certain subparts
of $X$ and $\mathbf{a}$, specified by the deconcatenations of $X$ and
$\mathbf{a}$ in the sense of
\Cref{def:diagonalConcatenation,def:concatDiagBox,def:coproduct}, respectively. 
As a numerical example, we compute 
\begin{align*}
\Psi_{\mathbf{a}}\!\left(
\begin{scriptsize}
\;{\begin{tabular}{llllll}
\hline
\multicolumn{1}{|l}{
\cellcolor{babypink}7}
   &\cellcolor{ballblue}2
   &\cellcolor{ballblue}2
   &\cellcolor{gray}0\\
\multicolumn{1}{|l}{
\cellcolor{darkorange}5}
   &\cellcolor{ballblue}2
   &\cellcolor{ballblue}2
   &\cellcolor{gray}0\\
\multicolumn{1}{|l}{
\cellcolor{ballblue}2}
   &\cellcolor{ballblue}2
   &\cellcolor{ballblue}2
   &\cellcolor{gray}0\\
\multicolumn{1}{|l}{
\cellcolor{ballblue}2}
   &\cellcolor{ballblue}2
   &\cellcolor{myGreen}3
   &\cellcolor{gray}0\\
\multicolumn{1}{|l}{
\cellcolor{ballblue}2}
   &\cellcolor{ballblue}2
   &\cellcolor{myGreen}3
   &\cellcolor{gray}0\\
   \multicolumn{1}{|l}{
\cellcolor{ballblue}2}
   &\cellcolor{ballblue}2
   &\cellcolor{myGreen}3
   &\cellcolor{gray}0\\
   \multicolumn{1}{|l}{
\cellcolor{ballblue}2}
   &\cellcolor{ballblue}2
      &\cellcolor{brown}4
   &\cellcolor{gray}0\\
\multicolumn{1}{|l}{
   \cellcolor{moccasin}1}
   &\cellcolor{moccasin}1
      &\cellcolor{brown}4
   &\cellcolor{gray}0\\
   \multicolumn{1}{|l}{
   \cellcolor{gray}0}
     &\cellcolor{gray}0
      &\cellcolor{gray}0
   &\cellcolor{gray}0\\
   &
\end{tabular}}_{\;\ddots\,}
\end{scriptsize}
\right)
=
\sum\limits_{\diag(\mathbf{b},\mathbf{c})=\mathbf{a}}
\!\Psi_{\mathbf{b}}\!\left(
\begin{scriptsize}
\;{\begin{tabular}{ll}
\hline
\multicolumn{1}{|l}{
\cellcolor{babypink}7}
   &\cellcolor{ballblue}2\\
\multicolumn{1}{|l}{
\cellcolor{darkorange}5}
   &\cellcolor{ballblue}2\\
\multicolumn{1}{|l}{
\cellcolor{ballblue}2}
   &\cellcolor{ballblue}2\\
   &
\end{tabular}}_{\;\ddots\,}
\end{scriptsize}
\right)
\Psi_{\mathbf{c}}\!\left(
\begin{scriptsize}
\;{\begin{tabular}{lllll}
\hline
\multicolumn{1}{|l}{
\cellcolor{ballblue}2}
   &\cellcolor{ballblue}2
   &\cellcolor{gray}0\\
\multicolumn{1}{|l}{
\cellcolor{ballblue}2}
   &\cellcolor{myGreen}3
   &\cellcolor{gray}0\\
\multicolumn{1}{|l}{
\cellcolor{ballblue}2}
   &\cellcolor{myGreen}3
   &\cellcolor{gray}0\\
   \multicolumn{1}{|l}{
\cellcolor{ballblue}2}
   &\cellcolor{myGreen}3
   &\cellcolor{gray}0\\
   \multicolumn{1}{|l}{
\cellcolor{ballblue}2}
      &\cellcolor{brown}4
   &\cellcolor{gray}0\\
\multicolumn{1}{|l}{
   \cellcolor{moccasin}1}
      &\cellcolor{brown}4
   &\cellcolor{gray}0\\
   \multicolumn{1}{|l}{
   \cellcolor{gray}0}
      &\cellcolor{gray}0
   &\cellcolor{gray}0\\
   &
\end{tabular}}_{\;\ddots\,}
\end{scriptsize}
\right)
\end{align*}
via the upper left and lower right subpart of $X$, and all diagonal submatrices of composition $\mathbf{a}$. 
Note that this formula includes an empty composition $\ec$ for which $\Psi_{\ec}$ is constant with image $\{1\}$. 

\Cref{section:2dimSS} introduces matrix compositions and two-parameter signatures, with its corresponding quasisymmetric functions presented in 
\Cref{section_QSym_2Param}.
Note that several details and proofs are postponed until \Cref{section:proofs}.

\section{Two-parameter sums signature}\label{section:2dimSS}

\index[general]{monoid@$\monoidComp$, commutative monoid}
\index[general]{eps@$\vareps$, neutral element of $\monoidComp$}
\index[general]{star@$\star$, monoid operation of $\monoidComp$}

\index[general]{k@$\groundRing$, integral domain}

Let $\groundRing$ denote an integral domain, that is a non-zero, commutative ring which includes a multiplicative identity and which contains no zero divisors. 
For every set $M$ let $\groundRing^M$ denote the $\groundRing$-algebra of functions from $M$ to $\groundRing$.

We now replace the monoid $(\N_0,+,0)$ used for compositions from the previous section (and not to be confused with
the role played by $\N$ in the indexing poset of a function $A:\N^2\rightarrow\groundRing$) by an arbitrary commutative monoid $(\monoidComp,\star,\monCompNeutrElem)$.
Apart from the non-negative integers,
the monoid most relevant to us is the free commutative monoid generated by $d$ letters $\w{1},\ldots,\w{d}$. 
The latter is used in this article for the same purpose as in \cite{diehl2020tropical,Diehl_2020},
to index monomials on a $d$-dimensional vector (the data at each point $\tuIn{i} \in \N^2$).

A \DEF{matrix composition}\index[general]{Composition@$\composition$, matrix compositions} (with entries in $\monoidComp$)
is an element in the set of all matrices without $\monCompNeutrElem$-lines or $\monCompNeutrElem$-columns\footnote{Let $\vareps_{s\times t}$ denote the $\monoidComp$-valued $s\times t$ matrix with only $\vareps$ as its entries.}, 
\begin{align*}
\composition:=\composition(\monoidComp)
:=
\left\{\mathbf{a}\in\monoidComp^{m\times n}\;
\begin{array}{|l}
\;\mathbf{a}_{i,\bullet}\not=\monCompNeutrElem_{1\times n}
\text{ for }1\leq i\leq n\\
\;\mathbf{a}_{\bullet,j}\not=\monCompNeutrElem_{m\times 1}\text{ for }1\leq j\leq m
\end{array}
\right\}\cup\{\ec\},
\end{align*}
including the \DEF{empty composition} denoted by $\ec$. \index[general]{EmptyComposition@$\ec$}
Matrix compositions have already appeared as 
    the building block of a certain Hopf algebra
    in \cite{duchamp2002noncommutative},
but the Hopf algebra we construct in \Cref{ss:hopfAlgebra} 
    is different.
The following binary operation turns $\composition$ into a monoid.

 \begin{definition}\label{def:diag}
    For each pair of matrices $(\mathbf{a},\mathbf{b})\in\monoidComp^{m\times n}\times\monoidComp^{s\times t}$, define the block matrix\footnote{Block matrices are allowed to be non-square, similar to the direct sum of two $\C$-valued matrices when considered as homomorphisms of $\C$-vector spaces.}
    \begin{align*}
        \diag(\mathbf a, \mathbf b)
        :=
        \begin{bmatrix}
            \mathbf a & \monCompNeutrElem_{m\times t} \\ 
            \monCompNeutrElem_{n\times s}        & \mathbf b
        \end{bmatrix}
        \in \monoidComp^{(m+s)\times (n+t)}.
    \end{align*}
\end{definition}
With 
$\diag(\ec,\mathbf{a}):=\diag(\mathbf{a},\ec):=\mathbf{a}$ for all $\mathbf{a}\in\composition$ this defines a binary operation  $\diag:\composition\times\composition\rightarrow\composition$, 
    which   
    extends to any sequence of matrices $(\mathbf a_i)_{1\leq i\leq k}$ 
    via%
    \footnote{Note that $\diag$ is associative,
    so we could equivalently take
        $\diag( \mathbf{a}_1, \diag( \mathbf a_2, \dots, \mathbf a_{k}))$. 
        For convenience we also allow single inputs $\diag( \mathbf{a}):=\mathbf{a}$. }
    \begin{align*}
        \diag( \mathbf a_1, \dots \mathbf a_k )
        :=
        \diag( \diag( \mathbf a_1, \dots, \mathbf a_{k-1}), \mathbf a_k ).
    \end{align*}

A non-empty composition $\mathbf{a}\not=\ec$ is called \DEF{connected},
if $\mathbf{a}=\diag(\mathbf{b},\mathbf{c})$ implies $\mathbf{b}=\ec$ or $\mathbf{c}=\ec$, i.e., if it has no non-trivial block matrix decomposition. 
Let $\compositionConnected$ \index[general]{connected@$\compositionConnected$} denote the set of connected compositions.
\begin{example}
Let $(\mathfrak{M}_3,\star,\vareps)$ denote the free commutative monoid generated by three elements $\w{1}$, $\w{2}$ and $\w{3}$. Then, 
$$
\begin{bmatrix} \w{1}&\w{2}\end{bmatrix},
\begin{bmatrix} \w{3}\end{bmatrix},
\begin{bmatrix} \w{2}&\w{1}\star\w{2}\\\w{3}&\w{1}\end{bmatrix}\in\compositionConnected$$ 
and as a non-example,
\begin{equation}\label{eq:example_connectedDec}
\diag(\begin{bmatrix}\w{1}&\w{2}\end{bmatrix},
\begin{bmatrix}\w{3}\end{bmatrix})
=\begin{bmatrix}\w{1}&\w{2}&\vareps\\
\vareps&\vareps&\w{3}\end{bmatrix}
\not\in\compositionConnected.
\end{equation}
\end{example}
For illustration of the following lemma,
whose immediate proof we omit,
see \Cref{eq:example_connectedDec} or \Cref{ex:coproduct}.
\begin{lemma}\label{lemma:primitive_decomposition}
Every non-empty composition $\mathbf{a}\not=\ec$ has a unique \DEF{factorization into connected compositions}, i.e., for uniquely determined $k\in\N$,
$$\mathbf{a}=\diag(\mathbf{v}_1,\ldots,\mathbf{v}_k),$$
where $\mathbf{v}\in\compositionConnected^k$ is unique. 
\end{lemma}

Next, we consider formal power series over connected compositions of the form 
\begin{equation}\label{formalPowerSeriesF}
F:=\sum_{\mathbf{a}\in\composition}c_{\mathbf{a}}\,\mathbf{a}\end{equation}
where $c:\composition\rightarrow\groundRing$ provides coefficients from the integral domain. 
The set of all such series is denoted by 
\index[general]{formalpowerseries@$\groundRing\llangle\compositionConnected\rrangle$, formal power series over connected compositions}
\index[general]{freeAlgebra@$\groundRing\langle\compositionConnected\rangle$, free algebra over connected compositions}
$$\groundRing\llangle\compositionConnected\rrangle\cong\groundRing^\composition.$$

The notation $\groundRing\llangle\compositionConnected\rrangle$,
indicating (the closure of) a free associative algebra over the ``letters'' $\compositionConnected$,
is justified since matrix compositions can be considered as words, with respect to $\diag$ from \Cref{def:diag}, over connected compositions, i.e.
$$\composition=\left\{\diag( \mathbf a_1, \dots \mathbf a_k )\mid k\in\N, \mathbf{a}\in\compositionConnected^k\right\}\cup\{\ec\}.$$

The \DEF{support} of a series (\ref{formalPowerSeriesF})  consists of all $\mathbf{a}\in\composition$ with $c_{\mathbf{a}}\not=0$.
We denote by 
$$\groundRing\langle\compositionConnected\rangle:=\bigoplus_{\mathbf{a}\in\composition}\groundRing\mathbf{a}$$
the subset of series with finite support, i.e., the direct sum of all modules $\groundRing\mathbf{a}$.  
Each series (\ref{formalPowerSeriesF}) can be viewed as $F\in\groundRing^\composition$ with
$$F(\mathbf{a}):=\langle F,\mathbf{a}\rangle:=c_{\mathbf{a}}$$
and extends uniquely to a linear map $F:\groundRing\langle\compositionConnected\rangle\rightarrow\groundRing$, i.e.
we obtain a pairing $$\langle\cdot,\cdot\rangle:\groundRing\llangle\compositionConnected\rrangle\times\groundRing\langle\compositionConnected\rangle\rightarrow\groundRing.$$
We turn $\groundRing\llangle\compositionConnected\rrangle$ into a $\groundRing$-algebra by equipping it with a linear and multiplicative structure
\begin{align*}
    \langle sF,\mathbf{a}\rangle&:=s \langle F,\mathbf{a}\rangle,\\
    \langle F+G,\mathbf{a}\rangle&:=\langle F,\mathbf{a}\rangle+\langle G,\mathbf{a}\rangle\text{, and}\\
    \langle F\cdot G,\mathbf{a}\rangle&:=\sum_{\diag(\mathbf{b},\mathbf{c})=\mathbf{a}}\langle F,\mathbf{b}\rangle\,\langle G,\mathbf{c}\rangle
\end{align*}
where $\mathbf{a}\in\composition$ and $s\in\groundRing$. 
The \DEF{constant series} are precisely the scaled empty compositions $s\ec$, including in particular the empty composition $\ec$ as the neutral element with respect to multiplication.

\index[general]{Md@$\mathfrak{M}_d$, free commutative monoid generated by $d$ elements}
Fixing a dimension $d\in\N$, let $\mathfrak{M}_d$ denote the free commutative monoid generated by $d$ elements $\w{1},\ldots,\w{d}$. 

\begin{center}
\textbf{We use, from now on, $\mathfrak{M}_d$ as
the underlying monoid of $\composition = \composition(\mathfrak{M}_d)$.}
\end{center}
\index[general]{weight@$\weight$}
\begin{definition}
    Let $\weight:\monoidComp_d\rightarrow\N_0$ denote the homomorphism of monoids, uniquely determined by $\weight(\w{j})=1$ for all $\w{j}\in\{\w{1},\ldots,\w{d}\}$.
    For every $\mathbf{a}\in\composition$ let 
    $$\weight(\mathbf{a}):=\sum\limits_{\tuIn{n}\leq\size(\mathbf{a})}\weight(\mathbf{a}_{\tuIn{n}}).$$
\end{definition}
This yields an $\N_0$-grading of the $\groundRing$-module
\begin{equation}\label{eq:grading}
\groundRing\langle\compositionConnected\rangle=\groundRing\ec\oplus\bigoplus_{i\in\N}\left(\bigoplus_{\substack{\mathbf{a}\in\composition\\\weight(\mathbf{a})=i}}
  \!\!\!\groundRing\mathbf{a}\right)
  \end{equation}
with respect to $\weight$. 
Note that (\ref{eq:grading}) is also connected,
i.e. its zero-weight component is one-dimensional. 
This grading is compatible with the diagonal operation from \Cref{def:diag}:
 for all $\mathbf{a},\mathbf{b}\in\composition$, 
 \begin{align}
    \label{lemma:idweightDiag}
    \weight(\diag(\mathbf{a},\mathbf{b})) = \weight(\mathbf{a}) +\weight(\mathbf{b}).
 \end{align}

For every $d$-dimensional data point $z\in\groundRing^d$ let  $z^{(\cdot)}:\mathfrak{M}_d\rightarrow\groundRing$ denote its \DEF{evaluation} homomorphism, where $z^{(\w{j})}:=z_j$ extends uniquely via the universal property of $\monoidComp_d$. 

\begin{definition}\label{def:ss}
\index[general]{SS@$\SS_{\tuIn{\ell};\tuIn{r}}(Z)$, two-parameter sums signature of $Z:\N^2\rightarrow\groundRing^d$}
For $Z: \N^2\to \groundRing^d$ and $\tuIn{\ell},\tuIn{r}\in\N_0^2$ define the \DEF{two-parameter sums signature} $\SS_{\tuIn{\ell};\tuIn{r}}(Z)\in\groundRing\llangle\compositionConnected\rrangle$
via its coefficients of $\mathbf{a}\in\composition$, 
\begin{align*}
\langle\SS_{\tuIn{\ell};\tuIn{r}}(Z),\mathbf{a}\rangle:=\sum_{\substack{\tuIn{\ell}_1<\iota_1<\cdots<\iota_{\row(\mathbf{a})}\leq \tuIn{r}_1\\\tuIn{\ell}_2<\kappa_1<\cdots<\kappa_{\col(\mathbf{a})}\leq \tuIn{r}_2}}\;
\prod_{s=1}^{\row(\mathbf{a})}
\prod_{t=1}^{\col(\mathbf{a})}
Z_{\iota_s,\kappa_t}^{(\mathbf{a}_{s,t})}\in\groundRing. 
\end{align*}
If $Z$ has finite \DEF{support}, that is $Z_{\tuIn{i}}\not=0_d$ for only finitely many $\tuIn{i}\in\N^2$, we also write $\SS(Z):=\lim_{\tuIn{r}\rightarrow\infty}\SS_{0_2;\tuIn{r}}(Z)$. 
\end{definition}
If $d=1$, one has $\monoidComp_1\cong\N_0$ as monoids, and thus obtains the definition of the two-parameter sums signature 
from the introduction. 

\begin{example}\label{ex:numericalExSS}
Let $d=1$ and $\groundRing=\C$ as in the introduction. 
If, for instance, $\mathbf{a}$ is of format $1\times 1$, then the product reduces to a single factor $Z_{\iota_1,\kappa_1}$ raised to the power of $\mathbf{a}_{1,1}$.
For instance 
  \begin{align*}
\langle\SS\!\left(\,{\begin{tabular}{llll}
\hline
\multicolumn{1}{|l}{
\cellcolor{ballblue}$2$
}
   &\cellcolor{moccasin}$1$
   &\cellcolor{gray}$0$\\
\multicolumn{1}{|l}{
\cellcolor{myGreen}$3$}
   &\cellcolor{moccasin}$1$
   &\cellcolor{gray}$0$\\
\multicolumn{1}{|l}{
    \cellcolor{gray}$0$}
   &\cellcolor{gray}$0$
   &\cellcolor{gray}$0$\\
   &
\end{tabular}}_{\;\ddots\,}\right),
\begin{bmatrix}
1\end{bmatrix}\rangle
=
{\begin{tabular}{|l|}
\hline
\cellcolor{ballblue}$2$\\
\hline\end{tabular}}^{\,1}
+
{\begin{tabular}{|l|}
\hline
\cellcolor{moccasin}$1$\\
\hline\end{tabular}}^{\,1}
+
{\begin{tabular}{|l|}
\hline
\cellcolor{myGreen}$3$\\
\hline\end{tabular}}^{\,1}
+
{\begin{tabular}{|l|}
\hline
\cellcolor{moccasin}$1$\\
\hline\end{tabular}}^{\,1}
\end{align*}
is simply the sum of all non-zero entries from the input function.

If $\mathbf{a}$ is of format $2\times 2$, then one considers all matrices 
$$\begin{bmatrix}
Z_{\iota_1,\kappa_1}&Z_{\iota_1,\kappa_2}\\
Z_{\iota_2,\kappa_1}&Z_{\iota_2,\kappa_2}
\end{bmatrix}\in\C^{2\times 2}$$
with increasing $\iota_1<\iota_2$ and $\kappa_1<\kappa_2$, raises it entrywise to the power of $\mathbf{a}$, multiplies its entries, and sums the resulting products over all constellations of independent $\iota$ and $\kappa$. 
For instance, in
\begin{align*}
\langle\SS\!\left(\,{\begin{tabular}{llll}
\hline
\multicolumn{1}{|l}{
\cellcolor{moccasin}$1$
}
   &\cellcolor{ballblue}$2$
   &\cellcolor{gray}$0$\\
\multicolumn{1}{|l}{
\cellcolor{gray}$0$}
   &\cellcolor{babypink}$7$
   &\cellcolor{gray}$0$\\
\multicolumn{1}{|l}{
\cellcolor{myGreen}$3$}
   &\cellcolor{moccasin}$1$
   &\cellcolor{gray}$0$\\
\multicolumn{1}{|l}{
    \cellcolor{gray}$0$}
   &\cellcolor{gray}$0$
   &\cellcolor{gray}$0$\\
   &
\end{tabular}}_{\;\ddots\,}\right),
\begin{bmatrix}1&
2\\
0&
1\end{bmatrix}\rangle
=
{\begin{tabular}{|l|}
\hline
\cellcolor{moccasin}$1$\\
\hline\end{tabular}}^{\,1}
\!\cdot
{\begin{tabular}{|l|}
\hline
\cellcolor{ballblue}$2$\\
\hline\end{tabular}}^{\,2}
\!\cdot
{\begin{tabular}{|l|}
\hline
\cellcolor{babypink}$7$\\
\hline\end{tabular}}^{\,1}
+
{\begin{tabular}{|l|}
\hline
\cellcolor{moccasin}$1$\\
\hline\end{tabular}}^{\,1}
\!\cdot
{\begin{tabular}{|l|}
\hline
\cellcolor{ballblue}$2$\\
\hline\end{tabular}}^{\,2}
\!\cdot
{\begin{tabular}{|l|}
\hline
\cellcolor{moccasin}$1$\\
\hline\end{tabular}}^{\,1}
\end{align*}  
only the two constellations $\iota_1=\kappa_1=1$, $\kappa_2=2$ and either $\iota_2=2$ or $\iota_2=3$ yield non-zero summands.  

\end{example}

\subsection{The Hopf algebra of matrix compositions}
\label{ss:hopfAlgebra}

  We shall now equip the $\groundRing$-module $\groundRing\langle\compositionConnected\rangle$ with another product by  introducing a two-parameter version of the well-known quasi-shuffle from  \cite{hoffman1999quasishuffle}. 
This novel product (\Cref{def:twodim_qshuffle}) is clearly different from the classical quasi-shuffle over $\groundRing\langle\compositionConnected\rangle$. 
We call the latter \DEF{one-parameter quasi-shuffle}, recalled in \Cref{subsection:quasishuffle} 
  to compute our two-parameter generalization using recursion. 
  
  As in the classical setting, let $\qSh(m,s;j)$ with $m,s,j\in\N$ denote the set of surjections
$$q:\{1,\ldots,\,m+s\}\twoheadrightarrow\{1,\ldots,\,j\}$$ 
such that $q(1)<\ldots<q(m)$ and $q(m+1)<\ldots<q(m+s)$.
  \index[general]{qSh@$\qSh(m,s;j)$}
  \index[general]{quasiShuffleProduct@$\qShuffle$}
\begin{definition}\label{def:twodim_qshuffle}
The \DEF{two-parameter quasi-shuffle product} is the bilinear mapping
$\qShuffle:\groundRing\langle\compositionConnected\rangle\times\groundRing\langle\compositionConnected\rangle \rightarrow\groundRing\langle\compositionConnected\rangle$ determined on compositions
$(\mathbf{a},\mathbf{b})\in\monoidComp^{m\times n}\times\monoidComp^{s\times t}$ via $\ec\qShuffle \ec:=\ec$, $\mathbf{a}\qShuffle \ec:=\ec\qShuffle\mathbf{a}:=\mathbf{a}$ and
$$\mathbf{a}\qShuffle\mathbf{b}
:=\sum_{j,k\in\N}\;\; 
\sum_{\substack{p\in\qSh(m,s;j)\\q\in\qSh(n,t;k)}}\;\;
\begin{bmatrix}
\mathbf{c}^{p,q}_{1,1}&\ldots&\mathbf{c}^{p,q}_{1,k}\\
\vdots&&\vdots\\
\mathbf{c}^{p,q}_{j,1}&\ldots&\mathbf{c}^{p,q}_{j,k}\\
\end{bmatrix}$$
where

\begin{equation}\label{eq:c_quasishuffle}
\mathbf{c}^{p,q}_{x,y}:=\underset{\substack{u\in p^{-1}(x)\\v\in q^{-1}(y)}}\bigstar{\diag(\mathbf{a},\mathbf{b})}_{u,v}\in\monoidComp.
\end{equation}
\end{definition}

\begin{example}\label{ex:qSh_small}
The quasi-shuffle of $\begin{bmatrix}
\w{1}
\end{bmatrix}$ and $\begin{bmatrix}
\w{2}
\end{bmatrix}$ from $\monoidComp_2^{1\times1}$ 
involves surjections only in $\qSh(1,1;1)$ and $\qSh(1,1;2)$, i.e.,
\begin{align*}
\begin{bmatrix} \w{1}\end{bmatrix}\qShuffle\begin{bmatrix}\w{2}\end{bmatrix}=
\begin{bmatrix} \w{1}&\vareps\\\vareps&\w{2}\end{bmatrix}
+&\begin{bmatrix}\vareps& \w{1}\\\w{2}&\vareps\end{bmatrix}
+\begin{bmatrix}\w{2}&\vareps\\\vareps& \w{1}\end{bmatrix}
+\begin{bmatrix}\vareps&\w{2}\\ \w{1}&\vareps\end{bmatrix}\\
+&\begin{bmatrix} \w{1}\\\w{2}\end{bmatrix}
+\begin{bmatrix}\w{2}\\ \w{1}\end{bmatrix}
+\begin{bmatrix} \w{1}&\w{2}\end{bmatrix}
+\begin{bmatrix}\w{2}& \w{1}\end{bmatrix}
+\begin{bmatrix} \w{1}\star\w{2}\end{bmatrix}.
\end{align*}
Larger examples are provided in \Cref{subsection:quasishuffle}, after relating two-parameter quasi-shuffles to the one-parameter setting. 
\end{example}
Note that if $j<\max(m,s)$ or $m+s<j$, then $\qSh(m,s;j)$ is empty, i.e., the sum in \Cref{def:twodim_qshuffle} is guaranteed to be finite.

We now endow $\groundRing\langle\compositionConnected\rangle$ with a coproduct, turning it into
a graded, connected bialgebra (\Cref{theoroem:bialgebra}) and hence into a Hopf algebra.
The coproduct will be seen to be compatible
with a certain ``concatenation" of the input data,
leading to a (weak) form of Chen's identity (\Cref{subsection:Chen}).

\begin{definition}\label{def:coproduct}
\index[general]{Delta@$\Delta$}
The \DEF{deconcatenation coproduct} $$\Delta:\groundRing\langle\compositionConnected\rangle\rightarrow\groundRing\langle\compositionConnected\rangle\otimes\groundRing\langle\compositionConnected\rangle$$ is defined on non-empty basis elements $\mathbf{a}\in\composition$
as
\begin{align*}\Delta\mathbf{a}&:
=\mathbf{a}\otimes \ec+\ec\otimes\mathbf{a}+\sum_{\alpha=1}^{a-1}\diag(\mathbf{v}_1,\ldots,\mathbf{v}_{\alpha})\otimes\diag(\mathbf{v}_{\alpha+1},\ldots,\mathbf{v}_a)\\
&=\sum_{\diag(\mathbf{b},\mathbf{c})=\mathbf{a}}\mathbf{b}\otimes\mathbf{c}.
\end{align*}
Here, $\mathbf{v}$ is a factorization  of composition $\mathbf{a}$ into connected components $\mathbf{v}_\alpha$ according to \Cref{lemma:primitive_decomposition}.
Furthermore, let $\Delta\ec:=\ec\otimes\ec$.
\end{definition}

\begin{example}\label{ex:coproduct}
For the factorization of the $4\times 5$ composition
$$\mathbf{a}=
\begin{bmatrix}
\w{1}&\w{2}&\vareps&\vareps&\vareps\\
\vareps&\vareps&\w{3}&\vareps&\vareps\\
\vareps&\vareps&\vareps&\w{2}&\w{1}\star\w{2}\\
\vareps&\vareps&\vareps&\w{1}&\w{1}
\end{bmatrix}
=\diag(
\begin{bmatrix} \w{1}&\w{2}\end{bmatrix},
\begin{bmatrix} \w{3}\end{bmatrix},
\begin{bmatrix} \w{2}&\w{1}\star\w{2}\\\w{1}&\w{1}\end{bmatrix})$$ 
into connected compositions, 
$$
\Delta\mathbf{a}
=\mathbf{a}\otimes\ec 
+
\begin{bmatrix}\w{1}&\w{2}&\vareps\\
\vareps&\vareps&\w{3}\end{bmatrix}\otimes
\begin{bmatrix} \w{2}&\w{1}\star\w{2}\\\w{1}&\w{1}\end{bmatrix}
+
\begin{bmatrix}\w{1}&\w{2}\\
\end{bmatrix}\otimes
\begin{bmatrix} \w{3}&\vareps&\vareps\\
\vareps&\w{2}&\w{1}\star\w{2}\\
\vareps&\w{1}&\w{1}\end{bmatrix}
+ \ec\otimes\mathbf{a}.$$
\end{example}

This coproduct is coassociative with canonical \DEF{counit} map $\eps:\groundRing\langle\compositionConnected\rangle\rightarrow\groundRing$, where $\eps(\ec):=1$ and $\eps(\mathbf{a}):=0$ for all $\mathbf{a}\not=\ec$. 
Furthermore, it is graded with respect to (\ref{eq:grading}) due to \Cref{lemma:idweightDiag}. 
Note that this coproduct is dual to the free associative product $\diag$ from \Cref{def:diag}.

\begin{theorem}
\label{theoroem:hopfalgebra}
If $\groundRing$ is a commutative $\Q$-algebra, 
then 
$$(\groundRing\langle\compositionConnected\rangle,\qShuffle,\eta,\Delta,\eps)$$ is a graded, connected and commutative  Hopf algebra with canonical \DEF{unit} map $\eta:\groundRing\rightarrow\groundRing\langle\compositionConnected\rangle$ determined through $\eta(1):=\ec$.
The \DEF{antipode} map  $$A:\groundRing\langle\compositionConnected\rangle\rightarrow\groundRing\langle\compositionConnected\rangle$$ is defined on basis elements via $A(\mathbf{a}):=$
$$
\sum_{(\iota_1,\ldots,\iota_\ell)\in\mathcal{C}(k)}{(-1)}^\ell
\diag(\mathbf{v}_1,\ldots,\mathbf{v}_{\iota_1})
\qShuffle\cdots\qShuffle
\diag(\mathbf{v}_{(\iota_1+\cdots+\iota_{(\ell-1)}+1)},\ldots,{\mathbf{v}}_{k}),$$
where $\mathbf{a}=\diag(\mathbf{v}_1,\ldots,{\mathbf{v}}_k)$ is according to \Cref{lemma:primitive_decomposition} and $\mathcal{C}$ contains all (classical) compositions of length $k\in\N$. 
In particular, $(\groundRing\langle\compositionConnected\rangle,+,\qShuffle,0,\ec)$ is free. 
\end{theorem}
\begin{proof}
In \Cref{lemma_qsh_ass,CorgradedProd,Cor_sqsProp} (\Cref{subsec:ommittedDetQSh}) we prove that the two-parameter quasi-shuffle is associative, commutative, graded and that its codomain is indeed $\groundRing\langle\compositionConnected\rangle$.
The proof of the bialgebra relation requires further machinery and is provided in
  \Cref{theoroem:bialgebra}.
The remaining follows since every graded, connected bialgebra is a Hopf algebra
where the antipode is explicit, e.g. \cite[Proposition 3.8.8]{hazewinkel2010algebras} or 
\cite[Remark 4.1.2, Example 4.1.3]{cartier2021classical}.
Freeness is a classical consequence, compare for instance \cite[Theorem 4.4.2]{cartier2021classical} or \cite[Theorem 1.7.29]{GR14}
\end{proof}

  \subsection{Quasi-shuffle identity}
  \label{subsection:quasishuffle}
 
Our first results concern the verification that $\SS$ is a proper sums ``signature''.
Specifically,
we mean that it possesses the two main properties mentioned in the introduction: Chen’s identity (\Cref{subsection:Chen}) and the \DEF{quasi-shuffle identity},  with the quasi-shuffle product taken from  \Cref{def:twodim_qshuffle}.

\begin{theorem}\label{lem:quasiShuffleRel}
  For all $Z:\N^2\rightarrow\groundRing^d$, bounds $\tuIn{\ell},\tuIn{r}\in\N^2$ and 
  $\mathbf{a},\mathbf{b}\in\composition$,
  $$
  \langle\SS_{\tuIn{\ell};\tuIn{r}}(Z),\mathbf{a}\rangle\,
  \langle\SS_{\tuIn{\ell};\tuIn{r}}(Z),\mathbf{b}\rangle
  =
  \langle\SS_{\tuIn{\ell};\tuIn{r}}(Z),\mathbf{a}\qShuffle\mathbf{b}\rangle.
  $$
\end{theorem}

The proof is postponed until \Cref{sec:polynomialInvariant}. 
In \Cref{subsec:ommittedDetQSh} we relate the two-parameter quasi-shuffle to classical one-parameter operations on columns and rows, resulting in an algorithm for its efficient computation. 
Here, we give a brief explanation and a first illustration of the method (\Cref{ex:algo2dimQSh}),
in order to make the product more transparent to the reader.
 
\index[general]{quasi-shuffle of columns@$\qShuffle_2$}
Equip the free associative algebra $\groundRing\langle \monoidComp^i\rangle$ generated by  column vectors of size $i$ with a (classical, one-parameter) \DEF{quasi-shuffle of columns} defined recursively via
\begin{align*}
    \ec\qShuffle_2 {w}&:={w}\qShuffle_2\ec:={w},\text{ and}\\ {a}{v}\qShuffle_2{b}{w}&:={a}({v}\qShuffle_2{b}{w})+{b}({a}{v}\qShuffle_2{w})+({a}\star {b})({v}\qShuffle_2 {w})
\end{align*}
for all monomials $v,\,{w}\in\groundRing\langle \monoidComp^i\rangle$ and letters ${a},\,{b}\in\monoidComp^i$. 
Here $\star$ operates entrywise on the vectors.
\index[general]{quasi-shuffle of rows@$\qShuffle_1$}
Analogously let $\qShuffle_1$ be the (classical, one-parameter) \DEF{quasi-shuffle of rows} in $\groundRing\langle \monoidComp^{1\times j}\rangle$.

\begin{example}
A classical, one-parameter quasi-shuffle of columns from $\monoidComp_3^2$ is  
$$\begin{bmatrix} \w{1}\\\w{3}\end{bmatrix}\qShuffle_2\begin{bmatrix}\vareps\\ \w{2}\end{bmatrix}=
\begin{bmatrix} \w{1}&\vareps\\\w{3}& \w{2}\end{bmatrix}
+\begin{bmatrix} \w{1}\\ \w{2}\star\w{3}\end{bmatrix}
+\begin{bmatrix}\vareps& \w{1}\\ \w{2}&\w{3}\end{bmatrix}\in\groundRing\langle\monoidComp_3^2\rangle, $$
similar for columns of size three, 
\begin{equation}\label{eq:ex_shuffle_col}
\begin{bmatrix} \w{1} \\\vareps\\\vareps\end{bmatrix}
\qShuffle_2
\begin{bmatrix}\vareps\\ \w{2}\\ \w{3}\end{bmatrix}=
\begin{bmatrix}
 \w{1}&\vareps\\
\vareps& \w{2}\\
\vareps& \w{3}
\end{bmatrix}
+\begin{bmatrix}
 \w{1}\\
 \w{2}\\
 \w{3}
\end{bmatrix}
+\begin{bmatrix}
\vareps& \w{1}\\
\w{2}&\vareps\\
\w{3}&\vareps
\end{bmatrix}
\in\groundRing\langle\monoidComp_3^3\rangle.
\end{equation}
A quasi-shuffle of monomials in rows from $\monoidComp_3^{1\times 2}$ is 
\begin{equation}\label{eq:ex_shuffle_row}
\begin{bmatrix} \w{1}&\vareps\end{bmatrix}\qShuffle_1
\begin{bmatrix}\vareps& \w{2}\\
\vareps& \w{3}
\end{bmatrix}=
\begin{bmatrix} \w{1}&\vareps\\
\vareps& \w{2}\\
\vareps& \w{3}
\end{bmatrix}
+\begin{bmatrix} \w{1}& \w{2}\\
\vareps&\w{3}\end{bmatrix}
+\begin{bmatrix}
\vareps& \w{2}\\
\w{1}&\vareps\\
\vareps& \w{3}
\end{bmatrix}
+\begin{bmatrix}
\vareps& \w{2}\\ 
\w{1}&\w{3}
\end{bmatrix}
+\begin{bmatrix}
\vareps& \w{2}\\
\vareps& \w{3}\\
\w{1}&\vareps
\end{bmatrix}
\end{equation}
and thus an element in the free algebra $\groundRing\langle\monoidComp_3^{1\times 2}\rangle$.
\end{example}

Note that $\groundRing\langle\monoidComp^i\rangle$ can be considered as a set of linear combinations in $\monoidComp^{i\times j}$, where $j\in\N$ is varying. 
\Cref{sec:algo2dim} provides a full description of the algorithm for computing the two-parameter quasi-shuffle. 
We close this section by applying it step by step. 
  
\begin{example}\label{ex:algo2dimQSh}
First, we compute the column quasi-shuffle of the input $\mathbf{a}\in\monoidComp^{m\times n}_d$ and $\mathbf{b}\in\monoidComp^{s\times t}_d$, when lifted to   monomials in columns of length $m+s$. 
This step is illustrated for instance in \Cref{eq:ex_shuffle_col} with $m=n=t=1$, $s=2$ and $d=3$. 
From the resulting polynomial
we take each of its monomials, and decompose it in two $\monoidComp_3$-valued matrices $\mathbf{c}$ and $\mathbf{d}$ of shape $m\times j$ and $s\times j$, respectively.
In the running example this would be 
\begin{equation}\label{eq:ex_2dimQshAlgo_1st_monomials}
\begin{bmatrix}
 \w{1}&\vareps\\
\vareps& \w{2}\\
\vareps& \w{3}
\end{bmatrix}
+\begin{bmatrix}
\vareps& \w{1}\\
\w{2}&\vareps\\
\w{3}&\vareps
\end{bmatrix}
+\begin{bmatrix}
 \w{1}\\
 \w{2}\\
 \w{3}
\end{bmatrix}\in\groundRing\langle\monoidComp_3^3\rangle,
\end{equation}
from which we would take the first degree $j=2$ monomial and decompose it into  
\begin{equation*}
(\begin{bmatrix}
 \w{1}&\vareps
\end{bmatrix},
\begin{bmatrix}
\vareps& \w{2}\\
\vareps& \w{3}
\end{bmatrix})\in\monoidComp_3^{1\times 2}
\times\monoidComp_3^{2\times 2}.
\end{equation*}
We now interpret $\mathbf{c}$ and $\mathbf{d}$ as monomials of rows, and compute its row quasi-shuffle as illustrated in \Cref{eq:ex_shuffle_row}. 
The resulting monomials are then the first five terms of the quasi-shuffle
\begin{align*}
\begin{bmatrix} \w{1} \end{bmatrix}\qShuffle\begin{bmatrix}\w{2}\\ \w{3}\end{bmatrix}=&
\begin{bmatrix} \w{1}&\vareps\\
\vareps& \w{2}\\
\vareps& \w{3}
\end{bmatrix}
+
\begin{bmatrix} \w{1}& \w{2}\\
\vareps&\w{3}\end{bmatrix}
+\begin{bmatrix}
\vareps& \w{2}\\
\w{1}&\vareps\\
\vareps& \w{3}
\end{bmatrix}
+\begin{bmatrix}
\vareps& \w{2}\\ 
\w{1}&\w{3}
\end{bmatrix}
+\begin{bmatrix}
\vareps& \w{2}\\
\vareps& \w{3}\\
\w{1}&\vareps
\end{bmatrix}\\
&+
\begin{bmatrix} \w{1}\\
\w{2}\\
\w{3}
\end{bmatrix}
+\begin{bmatrix} \w{1}\star \w{2}\\
\w{3}\end{bmatrix}
+\begin{bmatrix}
\w{2}\\
\w{1}\\
\w{3}
\end{bmatrix}
+\begin{bmatrix}
\w{2}\\ 
\w{1}\star\w{3}
\end{bmatrix}
+\begin{bmatrix}
\w{2}\\
\w{3}\\
\w{1}
\end{bmatrix}\\
&+
\begin{bmatrix}
\vareps&\w{1}\\
\w{2}&\vareps\\
\w{3}&\vareps
\end{bmatrix}
+\begin{bmatrix} \w{2}& \w{1}\\
\w{3}&\vareps\end{bmatrix}
+\begin{bmatrix}
\w{2}&\vareps\\
\vareps&\w{1}\\
\w{3}&\vareps
\end{bmatrix}
+\begin{bmatrix}
\w{2}&\vareps\\ 
\w{3}&\w{1}
\end{bmatrix}
+\begin{bmatrix}
\w{2}&\vareps\\
\w{3}&\vareps\\
\vareps&\w{1}
\end{bmatrix}\in\groundRing\langle\compositionConnected
\rangle. 
\end{align*}
The remaining ten summands result with the remaining two monomials from \Cref{eq:ex_2dimQshAlgo_1st_monomials}, when plugged into the quasi-shuffle of rows as described in detail above for the first monomial. 
\end{example}

\subsection{Invariance to zero insertion}\label{subsection:zeroinsertion}

\index[general]{eventuallyZero@$\evZ$}

Before discussing warping invariance in more detail, we introduce another closely related type of invariance.
Instead of a warping operation, we consider the insertion of zero rows and columns. 
We show in \Cref{theorem_bahntrennend} that the two-parameter signature is invariant under the latter, and furthermore, that it characterizes whether two elements are equal up to insertion of zeros.  
This result is then used in \Cref{sec:invariantsWarping} concerning warping invariants. 
 Let 
 \begin{align*}\evZ:=\evZ(\N^2,\groundRing^d)&:=
\left\{X:\N^2\rightarrow\groundRing^d\mid \exists \tuIn{n}\in\N^2\,:X_{\tuIn{i}}\not=0_d\implies \tuIn{i}\leq \tuIn{n}\right\}
\end{align*}
denote the set of all functions from the index set $\N^2$ to the $\groundRing$-module $\groundRing^d$ which are \DEF{eventually zero}, i.e., of finite support. 
 
\index[general]{Zero@$\Zero$}
Let $\Zero_{a,k}:\evZ\rightarrow\evZ$ be the \DEF{zero insertion operation} which puts a zero row or column (specified by axis $a\in\{1,2\}$) at position $k\in\N$ via
$$
{(\Zero_{a,k}X)}_{\tuIn{i}}:=\begin{cases}X_{\tuIn{i}}&\tuIn{i}_a<k\\
0_d& \tuIn{i}_a = k\\
X_{\tuIn{i}-\e{a}}& \tuIn{i}_a> k.\end{cases}
$$

\begin{example}\label{ex:zeroInsertion}
As in the introduction with $d=1$ and $\groundRing=\C$, 
\begin{align*}
\Zero_{2,2}\circ\Zero_{1,2}\!\left(\;
{\begin{tabular}{llll}
\hline
\multicolumn{1}{|l}{
\cellcolor{darkorange}$5$
}
   &\cellcolor{moccasin}$1$
   &\cellcolor{gray}$0$\\
\multicolumn{1}{|l}{
\cellcolor{myGreen}$3$}
   &\cellcolor{moccasin}$1$
   &\cellcolor{gray}$0$\\
\multicolumn{1}{|l}{
    \cellcolor{gray}$0$}
   &\cellcolor{gray}$0$
   &\cellcolor{gray}$0$\\
   &
\end{tabular}}_{\;\ddots\,}
\;\right)=
\;{\begin{tabular}{llll}
\hline
\multicolumn{1}{|l}{
\cellcolor{darkorange}$5$
}
   &\cellcolor{gray}$0$
   &\cellcolor{moccasin}$1$
   &\cellcolor{gray}$0$\\
\multicolumn{1}{|l}{
\cellcolor{gray}$0$}
   &\cellcolor{gray}$0$
   &\cellcolor{gray}$0$
   &\cellcolor{gray}$0$\\
\multicolumn{1}{|l}{
\cellcolor{myGreen}$3$}
   &\cellcolor{gray}$0$
   &\cellcolor{moccasin}$1$
   &\cellcolor{gray}$0$\\
\multicolumn{1}{|l}{
    \cellcolor{gray}$0$}
   &\cellcolor{gray}$0$
   &\cellcolor{gray}$0$
   &\cellcolor{gray}$0$\\
   &
\end{tabular}}_{\;\ddots\,}
\in\evZ.\end{align*}
\end{example}

\index[general]{NFzero@$\NFzero$}
\begin{lemma}\label{lem:commuting_zero}~
\begin{enumerate}
  \item\label{lem:commuting_zero1} $\Zero_{a,k}$ is an injective endomorphism of modules for all $(a,k)\in \{1,2\}\times \N$. 
\item\label{lem:commuting_zero2} 
$\Zero_{1,k}\circ\Zero_{2,j}=\Zero_{2,j}\circ\Zero_{1,k}$ for all $k,\,j$. 
\item\label{lem:commuting_zero3} 
$\Zero_{a,k}\circ\Zero_{a,j}=\Zero_{a,j}\circ\Zero_{a,k-1}$ for $k> j$ and all $a$.
\item\label{lem:commuting_zero_part4} For every $A\in\evZ$ there exists a unique \DEF{zero insertion normal form} $ \NFzero(A)\in\evZ$ such that
$$A=\Zero_{a_q,k_q} \circ \cdots \circ \Zero_{a_1,k_1}{}\circ\NFzero(A)$$
with suitable $a_i\in\{1,\,2\}$, $k_i\in\N$,  
and where $\NFzero({A})$ has no zero column before a non-zero column, and no zero row before a non-zero row.
\end{enumerate}
\end{lemma}

 Most proofs in this and in the remaining subsections are postponed until \Cref{subsection:invariants}.
For an illustration of normal forms, compare the argument in \Cref{ex:zeroInsertion}.

\begin{definition}\label{def_insertionOfZeros}
  A mapping $\phi:\evZ\rightarrow\groundRing$ is \DEF{invariant to inserting zero}
  (in both directions independently%
  \footnote{Invariance to \emph{simultaneous} insertion of zeros would
    demand $\phi\circ\Zero_{2,k}\circ\Zero_{1,k}=\phi$ for all $k \in\N$.
    }
  ), if 
\begin{equation}\label{zeroInvariant}
\phi\circ\Zero_{a,k}=\phi\quad\forall (a,k)\in\{1,2\}\times \N.
\end{equation}
\end{definition}

With \Cref{lem:commuting_zero} one obtains a partition of $\evZ$ into classes of functions which are \DEF{equal up to insertion of zeros}.

\index[general]{equal up to insertion of zero@$\EQzero$}
\begin{lemma}\label{lemma:eqRelInsertionZero}
The relation  $\EQzero $ defined via
\begin{align*}
    A \EQzero B
    \;:\Leftrightarrow\;
   \NFzero(A)=\NFzero(B)
\end{align*}
is an equivalence relation on $\evZ$.
\end{lemma}

It turns out that the two-parameter sums signature characterizes the resulting equivalence classes.
\begin{theorem}\label{theorem_bahntrennend}
For $A,B\in\evZ$, 
$$A\EQzero B\iff \SS(A)=\SS(B).$$
\end{theorem}
Its proof (provided in \Cref{subsection:invariants}) is similar to \cite[Theorem 3.15]{diehl2020tropical} after generalization for two parameters.

\subsection{Invariance to warping}\label{sec:invariantsWarping}
 \newcommand\funcEventConst{\operatorname{evC}(\N_0^2,\groundRing^{d})}
\index[general]{eventuallyConst@$\evC$}
\index[general]{Warp@$\Stutter$}

In this subsection we show that zero insertion invariants and warping invariants translate back and forth,
specified in \Cref{Theo:ZeroStutterBackForth}.  
Modulo constants, this results in \Cref{theorem:ssIffStuttInv}, a full characterization of when two time series are equal up to warping.
 
 From the introduction, recall the set of functions 
\begin{align*}\evC:=\evC(\N^2,\groundRing^d)&:=
\left\{X:\N^2\rightarrow\groundRing^d\mid \exists \tuIn{n}\in\N^2\,:X_{\tuIn{i}}\not=X_{\tuIn{j}}\implies \tuIn{i},\tuIn{j}\leq \tuIn{n}\right\}
\end{align*}
which are \DEF{eventually constant}. Note that $\evZ\subseteq\evC$. 

Let $\Stutter_{a,k}:\evC\rightarrow\evC$ be a single \DEF{warping operation} which inserts a copy of  the $k$-th row or column (specified by axis $a\in\{1,2\}$) via

  $$
    {(\Stutter_{a,k}X)}_{\tuIn{i}}
    :=
    \begin{cases}
      X_{\tuIn{i}}       & {\tuIn{i}}_a \le k\\
      X_{\tuIn{i}-e_{a}} & {\tuIn{i}}_a > k.
    \end{cases}
  $$

For illustrations, recall the introduction in  \Cref{sec:inroInvariants} with $d=1$. 
We collect analogous properties as in \Cref{lem:commuting_zero}, in particular a normal form with respect to warping. 

\index[general]{NFwarp@$\NFstut$}
\begin{lemma}\label{lem:commuting_warp}~
\begin{enumerate}
\item\label{lem:commuting_warp1} $\Stutter_{a,k}$ is an injective  endomorphism of modules for all $a\in\{1,2\}, k\in\N$.
\item\label{lem:commuting_warp2} 
$\Stutter_{1,k}\circ\Stutter_{2,j}=\Stutter_{2,j}\circ\Stutter_{1,k}$ for all $k,\,j\in\N$. 
\item\label{lem:commuting_warp3} 
$\Stutter_{a,k}\circ\Stutter_{a,j}=\Stutter_{a,j}\circ\Stutter_{a,k-1}$ for $k> j$ and $a \in \{1,2\}$.

\item\label{lem:commuting_warp4}  For every $A\in\evC$ there exists a unique \DEF{warping normal form} $\NFstut({A})\in\evC$ such that
$$A=\Stutter_{a_q,k_q} \circ \cdots \circ {\Stutter_{a_1,k_1}}\circ{\NFstut{}}(A)$$
with suitable $a_i\in\{1,\,2\}$, $k_i\in\N$ 
and where $\NFstut(A)$ has no warped columns and rows before the constant part.
\end{enumerate}
\end{lemma}

\index[general]{difference operator@$\delta$}
\index[general]{sigmavar@$\varsigma$, right inverse of $\delta$}
\index[general]{normal form modulo constants@$\NFconst$}
In \Cref{lemma:ZeroVsStutter} we show how warping and zero insertion relate to each other. 
For this, recall the \DEF{difference operator} 
$\delta:\evC\rightarrow\evZ$ with
\begin{equation}\label{eq:def_diff}
{(\delta X)}_{i,j}:=X_{i+1,j+1}-X_{i+1,j}-X_{i,j+1}+X_{i,j}
\end{equation}
explained also in the introduction. 
Factorizing its kernel, we obtain an isomorphism of $\groundRing$-modules as follows. 

\begin{lemma}\label{oneToOneCorrDiffSig}~
\begin{enumerate}
    \item\label{oneToOneCorrDiffSig1} $\delta$ is a surjective homomorphism of $\groundRing$-modules with 
    \item\label{oneToOneCorrDiffSig2} $\ker(\delta)=\{X\mid X_{\tuIn{i}}=X_{\tuIn{j}}\;\forall \tuIn{i},\tuIn{j}\in\N^2\}$,
    \item and linear right inverse
     $$ \varsigma:\evZ\rightarrow\evC,\;Z\mapsto
   {\left(\sum_{i\leq s}
   \sum_{j\leq t} Z_{s,t}\right)}_{i,j}$$
    \item\label{oneToOneCorrDiffSig3}
    which yields a \DEF{normal form modulo constants} $\NFconst(X):=\varsigma\circ\delta(X)$ as a representative of 
    $$X+\ker(\delta)\in\faktor{\evC}{\ker(\delta)}\cong\evZ.$$ 
\end{enumerate}
\end{lemma}

\begin{example}\label{ex:normalForm}For $d=1$ and $\groundRing=\C$, 
\begin{align*}
\varsigma\circ\delta\!\left(
\;{\begin{tabular}{lllll}
\hline
\multicolumn{1}{|l}{
\cellcolor{babypink}$7$
}
   &\cellcolor{myGreen}$3$
   &\cellcolor{ballblue}$2$
   &\cellcolor{ballblue}$2$\\
\multicolumn{1}{|l}{
\cellcolor{darkorange}$5$}
   &\cellcolor{myGreen}$3$
   &\cellcolor{ballblue}$2$
   &\cellcolor{ballblue}$2$\\
\multicolumn{1}{|l}{
   \cellcolor{ballblue}$2$}
   &\cellcolor{ballblue}$2$
   &\cellcolor{ballblue}$2$
   &\cellcolor{ballblue}$2$\\
   &
\end{tabular}}_{\;\ddots\,}\right)
=
\varsigma\!\left(
\;{\begin{tabular}{llll}
\hline
\multicolumn{1}{|l}{
\cellcolor{ballblue}$2$
}
   &\cellcolor{gray}$0$
   &\cellcolor{gray}$0$\\
\multicolumn{1}{|l}{
\cellcolor{ballblue}$2$}
   &\cellcolor{moccasin}$1$
   &\cellcolor{gray}$0$\\
\multicolumn{1}{|l}{
    \cellcolor{gray}$0$}
   &\cellcolor{gray}$0$
   &\cellcolor{gray}$0$\\
   &
\end{tabular}}_{\;\ddots\,}\right)
=
\;{\begin{tabular}{llll}
\hline
\multicolumn{1}{|l}{
\cellcolor{darkorange}$5$
}
   &\cellcolor{moccasin}$1$
   &\cellcolor{gray}$0$\\
\multicolumn{1}{|l}{
\cellcolor{myGreen}$3$}
   &\cellcolor{moccasin}$1$
   &\cellcolor{gray}$0$\\
\multicolumn{1}{|l}{
    \cellcolor{gray}$0$}
   &\cellcolor{gray}$0$
   &\cellcolor{gray}$0$\\
   &
\end{tabular}}_{\;\ddots\,} 
\end{align*}
is in normal form according to \Cref{oneToOneCorrDiffSig}. 
\end{example}

The three normal forms which were introduced in \Cref{lem:commuting_zero,lem:commuting_warp,oneToOneCorrDiffSig} are related as follows.

\begin{lemma}\label{lemma:ZeroVsStutter}
For all $(a,k)\in\{1,2\}\times \N$, 
\begin{enumerate}
    \item\label{lemma:ZeroVsStutter1}
   $\delta\circ\Stutter_{a,k}={\Zero_{a,k}}\circ{\delta}$,
    \item\label{lemma:ZeroVsStutter2}
    $\varsigma\circ \Zero_{a,k}={\Stutter_{a,k}}\circ{\varsigma}$, 
   \item\label{lemma:ZeroVsStutter3} $\delta\circ\NFstut{}=\NFzero{}\circ\delta$,
   \item\label{lemma:ZeroVsStutter4} 
    $\varsigma\circ\NFzero{}=\NFstut{}\circ\varsigma$, 
    \item\label{lemma:ZeroVsStutter5} $\NFconst{}\circ\Stutter_{a,k}{}=\Stutter_{a,k}{}\circ\NFconst{},$ and
    \item\label{lemma:ZeroVsStutter6} 
    $\NFstut{}\circ\NFconst{}=\NFconst{}\circ\NFstut{}$. 
   \end{enumerate}
    \end{lemma}

\begin{definition}\label{def_stutter}
  A mapping $\psi:\evC\rightarrow\groundRing$
  is
  \begin{enumerate}
   \item \DEF{invariant modulo constants}, if 
 \begin{equation}\label{eq:defInvariantModC}
     \psi(X)=\psi(X+C)\quad\forall X,C\in\evC\text{ with constant }C, 
 \end{equation}
      \item and \DEF{invariant to warping}  (in both directions independently
          \footnote{Invariance to \emph{simultaneous} warping would
                demand $\psi\circ\Stutter_{2,k}\circ\Stutter_{1,k}=\psi$ for all $ k\in\N$.
                Thinking of image data, this operation corresponds to scaling
                uniformly in both directions, whereas the one
                we are interested in allows independent scaling for both directions.}),
      if
  \begin{equation}\label{StutterInvariant}
   \psi\circ\Stutter_{a,k}=\psi\quad\forall(a,k)\in\{1,\,2\}\times\N.\end{equation} 
 
  \end{enumerate}

\end{definition}

With transitivity, both statements can be reformulated  using normal forms, i.e., 
 \Cref{eq:defInvariantModC} is equivalent to 
    $\psi\circ\NFconst{}=\psi$, and
  \Cref{StutterInvariant}   holds if and only if $\psi\circ\NFstut{}=\psi$. 

Warping and zero insertion invariants translate back and forth via the difference operator $\delta$ and its right inverse $\varsigma$. 
\begin{corollary}\label{Theo:ZeroStutterBackForth}~
\begin{enumerate}
    \item\label{Theo:ZeroStutterBackForth1}
If $\phi:\evZ\rightarrow \groundRing$ is invariant to inserting zero, then 
$\phi\circ\delta$ is invariant to warping.
\item\label{Theo:ZeroStutterBackForth2}
If $\psi:\evZ\rightarrow \groundRing$ is invariant to warping, then 
$\psi\circ\varsigma$ is invariant to inserting zero.
\end{enumerate}
\end{corollary}
\begin{proof}
Using \Cref{lemma:ZeroVsStutter}, part \ref{Theo:ZeroStutterBackForth1}. follows with $$\phi\circ\delta=\phi\circ\Zero_{a,k}{}\circ\delta=\phi\circ\delta\circ\Stutter_{a,k}{}$$ 
and analogously, part \ref{Theo:ZeroStutterBackForth2}. with
 $
      \psi\circ\varsigma
      =\psi\circ\varsigma\circ{\Zero_{a,k}} 
 $
for all $(a,k)\in\{1,2\}\times \N$. 
\end{proof}

\index[general]{psi@$\Psi$, class of   invariants modulo constants and warping}

We now define the family of invariants 
$\Psi$ motivated in the introduction. 
\begin{corollary}\label{theorem:invariants}
For every $\mathbf{a}\in\composition$, 
$$\Psi_{\mathbf{a}}:\evC\rightarrow\groundRing,\;X\mapsto\langle\SS(\delta X),\mathbf{a}\rangle$$
is invariant modulo constants and warping in both directions independently. 
\end{corollary}
\begin{proof}
 With $\delta\circ\NFzero{}=\delta\circ\varsigma\circ\delta=\delta$ holds \Cref{eq:defInvariantModC}, whereas 
\Cref{StutterInvariant} is an immediate consequence of  \Cref{theorem_bahntrennend} and the first part of \Cref{Theo:ZeroStutterBackForth}. 
\end{proof}

\begin{example}\label{ex:polynomial_invariant}
    With the summation rule (compare \Cref{lem:OnNormalFormsDeltaSigmaId}),  $$\sum_{i=1}^s\sum_{j=1}^t{(\delta X)}_{i,j}=X_{s+1,t+1}-X_{1,t+1}-X_{s+1,1}+X_{1,1}$$
     for all $X\in\evC$ and $(s,t)\in\N^2$,  
    $$\Psi_{\begin{scriptsize}
\begin{bmatrix}1\end{bmatrix}
\end{scriptsize}}(X)=\langle\SS\circ\delta(X),\begin{bmatrix}
    1
    \end{bmatrix}\rangle=X_{1,1}-{\lim_{s,t\rightarrow\infty}X_{s,t}}$$ 
    as it was claimed in the introduction. 
\end{example}

Invariants according to \Cref{def_stutter} form a $\groundRing$-subalgebra of $\groundRing^\evC$. 
The family $\Psi$ is chosen sufficiently large in the sense,
its linear span is closed under products.
\begin{corollary}\label{corr:suffLarge}
  The $\groundRing$-module $\Span_\groundRing(\Psi)$ is a $\groundRing$-subalgebra of $\groundRing^{\evC}$. 
\end{corollary}
\begin{proof}
 Follows immediately with the quasi-shuffle identity (\Cref{lem:quasiShuffleRel}). 
\end{proof}

\index[general]{equivalent modulo constants or warping@$\sim$, equivalence modulo constants or warping}
\begin{definition}\label{def:StuttInv}
Consider the union of the equivalence relation modulo $\ker(\delta)$ from \Cref{oneToOneCorrDiffSig} and $\EQzero$ from \Cref{lemma:eqRelInsertionZero}. 
Denote its transitive closure by $\sim$, which is an equivalence relation \textbf{modulo constants or warping}. 
\end{definition}

\begin{example}
For $d=1$ and $\groundRing=\C$, 
\begin{align*}
{\begin{tabular}{lllll}
\hline
\multicolumn{1}{|l}{
\cellcolor{babypink}$7$
}
   &\cellcolor{myGreen}$3$
   &\cellcolor{ballblue}$2$
   &\cellcolor{ballblue}$2$\\
\multicolumn{1}{|l}{
\cellcolor{darkorange}$5$}
   &\cellcolor{myGreen}$3$
   &\cellcolor{ballblue}$2$
   &\cellcolor{ballblue}$2$\\
\multicolumn{1}{|l}{
   \cellcolor{ballblue}$2$}
   &\cellcolor{ballblue}$2$
   &\cellcolor{ballblue}$2$
   &\cellcolor{ballblue}$2$\\
   &
\end{tabular}}_{\;\ddots\,}
\;\sim
\;{\begin{tabular}{llll}
\hline
\multicolumn{1}{|l}{
\cellcolor{darkorange}$5$
}
   &\cellcolor{moccasin}$1$
   &\cellcolor{gray}$0$\\
\multicolumn{1}{|l}{
\cellcolor{myGreen}$3$}
   &\cellcolor{moccasin}$1$
   &\cellcolor{gray}$0$\\
\multicolumn{1}{|l}{
    \cellcolor{gray}$0$}
   &\cellcolor{gray}$0$
   &\cellcolor{gray}$0$\\
   &
\end{tabular}}_{\;\ddots\,}
\;\sim
\;{\begin{tabular}{llll}
\hline
\multicolumn{1}{|l}{
\cellcolor{darkorange}$5$
}
   &\cellcolor{darkorange}$5$
   &\cellcolor{moccasin}$1$
   &\cellcolor{gray}$0$\\
\multicolumn{1}{|l}{
\cellcolor{myGreen}$3$}
   &\cellcolor{myGreen}$3$
   &\cellcolor{moccasin}$1$
   &\cellcolor{gray}$0$\\
\multicolumn{1}{|l}{
\cellcolor{myGreen}$3$}
   &\cellcolor{myGreen}$3$
   &\cellcolor{moccasin}$1$
   &\cellcolor{gray}$0$\\
\multicolumn{1}{|l}{
    \cellcolor{gray}$0$}
   &\cellcolor{gray}$0$
   &\cellcolor{gray}$0$
   &\cellcolor{gray}$0$\\
   &
\end{tabular}}_{\;\ddots\,}
\end{align*}
where the first equivalence is due to \Cref{ex:normalForm}, and the second is resulting from two warping operations on the second row and column.  
Note that the first function is equivalent to the third due to the transitive closure used in \Cref{def:StuttInv}. 
\end{example}

\index[general]{normal form modulo constants or warping@$\NFsim$}

Since warping and addition of constants commute (\Cref{lemma:ZeroVsStutter}), we obtain a \DEF{normal form modulo constants or warping} 
$\NFsim:=\NFconst{}\circ\NFstut$. 

\begin{theorem}\label{theorem:ssIffStuttInv}For $X,Y\in\evC$, 
\begin{align*}
    X\sim Y&\iff\NFsim(X)=\NFsim(Y) \\
    &\iff\SS\circ\delta(X)=\SS\circ\delta(Y)
\end{align*}
\end{theorem}
\begin{proof}
We sketch only the second equivalence, which is valid with \Cref{theorem_bahntrennend}, 
  $$\delta\circ\NFsim{}=
  \delta\circ\varsigma\circ\delta\circ\NFstut{}=
  \NFzero{}\circ\delta,$$
  and since $\delta$ is injective when restricted on $\NFconst(\evC)$. 
\end{proof}

With \Cref{theorem:ssIffStuttInv} we obtain for every $X,Y\in\evC$, 
$$X\sim Y\iff\Psi_{\mathbf{a}}(X)=\Psi_{\mathbf{a}}(Y)\quad\forall\mathbf{a}\in\composition,$$
i.e., $\Psi$ is expressive enough to decide whether $X$ and $Y$ are equivalent. 
In \Cref{section_QSym_2Param} we show under further assumptions on $\groundRing$, that $\Span_{\groundRing}(\Psi)$ contains precisely the so-called polynomial invariants from $\groundRing^\evC$. 

  \subsection{Chen's identity}\label{subsection:Chen}
  
We now return to Chen's identity, the second key property of signature-like objects.
In this section, we present
\DEF{Chen's identity with respect to diagonal deconcatenation},
\begin{equation}\label{eq:chenConZero}
\forall A,B\in\evZ:\SS(A\varoslash B)= \SS(A)\cdot \SS(B),
 \end{equation}
where the product of signatures is in $\groundRing\llangle\compositionConnected\rrangle$
and the product $\varoslash$ is defined below.
It provides an algebraic relation between input functions and their signatures, allowing to compute the signature of
the concatenation $X:=A\varoslash B$ (\Cref{def:diagonalConcatenation,def:concatDiagBox}) via the signatures of $A$ and $B$ alone.
We relate this result (\Cref{thm:chen}) to warping invariants in \Cref{chen:diff}.

\index[general]{row@$\row$}
\index[general]{col@$\col$}
\index[general]{size@$\size$}

For convenience, we extend the notion of size from matrices to $\evC$ via
$$\size(X):=\left(\row(X),\col(X)\right):=\min\left\{\tuIn{n}\in\N^2\mid X_{\tuIn{i}}\not=X_{\tuIn{j}}\implies \tuIn{i},\tuIn{j}\leq \tuIn{n}\right\}.$$
With $\evZ\subseteq\evC$ this implicitly defines a size for eventually-zero functions, specified further in \Cref{lem:sizeOfEvZero}.
We now define diagonal concatenation in the range of the difference operator $\delta$. 

\index[general]{starConcat@$\varoslash$, diagonal concatenation in  $\evZ$}

\begin{definition}\label{def:diagonalConcatenation}
The binary operation $\varoslash:\evZ\times\evZ\rightarrow\evZ$ 
sends $(A,B)$ to its
\DEF{diagonal concatenation}, 
$$A\varoslash B:=A+\Zero_{1,1}^{\row(A)}\circ\Zero_{2,1}^{\col(A)}(B).$$
\end{definition}

In \Cref{lem:concatenationMat} we verify that $(\evZ,\varoslash)$ is a non-commutative semigroup,  
i.e., that $\varoslash$ is associative. 

\begin{example}
For $d=1$ and $\groundRing=\C$, 
\begin{align*}
{\begin{tabular}{llll}
\hline
\multicolumn{1}{|l}{
\cellcolor{ballblue}$2$
}
   &\cellcolor{babypink}$7$
   &\cellcolor{gray}$0$\\
\multicolumn{1}{|l}{
\cellcolor{ballblue}$2$}
   &\cellcolor{darkorange}$5$
   &\cellcolor{gray}$0$\\
\multicolumn{1}{|l}{
    \cellcolor{gray}$0$}
   &\cellcolor{gray}$0$
   &\cellcolor{gray}$0$\\
   &
\end{tabular}}_{\;\ddots\,}
\;\varoslash
\;\;\;
{\begin{tabular}{llll}
\hline
\multicolumn{1}{|l}{
\cellcolor{ballblue}$2$
}
   &\cellcolor{ballblue}$2$
   &\cellcolor{gray}$0$\\
\multicolumn{1}{|l}{
\cellcolor{moccasin}$1$}
   &\cellcolor{myGreen}$4$
   &\cellcolor{gray}$0$\\
\multicolumn{1}{|l}{
    \cellcolor{gray}$0$}
   &\cellcolor{gray}$0$
   &\cellcolor{gray}$0$\\
   &
\end{tabular}}_{\;\ddots\,} 
\;=
\;\;\;{\begin{tabular}{lllll}
\hline
\multicolumn{1}{|l}{
\cellcolor{ballblue}$2$
}
   &\cellcolor{babypink}$7$
   &\cellcolor{gray}$0$
   &\cellcolor{gray}$0$
   &\cellcolor{gray}$0$\\
\multicolumn{1}{|l}{
\cellcolor{ballblue}$2$}
   &\cellcolor{darkorange}$5$
   &\cellcolor{gray}$0$
   &\cellcolor{gray}$0$
   &\cellcolor{gray}$0$\\
\multicolumn{1}{|l}{
\cellcolor{gray}$0$}
   &\cellcolor{gray}$0$
   &\cellcolor{ballblue}$2$
   &\cellcolor{ballblue}$2$
   &\cellcolor{gray}$0$\\
\multicolumn{1}{|l}{
\cellcolor{gray}$0$}
   &\cellcolor{gray}$0$
   &\cellcolor{moccasin}$1$
   &\cellcolor{myGreen}$4$
   &\cellcolor{gray}$0$\\
\multicolumn{1}{|l}{
    \cellcolor{gray}$0$}
   &\cellcolor{gray}$0$
   &\cellcolor{gray}$0$
   &\cellcolor{gray}$0$
   &\cellcolor{gray}$0$\\
   &
\end{tabular}}_{\;\ddots\,}\in\evZ.
\end{align*}
\end{example}

With this concatenation, we can formulate Chen's identity for eventually-zero functions.

\begin{lemma}
\label{thm:chen}
For all $A,B\in\evZ$ and $\tuIn{\ell},\tuIn{r}\in\N_0^2$ with $\tuIn{\ell}\leq \size(A)\leq \tuIn{r}$,  
\begin{equation*}
 \langle\SS_{\tuIn{\ell};\tuIn{r}}(A\varoslash B),\mathbf{a}\rangle=\sum_{\diag(\mathbf{b},\mathbf{c})=\mathbf{a}}\langle\SS_{\tuIn{\ell};\size(A)}(A \varoslash B),\mathbf{b}\rangle\,\langle\SS_{\size(A);\tuIn{r}}(A\varoslash B),\mathbf{c}\rangle
\end{equation*}
or equivalently, \Cref{eq:chenConZero}.
\end{lemma}

The proof is provided in \Cref{subsection:invariants}. 
When functions in $\evC$ are considered as pictures, one might think of a different notion of concatenation. 
In \Cref{def:concatDiagBox} we ``continue'' the initial lower right picture in all those data points that are not captured by the initial upper left function.  

\index[general]{concatenation along the diagonal@$\boxslash$}
\begin{definition}\label{def:concatDiagBox}

The binary operation $\boxslash: \evC\times\evC\rightarrow\evC$ 
sends $(X,Y)$ to its
\DEF{concatenation along the diagonal}, 
$$X\boxslash Y:=\NFconst(X)+\Stutter_{1,1}^{\row(X)}\circ\Stutter_{2,1}^{\col(X)}(Y).$$
\end{definition}
In \Cref{lem:concatenationMat} we verify that $(\evC,\boxslash)$ is a non-commutative semigroup,  
i.e., that $\boxslash$ is associative.
We furthermore show that $\delta$ and $\varsigma$ are semigroup homomorphisms to and from $(\evZ,\varoslash)$ respectively.

\begin{example}
For $d=1$ and $\groundRing=\C$, 
\begin{align*}
{\begin{tabular}{llll}
\hline
\multicolumn{1}{|l}{
\cellcolor{ballblue}$2$
}
   &\cellcolor{babypink}$7$
   &\cellcolor{ballblue}$2$\\
\multicolumn{1}{|l}{
\cellcolor{ballblue}$2$}
   &\cellcolor{darkorange}$5$
   &\cellcolor{ballblue}$2$\\
\multicolumn{1}{|l}{
    \cellcolor{ballblue}$2$}
   &\cellcolor{ballblue}$2$
   &\cellcolor{ballblue}$2$\\
   &
\end{tabular}}_{\;\ddots\,}
\;\boxslash
\;\;\;
{\begin{tabular}{llll}
\hline
\multicolumn{1}{|l}{
\cellcolor{ballblue}$2$
}
   &\cellcolor{ballblue}$2$
   &\cellcolor{gray}$0$\\
\multicolumn{1}{|l}{
\cellcolor{moccasin}$1$}
   &\cellcolor{myGreen}$4$
   &\cellcolor{gray}$0$\\
\multicolumn{1}{|l}{
    \cellcolor{gray}$0$}
   &\cellcolor{gray}$0$
   &\cellcolor{gray}$0$\\
   &
\end{tabular}}_{\;\ddots\,} 
\;=
\;\;\;{\begin{tabular}{lllll}
\hline
\multicolumn{1}{|l}{
\cellcolor{ballblue}$2$
}
   &\cellcolor{babypink}$7$
   &\cellcolor{ballblue}$2$
   &\cellcolor{ballblue}$2$
   &\cellcolor{gray}$0$\\
\multicolumn{1}{|l}{
\cellcolor{ballblue}$2$}
   &\cellcolor{darkorange}$5$
   &\cellcolor{ballblue}$2$
   &\cellcolor{ballblue}$2$
   &\cellcolor{gray}$0$\\
\multicolumn{1}{|l}{
\cellcolor{ballblue}$2$}
   &\cellcolor{ballblue}$2$
   &\cellcolor{ballblue}$2$
   &\cellcolor{ballblue}$2$
   &\cellcolor{gray}$0$\\
\multicolumn{1}{|l}{
\cellcolor{moccasin}$1$}
   &\cellcolor{moccasin}$1$
   &\cellcolor{moccasin}$1$
   &\cellcolor{myGreen}$4$
   &\cellcolor{gray}$0$\\
\multicolumn{1}{|l}{
    \cellcolor{gray}$0$}
   &\cellcolor{gray}$0$
   &\cellcolor{gray}$0$
   &\cellcolor{gray}$0$
   &\cellcolor{gray}$0$\\
   &
\end{tabular}}_{\;\ddots\,}\in\evC.
\end{align*}
\end{example}

\begin{corollary}\label{chen:diff}
For all $X,Y\in\evC$, 
\begin{equation*}\label{eq:Chen}
\SS(\delta(X\boxslash Y))=
 \SS(\delta(X))\cdot\SS(\delta(Y)). 
\end{equation*}
\end{corollary}
\begin{proof}
Follows immediately with \Cref{lem:concatenationMat} and \Cref{eq:chenConZero}. 
\end{proof}
  
\section{Two-parameter quasisymmetric functions}\label{section_QSym_2Param}

\index[general]{formalPowerFindeg@$\formalPorerSeriesOfFiniteDegree$}

For the entire section, let $\composition$ denote the set of compositions with entries in $\monoidComp=\N_0$. 
Let ${(x_{\tuIn{i}})}_{\tuIn{i}\in\N^2}$ be a family of
symbols and 
$$\formalPorerSeriesOfFiniteDegree:=\{f\in\groundRing\llbracket x\rrbracket\mid \deg(f)<\infty\}$$
be the 
$\groundRing$-algebra of all power series in $x$ which have finite degree. 
\begin{definition}
Call $f\in \formalPorerSeriesOfFiniteDegree$ a \DEF{two-parameter quasisymmetric function}, if 
for all matrix compositions $\mathbf{a}\in\composition$ and increasing chains $\iota_1<\ldots<\iota_{\row(\mathbf{a})}$ and $\kappa_1<\ldots<\kappa_{\col(\mathbf{a})}$ the coefficients of 
\begin{equation}
\prod_{s=1}^{\row(\mathbf{a})}\prod_{t=1}^{\col(\mathbf{a})}x_{s,t}^{\mathbf{a}_{s,t}}\;\text{ and }\;
\prod_{s=1}^{\row(\mathbf{a})}\prod_{t=1}^{\col(\mathbf{a})}x_{\iota_s,\kappa_t}^{\mathbf{a}_{s,t}}\end{equation}
are equal in $f$. 
\index[general]{quasisymmetric functions@$\QSym$, two-parameter quasisymmetric functions}
Let $\QSym$ denote the set of all two-parameter quasisymmetric functions.
\end{definition}
It is clear that $\QSym=\bigoplus_{d\in\N_0}\QSym_d$ is a graded $\groundRing$-module  
with homogeneous components $\QSym_d=\{f\in\QSym\mid\deg(f)=d\}$. 
\Cref{Cor_subalgebra} shows that $\QSym$ is in fact a graded $\groundRing$-subalgebra of $\formalPorerSeriesOfFiniteDegree$.
\begin{lemma}\label{lem:monomialBasis}
$\QSym$ is a free $\groundRing$-module with \DEF{monomial basis} $\mathcal{B}$ consisting of 
all
$$M_{\mathbf{a}}:=
\sum_{\substack{ \iota_1<\ldots<\iota_{\row(\mathbf{a})}\\
 \kappa_1<\ldots<\kappa_{\col(\mathbf{a})}}}\;
\prod_{s=1}^{\row(\mathbf{a})}\prod_{t=1}^{\col(\mathbf{a})}x_{\iota_s,\kappa_t}^{{\mathbf{a}}_{s,t}}$$
for non-empty $\mathbf{a}\in\composition$
and $M_{\ec} := 1 \in \groundRing$.
\end{lemma}
\begin{proof}
For $\mathcal{B}_d:=\{M_{\mathbf{a}}\mid\mathbf{a}\in\composition\text{ with }\sum_{\tuIn{i}\leq\size(\mathbf{a})}\mathbf{a}_{\tuIn{i}}=d\}$ and monomorphism 
 $$\pi:\QSym_d\rightarrow\Span_{\groundRing}\left(\prod_{s=1}^{\row(\mathbf{a})}\prod_{t=1}^{\col(\mathbf{a})} x_{s,t}^{\mathbf{a}_{s,t}}\;\begin{array}{|l}
 \mathbf{a}\in\composition\\
 \end{array}\right)=:V,$$
the image $\pi(\mathcal{B}_d)$ is 
linear independent by construction, and therefore so is $\mathcal{B}_d$. 
It suffices to show that $\QSym_d$ is generated by $\mathcal{B}_d$.
For $f\in\QSym_d$ let $\pi(f)=:\sum_v\lambda(v)\,v\in V$ with $\lambda:\{v\in V\mid v\text{ monomial}\}\rightarrow\groundRing$ of finite support. 
With $f=\sum_v\lambda(v)\,M_{\mathbf{a}(v)}$ follows $\QSym_d=\Span_{\groundRing}{\mathcal{B}}_d$ for unique $\mathbf{a}(v)$ such that $M_{\mathbf{a}(v)}$ contains $v$. 
\end{proof}

\begin{example}
  \begin{align*}
    M_{\begin{scriptsize}\begin{bmatrix}1\end{bmatrix}\end{scriptsize}}&=\sum_{\iota_1,\kappa_1}x_{\iota_1,\kappa_1} \\
    M_{\begin{scriptsize}\begin{bmatrix}3&1\\0&2\end{bmatrix}\end{scriptsize}}&=\sum_{\substack{\iota_1<\iota_2\\\kappa_1<\kappa_2}}x_{\iota_1,\kappa_1}^{3}\,x_{\iota_1,\kappa_2}\,x_{\iota_2,\kappa_2}^{2}\\
    M_{\begin{scriptsize}\begin{bmatrix}0&1\\2&0\end{bmatrix}\end{scriptsize}}&=\sum_{\substack{\iota_1<\iota_2\\\kappa_1<\kappa_2}}x_{\iota_1,\kappa_2}\,x_{\iota_2,\kappa_1}^{2}=\sum_{\substack{\iota_1<\iota_2\\\kappa_1>\kappa_2}}x_{\iota_1,\kappa_1}\,x_{\iota_2,\kappa_2}^{2} \\
    M_{\begin{scriptsize}\begin{bmatrix}2&1\end{bmatrix}\end{scriptsize}}&=\sum_{\substack{\iota_1,\kappa_1<\kappa_2}}x_{\iota_1,\kappa_1}^{2}\,x_{\iota_1,\kappa_2}
  \end{align*}
\end{example}

\subsection{Closure property under multiplication}

We verify the \DEF{quasi-shuffle identity} for two-parameter quasisymmetric functions. 

\begin{theorem}\label{theorem_QSh_id}
For all $\mathbf{a},\mathbf{b}\in\QSym$,  
$$M_{\mathbf{a}}\,M_{\mathbf{b}}=\sum_{j,k\in\N}\,\sum_{\substack{p\in\qSh(\row(\mathbf{a}),\row(\mathbf{b});j)\\q\in\qSh(\col(\mathbf{a}),\col(\mathbf{b});k)}} M_{\mathbf{c}^{p,q}}$$
with $\mathbf{c}^{p,q}\in\composition$ according to \Cref{eq:c_quasishuffle}. 
\end{theorem}

\begin{proof}
The proof is inspired by  \cite[Proposition 5.1.3.]{GR14} and  generalizes it for two parameters.
It uses the matrix notation for surjections introduced in \Cref{rem:onetoonesurjtomat}. 
Consider
\begin{align*}
M_{\mathbf{a}}\,M_{\mathbf{b}}&=
\sum_{\substack{ \iota_1<\ldots<\iota_{\row(\mathbf{a})}\\
 \kappa_1<\ldots<\kappa_{\col(\mathbf{a})}}}
 \sum_{\substack{ \sigma_1<\ldots<\sigma_{\row(\mathbf{b})}\\
 \tau_1<\ldots<\tau_{\col(\mathbf{b})}}}
\left(\prod_{i=1}^{\row(\mathbf{a})}\prod_{k=1}^{\col(\mathbf{a})}x_{\iota_i,\kappa_k}^{{\mathbf{a}}_{i,k}}\right)
\left(\prod_{s=1}^{\row(\mathbf{b})}\prod_{t=1}^{\col(\mathbf{b})}x_{\sigma_s,\tau_t}^{{\mathbf{b}}_{s,t}}\right)\\
&=\sum_{\mathbf{c}\in\composition}
\sum_{\substack{ \mu_1<\ldots<\mu_{\row(\mathbf{c})}\\
 \nu_1<\ldots<\nu_{\col(\mathbf{c})}}}N^{\mathbf{c}}_{\mu,\nu}\,
 \prod_{m=1}^{\row(\mathbf{c})}\prod_{n=1}^{\col(\mathbf{c})}x_{\mu_m,\nu_n}^{{\mathbf{c}}_{m,n}}
\end{align*}
where $N^{\mathbf{c}}_{\mu,\nu}\in\N$ is the number of 
all $4$-tuples 
\begin{equation}\label{eq:fourtuples}(\iota,\kappa,\sigma,\tau)\in \N^{\row(\mathbf{a})}\times\N^{\col(\mathbf{a})}\times\N^{\row(\mathbf{b})}\times\N^{\col(\mathbf{b})}\end{equation} 
such that $\iota,\kappa,\sigma$ and $\tau$ are strictly increasing with 
\begin{equation}\label{eq:proofQshId}
\left(\prod_{i=1}^{\row(\mathbf{a})}\prod_{k=1}^{\col(\mathbf{a})}x_{\iota_i,\kappa_k}^{{\mathbf{a}}_{i,k}}\right)
\left(\prod_{s=1}^{\row(\mathbf{b})}\prod_{t=1}^{\col(\mathbf{b})}x_{\sigma_s,\tau_t}^{{\mathbf{b}}_{s,t}}\right)
=\prod_{m=1}^{\row(\mathbf{c})}\prod_{n=1}^{\col(\mathbf{c})}x_{\mu_m,\nu_n}^{{\mathbf{c}}_{m,n}}.
\end{equation}
The central argument is that $N^{\mathbf{c}}_{\mu,\nu}$ is also the number of matrix pairs 
\begin{equation}\label{eq:matrixPairs}
    (\mathbf{P},\mathbf{Q})\in \QSH(\row(\mathbf{a}),\row(\mathbf{b});\row(\mathbf{c}))\times \qSh(\col(\mathbf{a}),\col(\mathbf{b});\col(\mathbf{c}))
    \end{equation}
    such that 
    \begin{equation}\label{eq:matrixPairsCondition}
      \mathbf{P}\diag(\mathbf{a},\mathbf{b})\mathbf{Q}^\top=\mathbf{c}.
    \end{equation}
To show this, we construct a bijection $\varphi$ from all $4$-tuples  in  (\ref{eq:fourtuples}) with (\ref{eq:proofQshId}) to the set of all matrix pairs (\ref{eq:matrixPairs})
with (\ref{eq:matrixPairsCondition}). 

Concerning the construction of $\varphi$, assume $(\iota,\kappa,\sigma,\tau)$ satisfies  (\ref{eq:proofQshId}). 
Then, for every  $(i,k)\leq\size(\mathbf{a})$ there exists a uniquely determined   $(p(i),q(i))\leq\size(\mathbf{c})$ such that $(\iota_i,\kappa_k)=(\mu_{p(i)},\nu_{q(i)})$. 
Furthermore, for every $(s,t)\leq\size(\mathbf{b})$ there is a unique  $(p' (s),q'(t))\leq\size(\mathbf{c})$ such that  $(\sigma_s,\tau_t)=(\mu_{p'(s)},\nu_{q'(t)})$.
We define
\begin{equation}\label{eq:defP_proof_Qsh}\mathbf{P}:=\begin{bmatrix}
\mathbf{P}_1&\mathbf{P}_2
\end{bmatrix}
:=\begin{bmatrix}
\begin{bmatrix}
e_{p(1)}
&\cdots 
& e_{p(\row(\mathbf{a}))}
\end{bmatrix}
&\begin{bmatrix}
e_{p'(1)}
&\cdots
&e_{p'(\row(\mathbf{b}))}
\end{bmatrix}
\end{bmatrix}
\end{equation}
for rows,
\begin{equation}\label{eq:defQ_proof_Qsh}\mathbf{Q}:=\begin{bmatrix}
\mathbf{Q}_1&
\mathbf{Q}_2
\end{bmatrix}:=
\begin{bmatrix}
\begin{bmatrix}
e_{q(1)}
&\cdots 
&e_{q(\col(\mathbf{a}))}
\end{bmatrix}
&
\begin{bmatrix}
e_{q'(1)}
&\cdots
&e_{q'(\col(\mathbf{b}))}
\end{bmatrix}
\end{bmatrix}
\end{equation}
for columns and set $\varphi(\iota,\kappa,\sigma,\tau):=(\mathbf{P},\mathbf{Q})$. 
Note that both $\mathbf{P}$ and $\mathbf{Q}$ are right invertible since (\ref{eq:proofQshId}) requires for all $m$ and $n$, that either  $x_{\mu_m,\nu_n}=x_{\iota_i,\kappa_k}$ or $x_{\mu_m,\nu_n}=x_{\sigma_s,\tau_t}$ with suitable $i,k,s$ and $t$. 
Hence $(\mathbf{P},\mathbf{Q})$ is of shape (\ref{eq:matrixPairs}). 
Furthermore,  $\mathbf{P}\diag(\mathbf{a},\mathbf{b})\mathbf{Q}^\top$ satisfies  \Cref{eq:proofQshId} via 
\begin{align*}\prod_{m=1}^{\row(\mathbf{c})}\prod_{n=1}^{\col(\mathbf{c})}&x_{\mu_m,\nu_n}^{{(\mathbf{P}\diag(\mathbf{a},\mathbf{b})\mathbf{Q}^\top)}_{m,n}}
=\prod_{m=1}^{\row(\mathbf{c})}\prod_{n=1}^{\col(\mathbf{c})}x_{\mu_m,\nu_n}^{({\mathbf{P}_1\mathbf{a}\mathbf{Q}_1^\top+\mathbf{P}_2\mathbf{b}\mathbf{Q}_2^\top)}_{m,n}}\\
&=\left(\prod_{m=1}^{\row(\mathbf{c})}\prod_{n=1}^{\col(\mathbf{c})}x_{\mu_m,\nu_n}^{({\mathbf{P}_1\mathbf{a}\mathbf{Q}_1^\top)}_{m,n}}\right)
\left(\prod_{m=1}^{\row(\mathbf{c})}\prod_{n=1}^{\col(\mathbf{c})}x_{\mu_m,\nu_n}^{{(\mathbf{P}_2\mathbf{b}\mathbf{Q}_2^\top)}_{m,n}}\right)
\end{align*}
where 
$$x_{\mu_m,\nu_n}^{{(\mathbf{P}_1\mathbf{a}\mathbf{Q}_1^\top)}_{m,n}}=
\begin{cases}
x_{\iota_i,\kappa_k}^{\mathbf{a}_{i,k}}&(p(i),q(k))=(m,n)\\
1&\text{elsewhere},
\end{cases}$$
and analogously, $$x_{\sigma_m,\tau_n}^{{(\mathbf{P}_2\mathbf{b}\mathbf{Q}_2^\top)}_{m,n}}=
\begin{cases}
x_{\sigma_s,\tau_t}^{\mathbf{b}_{i,k}}&(p'(s),q'(t))=(m,n)\\
1&\text{elsewhere}.
\end{cases}$$
Thus $\mathbf{c}=\mathbf{P}\diag(\mathbf{a},\mathbf{b})\mathbf{Q}^\top$ since a monomial is equal if and only if its exponent vector is equal, i.e., $\varphi$ is well-defined. 
Strictly order preserving $p,p',q$ and $q'$ yield that $\varphi$ is injective. 
Conversely, if $\mathbf{P}$ and $\mathbf{Q}$ are right invertible and according to \Cref{eq:defP_proof_Qsh,eq:defQ_proof_Qsh} 
but for arbitrary strictly increasing ${p},{p}',{q}$ and ${q}'$, then 
set $(\iota_{i},\kappa_k):=(\mu_{p(i)},\nu_{q(i)})$ and 
$(\sigma_{s},\tau_t):=(\mu_{p'(i)},\nu_{q'(i)})$. 
By construction, we have $\varphi(\iota,\kappa,\sigma,\tau)=(\mathbf{P},\mathbf{Q})$. 
This shows that $\varphi$ is bijective. 

Note that this implies $N^{\mathbf{c}}_{\mu,\nu}=0$ whenever $\mathbf{c}\in\N_0^{j\times k}$ for which there are no pairs (\ref{eq:matrixPairs})  satisfying
\Cref{eq:matrixPairsCondition}. 
Therefore, 
\begin{align*}
&M_{\mathbf{a}}\,M_{\mathbf{b}}
=\sum_{\mathbf{c}\in\composition}
\sum_{\substack{ \mu_1<\ldots<\mu_{\row(\mathbf{c})}\\
 \nu_1<\ldots<\nu_{\col(\mathbf{c})}}}
 N^{\mathbf{c}}_{\mu,\nu}\,
 \prod_{m=1}^{\row(\mathbf{c})}\prod_{n=1}^{\col(\mathbf{c})}x_{\mu_m,\nu_n}^{{\mathbf{c}}_{m,n}}\\
  &=\sum_{j,k\in\N}\!\!\!\!\!\!\!\!\!
  \sum_{\substack{\mathbf{c}\in\composition\cap\N_0^{j\times k}
  \\\text{where }\mathbf{c}=\mathbf{P}\diag(a,b)\mathbf{Q}^\top\\
  \text{for some}\\
  \mathbf{P}\in\QSH(\row(\mathbf{a}),\row(\mathbf{b});j)\\
 \mathbf{Q}\in\QSH(\col(\mathbf{a}),\col(\mathbf{b});k)}}
 \sum_{\substack{ \mu_1<\ldots<\mu_{\row(\mathbf{c})}\\
 \nu_1<\ldots<\nu_{\col(\mathbf{c})}}}
 \sum_{\substack{(\mathbf{P},\mathbf{Q})\text{ where }\\
 \mathbf{P}\in\QSH(\row(\mathbf{a}),\row(\mathbf{b});j)\\
 \mathbf{Q}\in\QSH(\col(\mathbf{a}),\col(\mathbf{b});k)\\\mathbf{P}\diag(a,b)\mathbf{Q}^\top=\mathbf{c}}}
 \prod_{m=1}^{\row(\mathbf{c})}\prod_{n=1}^{\col(\mathbf{c})}x_{\mu_m,\nu_n}^{{\mathbf{c}}_{m,n}}\\
 &=\sum_{j,k\in\N}\;\sum_{\substack{\mathbf{P}\in\qSh(\row(\mathbf{a}),\row(\mathbf{b});j)\\
 \mathbf{Q}\in\QSH(\col(\mathbf{a}),\col(\mathbf{b});k)}}M_{\mathbf{P}\diag(\mathbf{a},\mathbf{b})\mathbf{Q}^\top}
\end{align*}
which concludes the proof.
\end{proof}

\begin{corollary}\label{Cor_subalgebra}
$\QSym$ is a graded $\groundRing$-subalgebra of $\formalPorerSeriesOfFiniteDegree$. 
\end{corollary}
As in the one-parameter setting, e.g., \cite[p.969]{malvenuto1995duality},
we can endow $\QSym$ with a deconcatenation-type coproduct.
Let 
$I := \N \cup \overline{\N}$ be
the disjoint union of two copies of $\N$,
totally ordered by setting $\N < \overline{\N}$.
Let $(x_{\tuIn{i}})_{\tuIn{i}\in\N^2}$ and $(y_{\tuIn{i}})_{\tuIn{i}\in\overline{\N}^2}$ be two sets of variables and 
define $z$ indexed by $I \times I$ as
\begin{align*}
    z_{\tuIn{i}} :=
    \begin{cases}
      x_{\tuIn{i}} & \tuIn{i}\in \N^2 \\
      y_{\tuIn{i}} & \tuIn{i} \in {\overline{\N}}^2 \\
      0      & \text{elsewhere}.
    \end{cases}
\end{align*}
Any element $f \in \QSym$ can be evaluated at
the variables $z$
and can be written as
\begin{align*}
  f(z) = \sum_{k=1}^{\ell} g_k(x) h_k(y),
\end{align*}
for some uniquely determined two-parameter quasisymmetric functions $g_k$ and $h_k$, where $1\leq k\leq \ell$.
Then
\begin{align*}
  \Delta_\QSym f := \sum_{k=1}^{\ell} g_k \otimes h_k.
\end{align*}

\begin{example}
For the composition 
$\mathbf{a}=\diag(\begin{bmatrix}2&1\end{bmatrix},\begin{bmatrix}3\end{bmatrix})\in\N_0^{2\times 3}$, 
\begin{align*}
M_{\mathbf{a}}(z)=\sum_{\substack{\iota_1<\iota_2\\
 \kappa_1<\kappa_2<\kappa_3}}
 \!z_{\iota_1,\kappa_1}^2\,z_{\iota_1,\kappa_2}\,z_{\iota_2,\kappa_3}^3
&+\!\sum_{\substack{\iota_1<\overline\iota_2\\
 \kappa_1<\kappa_2<\overline\kappa_3}}\!
z_{\iota_1,\kappa_1}^2\,z_{\iota_1,\kappa_2}\,z_{\overline\iota_2,\overline\kappa_3}^3\\
&+\!\sum_{\substack{\overline\iota_1<\overline\iota_2\\
 \overline\kappa_1<\overline\kappa_2<\overline\kappa_3}}\!
z_{\overline\iota_1,\overline\kappa_1}^2\,z_{\overline\iota_1,\overline\kappa_2}\,z_{\overline\iota_2,\overline\kappa_3}^3
\end{align*}
which is equal to 
$M_{\mathbf{a}}(x)+M_{\begin{scriptsize}\begin{bmatrix}2&1\end{bmatrix}\end{scriptsize}}(x)\,M_{\begin{scriptsize}\begin{bmatrix}3\end{bmatrix}\end{scriptsize}}(y)+M_{\mathbf{a}}(y)$.

\end{example}

We omit the proof of the following result.
\begin{theorem}
  $\QSym$, with the induced product of power series,
  the coproduct $\Delta_\QSym$ and obvious unit and counit
  is a graded bialgebra, and hence a Hopf algebra.
  It is isomorphic to the Hopf algebra of \Cref{theoroem:hopfalgebra},
  with isomorphism given by the linear extension of
  \begin{align*}
    \groundRing\langle\compositionConnected\rangle &\to \QSym \\
    \mathbf{a} &\mapsto M_{\mathbf{a}}.
  \end{align*}
\end{theorem}

\subsection{Polynomial  invariants}\label{sec:polynomialInvariant}
\begin{definition}
Let $\zero_{a,k}\in\Endom\left(\formalPorerSeriesOfFiniteDegree\right)$
\index[general]{zero@$\zero$}
be the ``free analog'' of zero insertion defined via
$$\zero_{a,k}(x_{\tuIn{i}}):=\begin{cases}
x_{\tuIn{i}}&{\tuIn{i}}_a<k\\
 0 &{\tuIn{i}}_a=k\\
 x_{(\tuIn{i}-\e{a})}&{\tuIn{i}}_a>k.
\end{cases}$$
\end{definition}
\begin{example} For all $t\in\N$,
$\zero_{1,2}(x_{1,t}+x_{2,t}+x_{3,1}\,x_{4,1})=x_{1,t}+x_{2,1}\,x_{3,1}$.
\end{example}
\renewcommand\eval{\mathsf{eval}}
\index[general]{eval@$\eval$}
We evaluate power series to set-theoretic functions with domain $\evZ=\evZ(\N^2,\groundRing)$, 
i.e., we assume $d=1$ for the entire section. 
Let $\eval$ be the \DEF{evaluation} homomorphism of a formal power series $f\in\formalPorerSeriesOfFiniteDegree$ to its related $$\eval(f):\evZ\rightarrow\groundRing,$$
uniquely determined by $$\eval(x_{\tuIn{i}})(X):=X_{\tuIn{i}}$$
for $\tuIn{i}\in\N^2$. 
For convenience, we also write $f(X):=\eval(f)(X)$ for all $f$ and $X$. 

With this evaluation, we can  prove the quasi-shuffle identity for the two-parameter sums signature. 
\begin{proof}[Proof of \Cref{lem:quasiShuffleRel}]
We only treat the case $d=1$.
For all $Z\in\evZ$ and  
  $\mathbf{c}\in\composition$,
  $$\eval(M_{\mathbf{c}})(Z)=\langle\SS(Z),\mathbf{c}\rangle,$$
  and hence with \Cref{theorem_QSh_id}, 
\begin{align*}
  \langle\SS(Z),\mathbf{a}\rangle\,
  \langle\SS(Z),\mathbf{b}\rangle
  &=
  \eval(M_{\mathbf{a}}\,M_{\mathbf{b}})(Z)\\
  &=
  \sum_{j,k\in\N}\,\sum_{\substack{p\in\qSh(\row(\mathbf{a}),\row(\mathbf{b});j)\\q\in\qSh(\col(\mathbf{a}),\col(\mathbf{b});k)}} \eval\left(M_{\mathbf{c}^{p,q}}\right)(Z)\\
  &=\langle\SS(Z),\mathbf{a}\qShuffle\mathbf{b}\rangle. 
  \end{align*}
\end{proof}

\begin{lemma}\label{lemma:ZeroVsFormalZero}
  If $f\in\formalPorerSeriesOfFiniteDegree$ 
 with $\zero_{a,k}(f)=f$ for all $(a,k)\in\{1,2\}\times\N$, then $\eval(f)$ is invariant to inserting zero (in both directions independently). 
\end{lemma}
\begin{proof}
This follows from 
$$f(X)=\eval\circ\zero_{a,k}(f)(X)=f(\Zero_{a,k}X)$$ 
for all $a,k$ and  $X$, since $\eval\circ \zero_{a,k}$ is a homomorphism and 
$$\eval\circ\zero_{a,k}(x_{\tuIn{i}})(X)
=\eval(x_{\tuIn{i}})(\Zero_{a,k}X)$$
for every $\tuIn{i}\in\N^2$.
\end{proof}
The converse of \Cref{lemma:ZeroVsFormalZero} is not true in general.

\begin{example}
Let $\groundRing$ be the field of two elements. 
The polynomial $f=x^2_{4,1}-x_{4,1}$ induces the constant zero function $\eval(f)=0$, but 
$\zero_{1,2}(f)=x^2_{3,1}-x_{3,1}\not=f$.
\end{example}

\begin{theorem}\label{lemma_polynomialInvariants_form}
  Let $f \in \formalPorerSeriesOfFiniteDegree$.
  Then:

  \begin{center}
    $f\in\QSym$ if and only if  $f=\zero_{a,k}(f)$ for all $a,k$.
  \end{center}
\end{theorem}

\begin{proof}The proof is analogous to \cite[Theorem 3.11.]{diehl2020tropical} when generalized for two parameters.
For every $\mathbf{a}\in\composition$, 
\begin{align*}
\zero_{a,k}(M_{\mathbf{a}})
&=\sum_{\substack{\iota_1<\cdots<\iota_{\row(\mathbf{a})}\\
\kappa_1<\cdots<\kappa_{\col(\mathbf{a})}}}
\prod_{s=1}^{\row(\mathbf{a})}\prod_{t=1}^{\col(\mathbf{a})}
{\zero_{a,k}(x_{\iota_s,\kappa_t})}^{\mathbf{a}_{s,t}}\\
&=\sum_{u=1}^{\row(\mathbf{a})}\;\sum_{\substack{\iota_1<\cdots<\iota_u<k\leq \iota_{u+1}<\cdots<\iota_{\row(\mathbf{a})}\\
\kappa_1<\cdots<\kappa_{\col(\mathbf{a})}}}\;
\prod_{s=1}^{\row(\mathbf{a})}\prod_{t=1}^{\col(\mathbf{a})}
x_{\iota_s,\kappa_t}^{\mathbf{a}_{s,t}}
=M_{\mathbf{a}}
\end{align*}
thus $\QSym\subseteq\{f\mid \zero_{a,k}(f)=f\;\forall a,k\}$.

Conversely, consider every $f\in\QSym$ as a function from the free monoid generated by ${\left(x_{\tuIn{i}}\right)}_{\tuIn{i}\in\N^2}$ to the coefficient ring $\groundRing$. 
Then, 
\begin{align*}
&f\left(\prod_{s=1}^{\row(\mathbf{a})}\prod_{t=1}^{\col(\mathbf{a})}x_{s,t}^{\mathbf{a}_{s,t}}\right)\\&=
\multicirc_{1\leq m\leq \row(\mathbf{a})}\zero_{1,m}^{\iota_m-\iota_{m-1}-1}\circ\multicirc_{1\leq n\leq \col(\mathbf{a})}\zero_{2,n}^{\kappa_n-\kappa_{n-1}-1}(f)\left(
\prod_{s=1}^{\row(\mathbf{a})}\prod_{t=1}^{\col(\mathbf{a})}x_{\iota_s,\kappa_t}^{\mathbf{a}_{s,t}}
\right)\\
&=
f\left(
\prod_{s=1}^{\row(\mathbf{a})}\prod_{t=1}^{\col(\mathbf{a})}
x_{\iota_s,\kappa_t}^{\mathbf{a}_{s,t}}
\right)\end{align*}
for all $\mathbf{a}\in\composition$ and strictly increasing chains $0=\iota_0<\iota_1<\cdots<\iota_{\row(\mathbf{a})}$ and $0=\kappa_0<\kappa_1<\cdots<\kappa_{\col(\mathbf{a})}$.
\end{proof}
\begin{definition}\label{def:polyInv}
We call  $\psi:\evC\rightarrow\groundRing$
from \Cref{def_stutter}
 a \DEF{polynomial invariant}\footnote{Note that this is a mere naming convention, i.e., $\psi$ is not a polynomial function itself, e.g.,  \Cref{ex:polynomial_invariant}. 
 However, \Cref{eq:defInvariantModC,StutterInvariant} forbid polynomial functions per se, so there is no conflict in naming.}, if it is induced by a formal power series  $f\in\formalPorerSeriesOfFiniteDegree$ of finite degree, i.e.,
     \begin{equation}\label{eq:polynomialInvariantDef}
     \psi(X)=f(X)\;\forall X\in\evZ\subseteq\evC.\end{equation}
In this case, \Cref{eq:polynomialInvariantDef}
guarantees that $\psi$ is determined by a formal power series on the entire domain $\evC$. 
\end{definition}
\begin{theorem}\label{the:identificationPolynomailInvariants}
Let $\groundRing$ be an infinite field. 
A mapping $\psi:\evC\rightarrow\groundRing$ is a
\begin{center}
polynomial invariant according to \Cref{def:polyInv}, \\
~\\
\emph{if and only if}\\
~\\
it is induced by the 
two-parameter sums signature precomposed with the difference operator, i.e., there is $\mathbf{w}\in\groundRing\langle\compositionConnected\rangle$ such that 
$$\psi(X)=\langle\SS(\delta X),\mathbf{w}\rangle\quad\forall X\in\evC.$$
\end{center}

\end{theorem}
\newcommand\trunc{\mathsf{trunc}}
\index[general]{trunc@$\trunc_{\tuIn{t}}$}
In order to prove this, we recall properties of multivariate polynomials. 
For every truncation level $\tuIn{t}\in\N^2$ let the \DEF{polynomial $\tuIn{t}$-truncation}
\begin{align*}\trunc_{\tuIn{t}}:\formalPorerSeriesOfFiniteDegree&\rightarrow\groundRing[x_{\tuIn{i}}\mid \tuIn{i}\leq\tuIn{t}]
\end{align*}
be a homomorphism of $\groundRing$-algebras which projects all  $x_{\tuIn{i}}$ with $\tuIn{i}\not\leq \tuIn{t}$ to zero.
In particular, $\eval$  defines an evaluation of truncated $\trunc_{\tuIn{t}}(f)$ to its polynomial function such that 
\begin{equation}\label{eq:truncatesToPolynomials}
(\trunc_{\tuIn{t}}f)(X)=f(X)    
\end{equation}
for all $X\in\evZ$ with $\size(X)\leq\tuIn{t}$. 
We now identify multivariate polynomials with its corresponding functions. 
\begin{lemma}\label{IdentificationPolysWithFunction}
Let $\groundRing$ be an infinite field.
\begin{enumerate}
    \item 
If $f\in\formalPorerSeriesOfFiniteDegree$ with $f(X)=0$ for all $X\in\evZ$, then $f=0$. 
\item If $f,g\in\formalPorerSeriesOfFiniteDegree$ with $f(X)=g(X)$ for all $X\in\evZ$, then $f=g$.
\end{enumerate}
\end{lemma}
\begin{proof} The second part follows from the first.
 For every $f\not=0$  there is a truncation level $\tuIn{t}\in\N^2$ such that  $\trunc_{\tuIn{t}}(f)\not=0$, 
  and thus there is $\mathbf{X}\in\groundRing^{\tuIn{t}_1\times\tuIn{t}_2}$ such that $(\trunc_{\tuIn{t}}f)(\mathbf{X})\not=0$.
  Considering  $\mathbf{X}$ as an element $X\in\evZ$ via 
  $$X_{\tuIn i}:=\begin{cases}{\mathbf{X}}_{\tuIn i}&\text{if }\tuIn{i}\leq\tuIn{t}\\
  0&\text{elsewhere,}\end{cases}$$ 
  we obtain 
   $f(X)=(\trunc_{\tuIn{t}}f)(\mathbf{X})\not=0$.
\end{proof}

\begin{corollary}\label{cor:invariantModInsertionOfZero}~
\begin{enumerate}
    \item\label{cor:invariantModInsertionOfZeroPart1} If $f\in\QSym$, then $\eval(f)$ is invariant to insertion of zeros in both directions independently. 
    \item\label{cor:invariantModInsertionOfZeroPart2} If $\groundRing$ is an infinite field, then the converse of part \ref{cor:invariantModInsertionOfZeroPart1}. is also true. 
\end{enumerate}
\end{corollary}
\begin{proof}
Part \ref{cor:invariantModInsertionOfZeroPart1}. is an immediate consequence of \Cref{lemma:ZeroVsFormalZero} and \Cref{lemma_polynomialInvariants_form}.
Part \ref{cor:invariantModInsertionOfZeroPart2}. additionally uses \Cref{IdentificationPolysWithFunction}.
\end{proof}

\begin{proof}[Proof of \Cref{the:identificationPolynomailInvariants}]
The backward direction is covered by \Cref{theorem:invariants} since concatenation of polynomials with formal power series remains a formal power series. 
We show that for any $\psi$ which is 
\begin{enumerate}
    \item invariant to warping in both directions independently, 
    \item invariant modulo constants, and
    \item polynomial, i.e., satisfies \Cref{eq:polynomialInvariantDef}, 
\end{enumerate}
there is a $\mathbf{w}\in\groundRing\langle\compositionConnected\rangle$ such that $
\psi(X)=\langle\SS(\delta X),\mathbf{w}\rangle$ for all $X\in\evC$. 
With \Cref{Theo:ZeroStutterBackForth}, 
 $\psi\circ\varsigma$ is zero insertion invariant and clearly induced by a formal power series of finite degree, hence there is
 $$\mathbf{w}=\sum_{1\leq i\leq m} \lambda_m\mathbf{a}_m\in\groundRing\langle\compositionConnected\rangle$$
 with $\lambda_i\in\groundRing$ and $\mathbf{a}_i\in\composition$ such that 
 $$\eval \left(\sum_{1\leq i\leq m}\lambda_i\,M_{\mathbf{a}_i}\right)=\psi\circ\varsigma$$ 
 with \Cref{lem:monomialBasis} and \Cref{cor:invariantModInsertionOfZero}.
 With $\eval(M_{\mathbf{a}_i})(\delta X)=\langle\SS(\delta X),\mathbf{a}_i\rangle$ and invariance modulo constants,  
 $$\psi(X)=\psi\circ\varsigma\circ\delta(X)= \langle\SS(\delta X),\mathbf{w}\rangle$$ for all $X\in\evC$ as claimed. 
\end{proof}

\section{Algorithmic considerations}

\subsection{Iterated two-parameter sums}\label{subsection:IteratedSScoeff}

For every eventually-constant $X\in\evC$,
the computation of $\Psi_{\mathbf{a}}(X)$ from \Cref{theorem:invariants} involves differences $Z=\delta(X)$, and the coefficient of the  two-parameter sums signature $\SS(Z)$  when tested at composition $\mathbf{a}$. 

Taking differences can be performed in linear time, i.e., the evaluation of $\delta(X)$ according to \Cref{eq:def_diff} requires  $\mathcal{O}(\row(X)\cdot\col(X))$ arithmetic operations for every $X\in\evC$. 

The naive evaluation of $
\langle\SS(Z),\mathbf{a}\rangle$ with $(Z,\mathbf{a})\in\evZ\times\composition$ however, sums over all pairs of increasing chains with lengths $\col(\mathbf{a})$ and $\row(\mathbf{a})$,  respectively. 
In general, each of those resulting summands is a product with $\row(\mathbf{a})\cdot\col(\mathbf{a})$ factors, leading to an upper complexity bound of \begin{equation}\label{eq:costNaiv}\mathcal{O}(\row(\mathbf{a})\cdot\col(\mathbf{a})\cdot\binom{\row(X)}{\row(\mathbf{a})}\cdot\binom{\col(X)}{\col(\mathbf{a})})\end{equation}
arithmetic operations for  evaluating $\Psi_{\mathbf{a}}(X)$. 

In the current section we investigate a certain subclass of matrix compositions for which the coefficients of the  two-parameter sums signature $
\langle\SS(Z),\mathbf{a}\rangle$ can be evaluated in linear time. 
This subclass can be considered as a chain of connected $1\times1$ compositions, for which iterative methods  similar to those from the one-parameter setting \cite{diehl2020tropical} remain possible. 
For this we define three binary chaining operations, in particular covering block diagonal matrices from 
\Cref{def:diag}.

\newcommand\chain{\operatorname{chain}}

\index[general]{chain@$\chain$}

\begin{definition}\label{def_diag}
For $a\in\{0,1,2\}$ let 
$$\chain_a:\composition\times\composition\rightarrow\composition$$
be a \DEF{chaining operation} where $(\mathbf{a},\mathbf{b})$ maps to  $\chain_0(\mathbf{a},\mathbf{b}):=\diag(\mathbf{a},\mathbf{b})$,
$$\chain_1(\mathbf{a},\,\mathbf{b}):=
\begin{small}
\begin{bmatrix}
\mathbf{a}_{1,1}
&\cdots
&\mathbf{a}_{1,\col(\mathbf{a})}
&0
&\cdots
&0
\\
\vdots
&\ddots
&\vdots
&\vdots
&\ddots
&\vdots\\
\mathbf{a}_{\row(\mathbf{a})-1,1}
&\cdots
&\mathbf{a}_{\row(\mathbf{a})-1,\col(\mathbf{a})}
&0
&\cdots
&0\\
\mathbf{a}_{\row(\mathbf{a}),1}
&\cdots
&\mathbf{a}_{\row(\mathbf{a}),\col(\mathbf{a})}
&{\mathbf{b}}_{1,1}
&\cdots
&{\mathbf{b}}_{1,\col(\mathbf{b})}\\
0
&\cdots 
&0
&{\mathbf{b}}_{2,1}
&\cdots
&{\mathbf{b}}_{2,\col(\mathbf{b})}\\
\vdots
&\ddots
&\vdots
&\vdots
&\ddots
&\vdots
\\
0
&\cdots 
&0
&{\mathbf{b}}_{\row(\mathbf{b}),1}
&\cdots
&{\mathbf{b}}_{\row(\mathbf{b}),\col(\mathbf{b})}
\end{bmatrix}
\end{small}$$
with an overlapping at axis $a=1$, and $\chain_2(\mathbf{a},\,\mathbf{b}):=$
$$
\begin{small}
\begin{bmatrix}
\mathbf{a}_{1,1}
&\cdots
&\mathbf{a}_{1,\col(\mathbf{a})-1}
&\mathbf{a}_{1,\col(\mathbf{a})}
&0
&\cdots
&0
\\
\vdots
&\ddots
&\vdots
&\vdots
&\vdots
&\ddots
&\vdots\\
\mathbf{a}_{\row(\mathbf{a}),1}
&\cdots
&\mathbf{a}_{\row(\mathbf{a}),\col(\mathbf{a})-1}
&\mathbf{a}_{\row(\mathbf{a}),\col(\mathbf{a})}
&0
&\cdots
&0\\
0
&\cdots 
&0
&{\mathbf{b}}_{1,1}
&{\mathbf{b}}_{1,2}
&\cdots
&{\mathbf{b}}_{1,\col(\mathbf{b})}\\
\vdots
&\ddots
&\vdots
&\vdots
&\ddots
&\vdots
\\
0
&\cdots 
&0
&{\mathbf{b}}_{\row(\mathbf{b}),1}
&{\mathbf{b}}_{\row(\mathbf{b}),2}
&\cdots
&{\mathbf{b}}_{\row(\mathbf{b}),\col(\mathbf{b})}
\end{bmatrix}
\end{small}$$
with an analogous overlapping at axes $a=2$.  
Regarding empty compositions, we set $\chain_a(\ec,\mathbf{b}):=\chain_a(\mathbf{b},\ec):=\mathbf{b}$ for all $(a,\textbf{b})\in\{0,1,2\}\times\composition$. 
\end{definition}

\begin{example}\label{ex:chaining}
$$\chain_1\left(\chain_0\left(\begin{bmatrix}\w{1}\end{bmatrix},\begin{bmatrix}\w{2}\star\w{3}\end{bmatrix}\right),\begin{bmatrix}\w{4}\end{bmatrix}\right)
=\begin{bmatrix}\w{1}&\vareps&\vareps\\\vareps&\w{2}\star\w{3}&\w{4}\end{bmatrix}\in\composition.$$
\end{example}

\begin{lemma}\label{lem:associativChain}
The chaining operation is interassociative, i.e., 
$$\chain_a(\mathbf{a},\chain_b(\mathbf{b},\mathbf{c}))=
\chain_b(\chain_a(\mathbf{a},\mathbf{b}),\mathbf{c})$$
for all $(\mathbf{a},\mathbf{b},\mathbf{c})\in\composition^3$ and $(a,b)\in\{0,1,2\}^2$. 
\end{lemma}
In particular $(\composition,\chain_a)$ is 
a non-commutative monoid for all $a\in\{0,1,2\}$.  
\index[general]{cumcum@$\CS$}
\index[general]{Zero@$\Zero$}
For axis $a\in\{1,2\}$ we define the  \DEF{cumulative sum} $\CS_a:\evC\rightarrow\evC$ via 
\begin{equation}\label{eq:defCumsum}
    {(\CS_{a}X)}_{\tuIn{i}}:=\sum_{j=1}^{{\tuIn{i}}_a}X_{\tuIn{i}+(j-{\tuIn{i}}_a)\e{a}}
\end{equation}
and we recall the \DEF{zero insertion operation} 
 $\Zero_{a,1}:\evC\rightarrow\evC$ with 
\begin{equation}\label{eq:defZero}
    {(\Zero_{a,1}X)}_{\tuIn{i}}:=\begin{cases}
0_d& {\tuIn{i}}_a = 1\\
X_{\tuIn{i}-\e{a}}& {\tuIn{i}}_a> 1\end{cases}
\end{equation}
from \Cref{subsection:zeroinsertion}. 
For convenience we 
set $\CS_0:=\CS_1\circ\CS_2$, $\Zero_{0,1}:=\Zero_{1,1}\circ\Zero_{2,1}$ and $\Zero_{a}:=\Zero_{a,1}$ for all $a\in\{0,1,2\}$. 

\begin{lemma}\label{word_block11_letter22}For $(Z,\mathbf{b})\in\evZ\times\composition$ let $A\in\evC$ with  
$$A_{\tuIn{r}}:=\langle \SS_{0_2;\tuIn{r}}(Z),\,\mathbf{b}\rangle$$
denote the two-parameter sums signature of $Z$ tested at $\mathbf{b}$. 
Then, 
\begin{equation}\label{eq:iterativeArgument}
\langle\SS_{0_2;\tuIn{r}}(Z),\,\chain_a(\mathbf{b},\begin{bmatrix}\lambda\end{bmatrix})\rangle={\CS_a(Z^{(\lambda)}\cdot\Zero_aA)}_{\tuIn{r}}\end{equation}
for all $1\times 1$ compositions $\begin{bmatrix}\lambda\end{bmatrix}\in\composition$ and $\tuIn{r}\leq\size(Z)$. 
\end{lemma}
Note that both $\CS_a$ and $\Zero_a$ can be computed in linear time, that is in $\mathcal{O}(\row(Z)\cdot\col(Z))$, leading to an linear evaluation of \Cref{eq:iterativeArgument}.
\index[general]{chained matrix compositions@$\closureDiag$}
Iteratively, this yields an efficient method to compute two-parameter signature coefficients for \DEF{chained matrix compositions}, denoted by 
$$
  \closureDiag\subseteq\composition.$$
  This set is defined to be the smallest set containing all $1\times 1$ compositions which is closed under $\chain_a$ for all $a\in\{0,1,2\}$. 
With \Cref{lem:associativChain}, 
\begin{align}\label{lemma_seq_def_UV}
\closureDiag
&=\left\{\underset{1\leq t\leq \ell}\bigcirc\;\chain_{a_t}(\bullet,\begin{bmatrix}\lambda_t\end{bmatrix})(\begin{bmatrix}
\lambda_0\end{bmatrix})
\;\begin{array}{|l}
\ell\in\N,\;\lambda_0\in\monoidComp_d\\a\in\{0,\ldots,2\}^\ell\\\lambda\in {(\monoidComp_d\setminus\{\monCompNeutrElem\})}^\ell
\end{array}\right\}
\end{align}
 can be thought of as sequential objects.

\begin{theorem}\label{Theo_diagInST}
For every $(\mathbf{a},Z)\in \closureDiag\times\evZ$, the entire matrix
\begin{align*}
{\left(\langle\SS_{0_2;\tuIn{r}}(Z),\mathbf{a}\rangle\right)}_{\tuIn{r}\leq \size(Z)}
\end{align*}
can be evaluated in linear time, i.e., requires  $$\mathcal{O}\left(\row(\mathbf{a})\cdot\col(\mathbf{a})\cdot\row(Z)\cdot\col(Z)\right)$$ 
arithmetic operations.  
In particular, so is $\langle\SS(Z),\mathbf{a}\rangle$. 
\end{theorem}
\begin{proof}
With \Cref{lemma_seq_def_UV} one can write 
$$\mathbf{a}=\underset{1\leq t\leq \ell}\bigcirc\;\chain_{a_t}(\bullet,\begin{bmatrix}\lambda_t\end{bmatrix})(\begin{bmatrix}\lambda_0\end{bmatrix}).$$
With  
$h_t(U):=(\Zero_{a_t}\circ\CS_{a_t})(Z^{(\lambda_t)}\cdot U)$
follows inductively
$$\langle\SS_{0_2;\tuIn{r}}(Z),\mathbf{a}\rangle={\left(\CS_0(Z^{(\lambda_{\ell})}\cdot (h_{\ell-1}\circ\cdots\circ h_{1})(1_\evC))\right)}_{\tuIn{r}}$$
for all $\tuIn{r}\leq\size(Z)$.
\end{proof}

We note that for arbitrary  $\lambda,\mu,\nu\in\monoidComp_d$ with $\lambda\not=\monCompNeutrElem\not=\nu$, the following $2 \times 2$  matrices 
\begin{align*}
  \begin{bmatrix}
    \lambda & \monCompNeutrElem \\
    \mu & \nu
  \end{bmatrix}\text{ and }
  \begin{bmatrix}
    \lambda & \mu \\
    \monCompNeutrElem & \nu
  \end{bmatrix}
\end{align*}
are all in $\closureDiag$. 
Chaining along the anti-diagonal leads to jet another class of matrix compositions similar to \Cref{lemma_seq_def_UV}.  
The only $2\times 2$ matrices which are not covered
by methods similar to \Cref{Theo_diagInST}
are 
\begin{align*}
  \begin{bmatrix}
    \lambda & \mu \\
    \xi & \nu
  \end{bmatrix}
\end{align*}
with $\lambda,\mu,\nu,\xi\in\monoidComp\setminus\{\monCompNeutrElem\}$ all being non-trivial.
We will momentarily see that for this type of matrix
we can still do better than the naive (in this case quadratic) cost explained in  \Cref{eq:costNaiv}.
\begin{lemma}\label{lem:OneDirecEfficOneNaive}
For every $Z\in\evZ$ and $2\times 2$ composition $\mathbf{a}$, the entire matrix
\begin{align*}
{\left(\langle\SS_{0;\tuIn{r}}(Z),\mathbf{a}\rangle\right)}_{\tuIn{r}\leq \size(Z)}
\end{align*}
can be evaluated in  $\mathcal{O}({\row(Z)}^2\cdot{\col(Z)})$ arithmetic operations. 
\end{lemma}
\begin{proof}
For all $(2,2)\leq\mathbf{r}\leq\size(Z)$,
$$\langle\SS_{0;\tuIn{r}}(Z),\mathbf{a}\rangle=\sum_{\iota_2=2}^{\mathbf{r}_1}\sum_{\kappa_2=2}^{\mathbf{r}_2}
Z_{\iota_2,\kappa_2}^{({\mathbf{a}}_{2,2})}\sum_{\iota_1=1}^{\iota_2-1}Z_{\iota_1,\kappa_2}^{({\mathbf{a}}_{1,2})}
\sum_{\kappa_1=1}^{\kappa_2-1}Z_{\iota_1,\kappa_1}^{({\mathbf{a}}_{1,1})}Z_{\iota_2,\kappa_1}^{({\mathbf{a}}_{2,1})}.$$
\end{proof}

\begin{lemma}
    If $\groundRing$ is the Boolean semiring and $Z\in\evC$ with $\size(Z)=(T,T)$, then
    \begin{align}
       \label{eq:boolean}
        \Big\langle \SS( Z ),
        \begin{bmatrix}
            1 & 1 \\
            1 & 1
        \end{bmatrix}
        \Big\rangle
    \end{align}
    can be computed in $   \O(T^\omega)$,
    where
    $\omega$ is the matrix-multiplication exponent.
\end{lemma}
\begin{proof}
    Consider $Z$ as a  $\{0,1\}$-valued matrix in $\mathbf{Z} \in \Z^{T\times T}$.
    Calculate
    \begin{align*}
        \mathbf{Z} \cdot \mathbf{Z}^\top,
    \end{align*}
    at cost $\O(T^\omega)$.
    The result contains an
    entry $2$ that is \emph{not} on the diagonal,
    if and only if there exist $\iota_1 < \iota_2$ and $\kappa_1 < \kappa_2$
    with
    \begin{align*}
        Z_{\iota_1,\kappa_1}
        \wedge
        Z_{\iota_1,\kappa_2}
        \wedge
        Z_{\iota_2,\kappa_1}
        \wedge
        Z_{\iota_2,\kappa_2}
    \end{align*}
    being true. 
    This is exactly
    what has to be checked in order to calculate
    \eqref{eq:boolean}.
\end{proof}

\subsection{Two-parameter quasi-shuffle}\label{sec:algo2dim}

We provide an algorithm to compute two-parameter quasi-shuffles of matrix compositions efficiently. 
A step-by-step illustration is given in \Cref{ex:algo2dimQSh}. 
Detailed explanations and a mathematical verification of soundness follows in   \Cref{theo_char_qs}.

\begin{algorithm}[H]
\SetKwInput{KwInput}{Input}                
\SetKwInput{KwOutput}{Output}              
\DontPrintSemicolon
  
  \KwInput{$\ \,$ Nonempty matrix compositions $\mathbf{a},\mathbf{b}\in\composition$.}
  \KwOutput{List of matrix compositions $L$ which sums up to the two-parameter  quasi-shuffle $\sum_{\mathbf{s}\in L}\mathbf{s}=\mathbf{a}\qShuffle\mathbf{b}\in\groundRing\langle\compositionConnected\rangle$.}

  \SetKwFunction{FMain}{main}
  \SetKwFunction{FcolQsh}{col_qsh}
  \SetKwFunction{FrowQsh}{row_qsh}
  \SetKwFunction{len}{len}
  \SetKwFunction{Fextend}{extend}
 
 ~\\
 
  \SetKwProg{Fn}{function}{:}{}
  \Fn{\FMain{$\mathbf{a}\in\monoidComp^{m\times n}$,
  $\mathbf{b}\in\monoidComp^{s\times t}$}}{
         $L,P\gets$ [ ], 
         \FcolQsh{$\begin{bmatrix}\mathbf{a}\\\monCompNeutrElem_{s\times n}\end{bmatrix}$,
         $\begin{bmatrix}\monCompNeutrElem_{m\times t}\\\mathbf{b}\end{bmatrix}$}
 
  \For{$\mathbf{p}\in P$}{
   $\begin{bmatrix}
   \mathbf{c}\\\mathbf{d}
   \end{bmatrix}\gets\mathbf{p}
   \text{ where }\mathbf{c}\in\monoidComp^{m\times j}\text{ and }\mathbf{d}\in\monoidComp^{s\times j}$\;
   $L$.\Fextend{\FrowQsh{$\mathbf{c}$,$\mathbf{d}$}}}

    \KwRet $L$
  }
\;
  \SetKwProg{Fn}{def}{:}{}
  \Fn{\FcolQsh{$c$, $d\in\groundRing\langle \monoidComp^{m+s}\rangle$}}{
        \If{$c=\ec$ or $d=\ec$}
        {\KwRet [$c\cdot d$]}
        $a\cdot v, b\cdot w\gets c,d
         \text{ where }a,b\in\monoidComp^{m+s}$\;
        \KwRet  ${a}\cdot
        \FcolQsh{${v}$,$b\cdot w$}+
        {b}\cdot
        \FcolQsh{$a\cdot v$,${w}$}+
        ({a}\star {b})\cdot
        \FcolQsh{${v}$, ${w}$}$\;
  }
  \;
  \SetKwProg{Fn}{def}{:}{}
  \Fn{\FrowQsh{$c$, $d\in\groundRing\langle \monoidComp^{1\times j}\rangle$}}{
        \KwRet [$\mathbf{x}^\top\text{ where }\mathbf{x}\in\FcolQsh{$c^\top$,$d^\top$}$]\;
  }
  \;

\end{algorithm}

\section{Technical details and proofs}\label{section:proofs}

\subsection{Signature properties}\label{subsection:invariants}
We provide omitted details and proofs from \Cref{subsection:zeroinsertion,sec:invariantsWarping,subsection:Chen}. 
In \Cref{subsection:Chen}, we extend the notion of size to $X \in \evC$, via
\begin{equation}\label{eq:minExist}
  \size(X):=\left(\row(X),\col(X)\right):=\min\left\{\tuIn{n}\in\N^2\mid X_{\tuIn{i}}\not=X_{\tuIn{j}}\implies \tuIn{i},\tuIn{j}\leq \tuIn{n}\right\},
\end{equation}
where the minimum is taken over the poset $(\N^2,\leq)$.
\Cref{lem:sizeOfEvZero} guarantees its existence. 
Note that every eventually-zero function can be described as a matrix after ``filling it up with zeros''. 
We show that \Cref{eq:minExist} relates to the size of matrices as expected. 

\begin{lemma}\label{lem:sizeOfEvZero}~
\begin{enumerate}
    \item\label{lem:sizeOfEvZero0} The minimum in \Cref{eq:minExist} exists.
    \item\label{lem:sizeOfEvZero1} For all $Z\in\evZ$, 
$$\size(Z)=\min\{\tuIn{n}\in\N^2\mid Z_{\tuIn{i}}\not=0_2\implies \tuIn{i}\leq \tuIn{n}\}$$
\end{enumerate}
\end{lemma}
\begin{proof}
For all $X\in\evC$, the maxima in 
$$\mathbf{k}:=\left(
\max\left\{m\in\N\mid X_{m,\bullet}\not=X_{m+1,\bullet}\right\},\max\left\{n\in\N\mid X_{\bullet,n+1}\not=X_{\bullet,n}\right\}\right)$$
exist. 
Furthermore, $\mathbf{k}$ satisfies the implication $X_{\tuIn{i}}\not=X_{\tuIn{j}}\implies \tuIn{i},\tuIn{j}\leq \tuIn{k}$ and is clearly minimal with this property. 
Part \ref{lem:sizeOfEvZero1}. follows from
$\lim_{\tuIn{n}\rightarrow\infty}Z_{\tuIn{n}}=0_2$.  
\end{proof}

\begin{proof}[Proof of \Cref{oneToOneCorrDiffSig}]
 The difference $\delta$ is clearly well-defined, i.e., $\delta(\evC)\subseteq\evZ$
 and linear.
From \Cref{lem:OnNormalFormsDeltaSigmaId} we get 
\begin{align*}
  \varsigma(\evZ)\subseteq \evZ \subset \evC.
\end{align*}
Linearity, 
\begin{align*}
\lambda{(\varsigma X)}_{i,j}+{(\varsigma Y)}_{i,j}
&=\lambda\left(\sum_{s=i}^\infty
   \sum_{t=j}^\infty X_{s,t}\right)+\sum_{s=i}^\infty
   \sum_{t=j}^\infty Y_{s,t}
   =\sum_{s=i}^\infty
   \sum_{t=j}^\infty \lambda X_{s,t}+ Y_{s,t}\\
&={\varsigma(\lambda X+Y)}_{i,j}\quad\forall(i,j,\lambda,X,Y)\in\N^2\times\groundRing\times\evZ^2, 
\end{align*}
where the sums are all finite since $\evZ$ is a $\groundRing$-module.
Moreover, $\varsigma$ is a linear right inverse of $\delta$.
Indeed,
$$\delta(\varsigma Z)_{i,j}=
\sum_{i+1\leq s}\left(
   \sum_{j+1\leq t} Z_{s,t}
-
   \sum_{j\leq t} Z_{s,t}\right)
+\sum_{i\leq s}\left(
\sum_{j\leq t} Z_{s,t}
   -\sum_{j+1\leq t} Z_{s,t}
   \right)=Z_{i,j} 
$$
for all $(i,j,Z)\in\N^2\times\evZ$. In particular, $\delta$ is surjective.
Regarding \ref{oneToOneCorrDiffSig2}., if $X\in\ker(\delta)$,  
then $$0=X_{\row(X)+1,\col(X)+1}-{X}_{\row(X)+1,\col(X)}-{ X}_{\row(X),\col(X)+1}+X_{\row(X),\col(X)}$$
and hence $X_{\row(X),\col(X)}=\lim_{\tuIn{i}\rightarrow\infty}X_{\tuIn{i}}$. 
Recursively one obtains $X_{\tuIn{j}}=\lim_{\tuIn{i}\rightarrow\infty}X_{\tuIn{i}}$ for all $\tuIn{j}\in\N^2$. 
Part \ref{oneToOneCorrDiffSig3}. follows from the homomorphism theorem for $\groundRing$-modules.  
\end{proof}

Furthermore, we show that $\delta$ and $\varsigma$ do not change the size of its input.  

\begin{lemma}\label{lem:sizeOfEvZeroDelta}~
\begin{enumerate}
\item\label{lem:sizeOfEvZero2} For all $X\in\evC$, 
$\size(\delta X)=\size(X)$.
\item\label{lem:sizeOfEvZero3} For all $Z\in\evZ$, 
$\size(\varsigma Z)=\size(Z)$.
\end{enumerate}
\end{lemma}
\begin{proof}
  For part \ref{lem:sizeOfEvZero2}., note that 
  $(\delta X)_{\tuIn{i}}=0$ for all $\tuIn{i}\not\leq \size(X)$, hence $\size(\delta X)\leq\size(X)$.
  Assuming $\size(\delta X)\not=\size(X)$, then 
  $(\delta X)_{\col(X),t}=0$ or $(\delta X)_{t,\row(X)}=0$ for all $t\in\N$. 
  We pursue the first case, the second yields a similar contradiction. 
  With $(\delta X)_{\col(X),t}=0$ for all $t$ and $$X_{\col(X),n}\not=X_{\col(X)+1,n}=X_{\col(X)+1,s}$$
  for suitable $n\leq \row(X)$ and arbitrary $s\in\N$ follows
  $X_{\col(X),n}=X_{\col(X),s}$ for all $s$. 
  In particular this implies $(\delta X)_{\size(X)}=X_{\size(X)}-X_{\col(X),\row(X)+1}\not=0$. 
  This also implies \ref{lem:sizeOfEvZero3}., since 
  $\size(Z)=\size(\delta(\varsigma Z))=\size(\varsigma Z)$ for every $Z\in\evC$. 
\end{proof}

Restricted to eventually-zero functions, $\delta$ and $\varsigma$ are isomorphisms, as the following statement implies.
\begin{lemma}\label{lem:OnNormalFormsDeltaSigmaId}~
\begin{enumerate}
    \item\label{lem:OnNormalFormsDeltaSigmaId1} $\varsigma(\evZ)\subseteq\evZ$.
    \item\label{lem:OnNormalFormsDeltaSigmaId2} $\forall Z\in\evZ\subseteq\evC:\varsigma(\delta Z)=Z$.
\end{enumerate}
\end{lemma}
\begin{proof}
  With \Cref{lem:sizeOfEvZero}, $\size(Z)=\size(\varsigma Z)$ and 
  $$\varsigma(Z)_{\row(Z),\col(Z)+1}=\sum_{s=\row(Z)}^\infty\;
   \sum_{t=\col(Z)+1}^\infty Z_{s,t}=0,$$
   hence $(\varsigma Z)_{\tuIn{n}}=0$ for all $\tuIn{n}\not\leq\size(Z)$, i.e., $\varsigma Z\in\evZ$. 
   Part \ref{lem:OnNormalFormsDeltaSigmaId2}. follows with 
   \begin{align*}
       {\left(\varsigma(\delta Z)\right)}_{i,j}
       &=\sum_{s=i}^{\row(Z)}\;
   \sum_{t=j}^{\col(Z)} Z_{s+1,\,t+1}-Z_{s+1,\,t}-Z_{s,\,t+1}+Z_{s,\,t}\\
   &=\sum_{s=i}^{\row(Z)}\;
    Z_{s+1,\,\col(Z)+1}-Z_{s+1,\,j}-Z_{s,\,\col(Z)+1}+Z_{s,\,j}=Z_{i,\,j}
   \end{align*}
   and $\size(Z)=\size(\delta Z)$.
\end{proof}

\begin{proof}[Proof of \Cref{lem:commuting_warp}, parts 
\ref{lem:commuting_warp1}., \ref{lem:commuting_warp2}. and \ref{lem:commuting_warp3}.]
 The mapping $\Stutter_{a,k}:\evC\rightarrow\evC$ from \Cref{sec:invariantsWarping} is well-defined (compare \Cref{lem:sizeOfStutter} for $\Stutter_{a,k}(\evC)\subseteq\evC$), linear with 
 \begin{align*}
    {(\Stutter_{a,k}\lambda X+Y)}_{\tuIn{i}}
    &=
    \begin{cases}
      {(\lambda X+Y)}_{\tuIn{i}}   & {\tuIn{i}}_a \le k\\
      {(\lambda X+Y)}_{\tuIn{i}-e_{a}}  & {\tuIn{i}}_a > k.
    \end{cases}\\
    &=\lambda{(\Stutter_{a,k}X)}_{\tuIn{i}}+{(\Stutter_{a,k}Y)}_{\tuIn{i}}
  \end{align*}
 for all $(\lambda,X,Y)\in\groundRing\times\evC^2$ and 
  injective since $\ker(\Stutter_{a,k})=0$. 
 Part \ref{lem:commuting_warp2}. follows with 
 \begin{align*}
    {(\Stutter_{1,k}(\Stutter_{2,j}X))}_{\tuIn{i}}
    &=
    \begin{cases}
      {(\Stutter_{2,j}X)}_{\tuIn{i}}       & {\tuIn{i}}_1 \le k\\
      {(\Stutter_{2,j}X)}_{(\tuIn{i}_1-1,\tuIn{i}_2)} & {\tuIn{i}}_1 > k
    \end{cases}\\
    &=
     \begin{cases}
      {X}_{\tuIn{i}}             & {\tuIn{i}}_1 \le k\land {\tuIn{i}}_2\leq j\\
      {X}_{({\tuIn{i}}_1,{\tuIn{i}}_2-1)} & {\tuIn{i}}_1 \le k\land {\tuIn{i}}_2>j\\
      {X}_{({\tuIn{i}}_1-1,{\tuIn{i}}_2)} & {\tuIn{i}}_1 > k \land {\tuIn{i}}_2\leq j \\
      {X}_{({\tuIn{i}}_1-1,{\tuIn{i}}_2-1)} & {\tuIn{i}}_1 > k \land {\tuIn{i}}_2>j\\
    \end{cases}\\
    &=
    \begin{cases}
      {(\Stutter_{1,k}X)}_{\tuIn{i}}       & {\tuIn{i}}_2 \le j\\
      {(\Stutter_{1,k}X)}_{({\tuIn{i}}_1,{\tuIn{i}}_2-1)} & {\tuIn{i}}_2 > j
    \end{cases}\\
    &={(\Stutter_{1,k}(\Stutter_{2,j}X))}_{\tuIn{i}}\quad \forall (\tuIn{i},j,k,X)\in\N^4\times\evC,
  \end{align*}
  and part \ref{lem:commuting_warp3}. with 
\begin{align*}
{(\Stutter_{a,k}(\Stutter_{a,j}X))}_{\tuIn{i}}
&=
    \begin{cases}
     {(\Stutter_{a,j}X)}_{\tuIn{i}}       & {\tuIn{i}}_a \le k\\
      {(\Stutter_{a,j}X)}_{\tuIn{i}-e_{a}} & {\tuIn{i}}_a > k
    \end{cases}\\
&=
    \begin{cases}
     {X}_{\tuIn{i}}       & {\tuIn{i}}_a \le k\land {\tuIn{i}}_a\leq j<k\\
     {X}_{\tuIn{i}-e_{a}}      & \tuIn{i}_a \le k \land \tuIn{i}_a > j\\
      {X}_{\tuIn{i}-e_{a}} & k<\tuIn{i}_a \land \tuIn{i}_a-1={(\tuIn{i}-e_{a})}_a \leq j<k\\
      {X}_{\tuIn{i}-2e_{a}} & \tuIn{i}_a > k>j  \land \tuIn{i}_a-1={(\tuIn{i}-e_{a})}_a > j
    \end{cases}\\
    &=
    \begin{cases}
     {X}_{\tuIn{i}}       & {\tuIn{i}}_a \le j<k\land {\tuIn{i}}_a\leq k-1\\
      {X}_{{\tuIn{i}}-e_{a}} & {\tuIn{i}}_a > j\land {\tuIn{i}}_a-1={({\tuIn{i}}-e_a)}_a\leq k-1\\
     {X}_{{\tuIn{i}}-e_{a}}   & {\tuIn{i}}_a \le j<k\land k-1< {\tuIn{i}}_a\\
     {X}_{{\tuIn{i}}-2e_{a}} & {\tuIn{i}}_a > j\land {\tuIn{i}}_a-1={({\tuIn{i}}-e_a)}_a>k-1>j-1
    \end{cases}\\
    &=
    \begin{cases}
     {(\Stutter_{a,k-1}X)}_{\tuIn{i}}       & {\tuIn{i}}_a \le j\\
      {(\Stutter_{a,k-1}X)}_{\tuIn{i}-e_{a}} & {\tuIn{i}}_a > j
    \end{cases}\\
    &={(\Stutter_{a,j}(\Stutter_{a,k-1} X))}_{\tuIn{i}}
\end{align*}
where $(a,\tuIn{i},j,k,X)\in\{1,2\}\times\N^4\times\evC$ and $k>j$. 
\end{proof}

\begin{lemma}\label{lem:sizeOfStutter}
For all $(a,k,X)\in\{1,2\}\times \N\times\evC$, 
$$\size(\Stutter_{a,k}X)=
\begin{cases}
\left(\row(X)+1,\col(X)\right)&a=1\land k\leq\row(X),\\
\left(\row(X),\col(X)+1\right)&a=2\land k\leq\col(X),\\
\size(X)&\text{elsewhere.}
\end{cases}$$
\end{lemma}
\begin{proof}
 Assume $a=1$ and $\row(X)\geq k$. 
 For all $m>\row(X)$, 
 \begin{align*}
      {(\Stutter_{1,k}X)}_{m+1,\,\bullet}
      &={X}_{m,\,\bullet}
      ={X}_{m+1,\,\bullet}
      ={(\Stutter_{1,k}X)}_{m+2,\,\bullet}
  \end{align*}
which implies $\row(\Stutter_{1,k}X)\leq\row(X)+1$. With 
\begin{align*}
      {(\Stutter_{1,k}X)}_{\row(X)+1,\,\bullet}
      &={X}_{\row(X),\,\bullet}
      \not={X}_{\row(X)+1,\,\bullet}
      ={(\Stutter_{1,k}X)}_{\row(X)+2,\,\bullet}
  \end{align*}
  follows equality. 
  Furthermore 
  $${(\Stutter_{1,k}X)}_{\bullet,\,n}={(\Stutter_{1,k}X)}_{\bullet,\,n+1}\iff
      {X}_{\bullet,\,n}
      \not={X}_{\bullet,\,n+1}
      $$
      for all $n\in\N$, hence $\col(\Stutter_{1,k}X)=\col(X)$. 
      The case $a=2$ and $\col(X)\geq k$ is treated similar. 
      The following cases cover all remaining possibilities.   
      \begin{enumerate}
          \item\label{caseStudySizeStutter1} If $a=2$ and $k>\col(X)$,
          then 
          $${(\Stutter_{2,k}X)}_{i,\,j}
          =\begin{cases}
          X_{i,\,j}&j\leq k\\
          X_{i,\,j-1}=X_{i,\,j}&j> k>\col(X),
          \end{cases}$$
          thus 
          $\size(\Stutter_{2,k}X)=\size(X)$.
          \item\label{caseStudySizeStutter2} If $k>\row(X)$ and $a=1$,
          then analogously to part \ref{caseStudySizeStutter1}.,  $\Stutter_{1,k}X=X$.
          \item The case $k>\row(X)$ and $k>\col(X)$ is covered by either part \ref{caseStudySizeStutter1}. or \ref{caseStudySizeStutter2}.
      \end{enumerate}
\end{proof}

    \begin{proof}(of the remaining part \ref{lem:commuting_warp4}. from \Cref{lem:commuting_warp}.)
We use the parts \ref{lem:commuting_warp2}. and \ref{lem:commuting_warp3}. to rewrite every 
    \begin{align*}
    A&=\Stutter_{a_q,k_q} \circ \cdots \circ {\Stutter_{a_1,k_1}}(B)
    \end{align*}
    with $(a,k,A,B)\in\{1,2\}^q\times \N^q\times \evC^2$ into the form 
    \begin{align*}
    A
    &=\Stutter_{2,j_m} \circ \cdots \circ {\Stutter_{2,j_1}}\circ\Stutter_{1,i_\ell} \circ \cdots \circ {\Stutter_{1,i_1}}(B)
    \end{align*}
    with strictly increasing chains $i_1<\cdots<i_\ell\leq\row(A)$ and $j_1<\cdots< j_m\leq\col(A)$. 
    With this and \Cref{lem:sizeOfStutter}, there exists the minimum 
    $$\min\left\{\size(B)\mid A=\Stutter_{a_q,k_q} \circ \cdots \circ {\Stutter_{a_1,k_1}}(B),\;
    (a,k,B)\in\{1,2\}^q\times\N^q\times\evC\right\}$$
    over the poset $(\N^2,\leq)$ for every $A\in\evC$.
    Let $B,B'\in\evC$ such that the minimum is reached with 
    \begin{align*}
    A&=\Stutter_{2,j_m} \circ \cdots \circ {\Stutter_{2,j_1}}\circ\Stutter_{1,i_\ell} \circ \cdots \circ {\Stutter_{1,i_1}}(B)\\
    &=\Stutter_{2,j'_m} \circ \cdots \circ {\Stutter_{2,j'_1}}\circ\Stutter_{1,i'_\ell} \circ \cdots \circ {\Stutter_{1,i'_1}}(B'),
    \end{align*}
    where the strictly increasing chains $i,i'$ and $j,j'$ are of same length due to \Cref{lem:sizeOfStutter} and $\size(B)=\size(B')$. 
    If $i_1<i'_1<\cdots<i'_\ell$, then 
    $A_{i_1,\bullet }=A_{i_1+1,\bullet}$ by definition, and hence, there exists $B'':=\Stutter_{1,i_1}(B')\in\evC$ such that 
    \begin{align*}
    A
    =\Stutter_{2,j'_m} \circ \cdots \circ {\Stutter_{2,j'_1}}\circ\Stutter_{1,i'_\ell} \circ \cdots \circ {\Stutter_{1,i'_1}}\circ{\Stutter_{1,i_1}}(B''),
    \end{align*}
    contradicting that $\size(B)=\size(B')$ is minimal. 
    Therefore $i_1=i_1'$, thus with injectivety of $\Stutter_{1,i_1}$ and parts \ref{lem:commuting_warp2}. and \ref{lem:commuting_warp3}., 
    \begin{align*}
    A'&=\Stutter_{2,j_m} \circ \cdots \circ {\Stutter_{2,j_1}}\circ\Stutter_{1,i_\ell-1} \circ \cdots \circ {\Stutter_{1,i_2-1}}(B)\\
    &=\Stutter_{2,j'_m} \circ \cdots \circ {\Stutter_{2,j'_1}}\circ\Stutter_{1,i'_\ell-1} \circ \cdots \circ {\Stutter_{1,i'_2-1}}(B')
    \end{align*}
    for $A'\in\evC$ with $\Stutter_{1,i_1}(A')=A$. 
    Recursively we obtain $B=B'$. 
    \end{proof}

Clearly $\Zero_{a,k}:\evZ\rightarrow\evZ$ from  \Cref{subsection:zeroinsertion} is  well-defined. 
We use the following relations to show properties of $\Zero_{a,k}$ via the properties of $\Stutter_{a,k}$. 

\begin{proof}[Proof of \Cref{lemma:ZeroVsStutter}, parts \ref{lemma:ZeroVsStutter1}. and \ref{lemma:ZeroVsStutter2}.]
Let  $a=1$ and $X\in\evC$.
Then, 
\begin{align*}
{\left(\Zero_{1,k}(\delta X)\right)}_{i,j}
&=
\begin{cases}X_{i+1,j+1}
-X_{i+1,j}
-X_{i,j+1}
+X_{i,j}&i<k\\
X_{i,j+1}
-X_{i,j}
-X_{i,j+1}
+X_{i,j}& i = k\\
X_{i,j+1}
-X_{i,j}
-X_{i-1,j+1}
+X_{i-1,j}& i> k\end{cases}\\
&={\delta(\Stutter_{1,k}X)}_{i,\,j}\quad\forall (i,j,k,X)\in\N^3\times\evC.
\end{align*}
The case $a=2$ is treated similar. 
With this and \Cref{oneToOneCorrDiffSig,lem:OnNormalFormsDeltaSigmaId} follows part  \ref{lemma:ZeroVsStutter2}. via  
    $$\varsigma\circ \Zero_{a,k}{}=\varsigma\circ \Zero_{a,k}{}\circ\delta\circ\varsigma=\varsigma\circ\delta\circ \Stutter_{a,k}{}\circ\varsigma={\Stutter_{a,k}}{}\circ{\varsigma}$$
    for every $(a,k)\in\{1,2\}\times\N$. 
    \end{proof}

       \begin{lemma}\label{lemma:sizeZeroZ}
For all $(a,k,Z)\in\{1,2\}\times\N\times\evZ$, 
$$\size(\Zero_{a,k}Z)=
\begin{cases}
\left(\row(Z)+1,\col(Z)\right)&a=1\land k\leq\row(Z),\\
\left(\row(Z),\col(Z)+1\right)&a=2\land k\leq\col(Z),\\
\size(Z)&\text{elsewhere.}
\end{cases}$$
\end{lemma}
\begin{proof}
Follows with \Cref{lem:sizeOfStutter},  $\size{}\circ\varsigma=\size$ and 
$$\size{}\circ\Zero_{a,k}{}
=\size{}\circ\Zero_{a,k}{}\circ\delta\circ\varsigma
=\size{}\circ\delta\circ\Stutter_{a,k}{}\circ\varsigma 
=\size{}\circ\Stutter_{a,k}{}\circ\varsigma$$
for all $(a,k)\in\{1,2\}\times\N$. 
\end{proof}

    \begin{proof}[Proof of \Cref{lem:commuting_zero}]
    With \Cref{oneToOneCorrDiffSig,lemma:ZeroVsStutter}, 
    $$\Zero_{a,k}{}=\Zero_{2,j}{}\circ\delta\circ\varsigma=\delta\circ\Stutter_{2,j}{}\circ\varsigma$$ is 
    linear for all $(a,k)\in\{1,2\}\times \N$, 
    \begin{align*}
    \Zero_{1,k}{}\circ\Zero_{2,j}{}
    &=\delta\circ\Stutter_{1,k}{}\circ\Stutter_{2,j}{}\circ\varsigma\\
    &=\delta\circ\Stutter_{2,j}{}\circ\Stutter_{1,k}{}\circ\varsigma
    =
    \Zero_{2,j}\circ\Zero_{1,k}
    \end{align*}
 for all $(j,k)\in\N^2$ with \Cref{lem:commuting_warp}, and analogously, part  \ref{lem:commuting_zero3}.   
 With \Cref{lemma:sizeZeroZ},
 we get a minimal $\delta(\NFstut(\varsigma(Z)))\in\evZ$ for every $Z\in\evZ$ with 
 \begin{align*}
Z&=\delta(\varsigma(Z))
=\delta\circ\Stutter_{2,j_m} \circ \cdots \circ {\Stutter_{2,j_1}}\circ\Stutter_{1,i_\ell} \circ \cdots \circ {\Stutter_{1,i_1}}(\NFstut(\varsigma(Z)))\\
&=
\Zero_{2,j_m} \circ \cdots \circ {\Zero_{2,j_1}}\circ\Zero_{1,i_\ell} \circ \cdots \circ {\Zero_{1,i_1}}(\delta(\NFstut(\varsigma(Z))))
\end{align*}
for increasing chains $i_1<\cdots <i_\ell$ and $j_1<\cdots< j_m$. 
 Uniqueness follows analogously to \Cref{lem:commuting_warp} via injective $\Zero_{a,k}$ and previous parts \ref{lem:commuting_zero2}. and  \ref{lem:commuting_zero3}.   
\end{proof}
    
    \begin{proof}[Proof of the remaining parts from  \Cref{lemma:ZeroVsStutter}]
    For $A\in\evC$, let $a,k$ such that 
        \begin{align*}
    \delta(A)&=\delta\circ\Stutter_{a_q,k_q} \circ \cdots \circ {\Stutter_{a_1,k_1}}\circ\NFstut(A)\\
    &=\Zero_{a_q,k_q} \circ \cdots \circ {\Zero_{a_1,k_1}}\circ\delta\circ\NFstut(A), 
    \end{align*}
   hence part \ref{lemma:ZeroVsStutter3}. via uniqueness. 
    With this follows also 
   $$\varsigma\circ\NFzero{}
   =\varsigma\circ\NFzero{}\circ\delta\circ\varsigma
   =\varsigma\circ\delta\circ\NFstut{}\circ\varsigma,$$
   and thus part \ref{lemma:ZeroVsStutter4}.  since  $\varsigma\circ\delta(X)=X$ for all $X\in\evZ\subseteq\evC$. 
    Part 
    \ref{lemma:ZeroVsStutter5}. holds via $$\NFconst{}\circ\Stutter_{a,k}{}=\varsigma\circ\Zero_{a\,k}{}\circ\delta=\Stutter_{a,k}{}\circ\NFconst{},$$
    resulting in part \ref{lemma:ZeroVsStutter6}.  analogously as in \ref{lemma:ZeroVsStutter3}.
\end{proof}

We show that $(\evC,\boxslash)$ is a semigroup. 
Together with \Cref{lem:assOfObash}, $\delta$ is a surjective homomorphism of semigroups.

\begin{lemma}\label{lem:concatenationMat}
For all $X,Y,Z\in\evC$, 
\begin{enumerate}
    
    \item\label{lem:concatenationMat1} $\row(X\boxslash Y)=\row(X)+\row(Y)$, 
    \item\label{lem:concatenationMat2} $\col(X\boxslash Y)=\col(X)+\col(Y)$, 
    \item\label{lem:concatenationMat3} $(X\boxslash Y)\boxslash Z=X\boxslash(Y\boxslash Z)$, and 
    \item\label{lem:concatenationMat4}
     $\delta(X\boxslash Y)=\delta(X)\varoslash\delta(Y)$.
\end{enumerate}
\end{lemma}
\begin{proof}
Clearly 
$$\left(\row(X),\col(X)\right)\leq\left(\row(X\boxslash Y), \col(X\boxslash Y)\right).$$ 
 Assume $\tuIn{n}\not\leq\left(\row(X)+\row(Y),\col(X)+\col(Y))\right)$ with   $\tuIn{n}_1>\row(X)+\row(Y)$. 
 Then 
\begin{align*}
{(X\boxslash Y)}_{\tuIn{n}}
&={\left(\NFconst(X)+\Stutter_{1,1}^{\row(X)}\left(\Stutter_{2,1}^{\col(X)}(Y)\right)\right)}_{\tuIn{n}}\\
&=\NFconst(X)_{\tuIn{n}}+Y_{(\tuIn{n}_1-\row(X),m)}=\lim_{\tuIn{i}\rightarrow\infty}Y_{\tuIn{i}}
\end{align*}
for suitable $m\in\N$. 
The remaining case ${\tuIn{n}}_2>\col(X)+\col(Y)$ is treated similar, hence 
$$\left(\row(X)+\row(Y),\col(X)+\col(Y)\right)\leq\left(\row(X\boxslash Y), \col(X\boxslash Y)\right).$$ 
Equality follows with 
\begin{equation*}\label{eq:proofChenCase1}
\left(\row(X),\col(X)\right)< \tuIn{n}\implies 
    {(X\boxslash Y)}_{\tuIn{n}}=Y_{\tuIn{n}}. 
\end{equation*}
This shows part \ref{lem:concatenationMat1}. and \ref{lem:concatenationMat2}. 
\Cref{lem:commuting_warp,oneToOneCorrDiffSig,lemma:ZeroVsStutter} yield
\begin{align*}
(X\boxslash Y)\boxslash Z
&=\NFconst\left(\NFconst(X)+
\Stutter_{1,1}^{\row(X)}\left(\Stutter_{2,1}^{\col(X)}(Y)\right)\right)
\\&\;\;\;+\Stutter_{1,1}^{\row(X\boxslash Y)}\left(\Stutter_{2,1}^{\col(X\boxslash Y)}(Z)\right)
\\
&=\NFconst(X)+
\Stutter_{1,1}^{\row(X)}\left(\Stutter_{2,1}^{\col(X)}\left(\NFconst(Y)\right)\right)\\
&\;\;\;+\Stutter_{1,1}^{\row(X)}\left(\Stutter_{2,1}^{\col(X)}\left(\Stutter_{1,1}^{\row(Y )}\left(\Stutter_{2,1}^{\col(Y)}(Z)\right)\right)\right)
\\
&=X\boxslash(Y\boxslash Z),
\end{align*}
i.e. associativity of $\boxslash$. 
Part  \ref{lem:concatenationMat4}. follows with 
\begin{equation*}
\delta(X\boxslash Y)=\delta\circ\varsigma\circ\delta(X)+\Zero_{1,1}^{\row(\delta X)}\left(\Zero_{2,1}^{\col(\delta X)}(\delta Y)\right)= \delta(X)\varoslash\delta(Y), 
\end{equation*}
and \Cref{lem:sizeOfEvZero}.
\end{proof}

Also $\varoslash$ is associative and $\varsigma$ an injective homomorphism of semigroups. 

\begin{lemma}\label{lem:assOfObash}For all $X,Y,Z\in\evC$, 
\begin{enumerate}
    \item\label{lem:concatenationdiag1} $\row(X\varoslash Y)=\row(X)+\row(Y)$, 
    \item\label{lem:concatenationdiag2} $\col(X\varoslash Y)=\col(X)+\col(Y)$, 
    \item\label{lem:concatenationdiag3} $(X\varoslash Y)\varoslash Z=X\varoslash(Y\varoslash Z)$, and 
    \item\label{lem:concatenationdiag4}  $\varsigma(X\varoslash Y)=\varsigma(X)\boxslash\varsigma(Y)$.
\end{enumerate}
\end{lemma}
\begin{proof}
  Part \ref{lem:concatenationdiag1} 
  follows with \Cref{lem:sizeOfEvZero,lem:concatenationMat} via 
  \begin{align*}
  \row(X\varoslash Y)
  &=\row\left(\delta\left(\varsigma X\right)\varoslash\delta\left(\varsigma Y\right)\right)\\
  &=\row\left(\delta\left(\varsigma(X)\boxslash\varsigma(Y)\right)\right)
  =\row(\varsigma X)+\row(\varsigma Y)=\row(X)+\row(Y),\end{align*}
  similarly part \ref{lem:concatenationdiag2}. 
 Associativity is similar as in \Cref{lem:concatenationMat}, and with
  \begin{equation*}
\varsigma(X\varoslash Y)=\varsigma\circ\delta\circ\varsigma(X)+\Stutter_{1,1}^{\row( \varsigma X)}\left(\Stutter_{2,1}^{\col(\varsigma  X)}(\varsigma Y)\right)= \varsigma(X)\boxslash\varsigma(Y) 
\end{equation*}
follows part \ref{lem:concatenationdiag4}.
\end{proof}

With now prove Chen’s identity with respect to
diagonal concatenation. 

\begin{proof}[Proof of \Cref{thm:chen}.] 
Let $\tuIn{j}=\size(A)$ and $Z=A\varoslash B$. 
For every $\mathbf{a}\in\composition$ with $\size(\mathbf{a})=(m,n)$ follows
$\langle \ISS_{\tuIn{\ell};\tuIn{r}}(Z),\mathbf{a}\rangle$
\begin{align}
    &=
    \sum_{\substack{{\tuIn{\ell}}_1<\iota_1<\ldots<\iota_m\leq {\tuIn{r}}_1\\{\tuIn{\ell}}_2<\kappa_1<\ldots<\kappa_n\leq {\tuIn{r}}_2}}\;\prod_{s=1}^m\prod_{t=1}^nZ_{\iota_s,\kappa_t}^{({\mathbf{a}}_{s,t})}\nonumber\\
    &=\sum_{\substack{0\leq u\leq m\\0\leq v\leq n}}\;\sum_{\substack{{\tuIn{\ell}}_1<\iota_1<\ldots<\iota_{u}\leq {\tuIn{j}}_1<\iota_{u+1}<\ldots<\iota_m\leq {\tuIn{r}}_1\\{\tuIn{\ell}}_2<\kappa_1<\ldots<\kappa_{v}\leq \tuIn{j}_2<\kappa_{v+1}<\ldots<\kappa_n\leq {\tuIn{r}}_2}}\;\prod_{s=1}^m\prod_{t=1}^nZ_{\iota_s,\kappa_t}^{({\mathbf{a}}_{s,t})}\label{eq:bigSum_withZeros}\\
    &=\sum_{\substack{\diag(\mathbf{b},\mathbf{c})=\mathbf{a}\\\text{where } \mathbf{c}\not=\ec\\\
    \text{and }\mathbf{b}\in\monoidComp_d^{u\times v}}}
    \left(
    \sum_{\substack{{\tuIn{\ell}}_1<\iota_1<\ldots<\iota_u\leq {\tuIn{j}}_1\\{\tuIn{\ell}}_2<\kappa_1<\ldots<\kappa_v\leq {\tuIn{j}}_2}}\;\prod_{s=1}^u\prod_{t=1}^vZ_{\iota_s,\kappa_t}^{({\mathbf{b}}_{s,t})}
    \right)
    \left(
    \sum_{\substack{{\tuIn{j}}_1<\iota_1<\ldots<\iota_m\leq {\tuIn{r}}_1\\{\tuIn{j}}_2<\kappa_1<\ldots<\kappa_n\leq {\tuIn{r}}_2}}\;\prod_{s=1}^{m-u}\prod_{t=1}^{n-v}Z_{\iota_s,\kappa_t}^{({\mathbf{c}}_{s,t})}
    \right)\nonumber\\
    &\;\;\;\;\;+\sum_{\substack{{\tuIn{\ell}}_1<\iota_1<\ldots<\iota_m\leq {\tuIn{j}}_1\\{\tuIn{\ell}}_2<\kappa_1<\ldots<\kappa_n\leq {\tuIn{j}}_2}}\;\prod_{s=1}^m\prod_{t=1}^nZ_{\iota_s,\kappa_t}^{({\mathbf{a}}_{s,t})}
    +\sum_{\substack{{\tuIn{j}}_1<\iota_1<\ldots<\iota_m\leq {\tuIn{r}}_1\\{\tuIn{j}}_2<\kappa_1<\ldots<\kappa_n\leq {\tuIn{r}}_2}}\;\prod_{s=1}^m\prod_{t=1}^nZ_{\iota_s,\kappa_t}^{({\mathbf{a}}_{s,t})}\nonumber
\end{align}
where the summands in \Cref{eq:bigSum_withZeros} are always zero whenever $$\mathbf{a}=\begin{bmatrix}\mathbf{b}&\mathbf{u}\\\mathbf{v}&\mathbf{c}\end{bmatrix}$$ with $\vareps_{u\times(n-v)}\not=\mathbf{u}$ or $\vareps_{(m-u)\times v}\not=\mathbf{v}$ for $1\leq u\leq m-1$ and $1\leq v\leq n-1$. 
\end{proof}

\begin{proof}[Proof of \Cref{theorem_bahntrennend}]
The forward direction is discussed in  \Cref{cor:invariantModInsertionOfZero} for the case $d=1$.
As usual, the case $d>1$ is analogous.
We show the backwards direction. 
If $\size(X)\not=\size(Y)$, then $\SS(X)\not=\SS(Y)$ and $X\not\sim Y$. 
It therefore suffices to assume $\SS(X)=\SS(Y)$,  $X=\NFsim(X)$, $\size(X)=\size(Y)=(S,T)$, and to verify $X=Y$. 
Define the matrix composition $\mathbf{a}(\w{j},\mathbf{X})\in\monoidComp_{d}^{S\times T}$
with entries 
\begin{equation}\label{eq:def:aSj}
{\mathbf{a}(\w{j},X)}_{s,t}:=
\begin{cases}
\w{j}&\text{if }X_{s,t}^{(\w{j})}\not=0,\\
\w{i}&\text{if there exists }i\leq d\text{, chosen minimally, such that }X_{s,t}^{(\w{i})}\not=0,\\
\vareps&\text{elsewhere.}
\end{cases}
\end{equation}
This matrix is a composition since 
$(X_{s,t})_{s\leq S,t\leq T}$ has no zero lines (columns treated analogously), and therefore, for every column $X_{\bullet,t}$ with $t\leq T$ there exist $i\leq d$ and $s\leq S$ such that $X_{s,t}^{(\w i)}\not=0$ and   ${\mathbf{a}(\w{j},X)}_{s,t}=\w{i}$. 
Furthermore, since $\groundRing$ has no zero divisors, 
$$\langle\SS(X),\mathbf{a}(\w{j},X)\rangle\not=0$$
for all $j\leq d$. 
Assume there is $j\leq d$ with $\mathbf{a}(\w{j},X)\not=\mathbf{a}(\w{j},Y)$, then there is $(s,t)\leq(S,T)$ with either
$$\begin{cases}
\mathbf{a}(\w{j},X)_{s,t}=\w{j}\not=\w{i}=\mathbf{a}(\w{j},Y)_{s,t}\text{ or}\\
\mathbf{a}(\w{j},X)_{s,t}=\w{i}>\w{k}=\mathbf{a}(\w{j},Y)_{s,t}\text{ or}\\
\mathbf{a}(\w{j},X)_{s,t}=\w{i}\not=\vareps=\mathbf{a}(\w{j},Y)_{s,t}\text{ for suitable }i\leq d, 
\end{cases}$$
after possibly exchanging the role of $X$ and $Y$.
With the minimality condition in \Cref{eq:def:aSj},
$$0\not=\langle\SS(X),\mathbf{a}(\w{j},X)\rangle
=\langle\SS(Y),\mathbf{a}(\w{j},X)\rangle=0,$$
thus $\mathbf{a}(\w{j}):=\mathbf{a}(\w{j},X)=\mathbf{a}(\w{j},Y)$. In  particular holds $X_{s,t}^{(\w{j})}=0 \iff Y_{s,t}^{(\w{j})}=0$.
If $X_{s,t}^{(\w{j})}\not=0$, then
$$X_{s,t}^{(\w{j})}
\prod_{s=1}^S\prod_{t=1}^TX_{s,t}^{({\mathbf{a}(\w{j})}_{s,t})}=
Y_{s,t}^{(\w{j})}\prod_{s=1}^S\prod_{t=1}^TY_{s,t}^{({\mathbf{a}(\w{j})}_{s,t})}=
Y_{s,t}^{(\w{j})}\prod_{s=1}^S\prod_{t=1}^TX_{s,t}^{({\mathbf{a}(\w{j})}_{s,t})},
$$
thus $X_{s,t}^{(\w j)}=Y_{s,t}^{(\w j)}$, by the cancelation property\footnote{The cancelation property is satisfied via the quotient field of commutative rings without zero-divisors.} of $\groundRing$. 
\end{proof}

\subsection{Two-parameter quasi-shuffle}\label{subsec:ommittedDetQSh}

We provide omitted details and proofs from \Cref{ss:hopfAlgebra,subsection:quasishuffle,sec:algo2dim}. 
We define a more convenient description of the two-parameter quasi-shuffle via surjections, acting on columns or rows of compositions. 
For $1\leq  i\leq j$, let $e_i\in\N_0^j$ encode the $i$-th \DEF{standard column} defined via the Kronecker delta  ${(e_i)}_k:=\delta_{i,k}$. 
  
Consider the following set of non-negative integer matrices
$$\QSH(m,s;j):=\left\{\begin{bmatrix}\e{\iota_1}&\cdots&\e{\iota_m}&\e{\kappa_1}&\cdots&\e{\kappa_s}\end{bmatrix}\in\N_0^{j\times(m+s)}\;
\begin{array}{|l}
  \iota_{1}<\cdots<\iota_{m} \\
  \kappa_{1}<\cdots<\kappa_s \\
  \text{right invertible}
\end{array}\right\}.$$

\begin{remark}
\label{rem:onetoonesurjtomat}
$\QSH(m,s;j)$ is in one-to-one correspondence\footnote{
Throughout, the bijection $\varphi$ from $\qSh(m,s;j)$ to $\QSH(m,s;j)$ is  omitted.
For $q\in\qSh(m,s;j)$ the matrix 
$\varphi(q):=\mathbf{Q}:=\begin{bmatrix}\e{q(1)}&\cdots&\e{q(m)}&\e{q(m+1)}&\cdots&\e{q(m+s)}\end{bmatrix}$
is right invertible by construction and thus contained in $\QSH(m,s;j)$. 
The converse direction is similar. 
Compare also \Cref{lem:relPandp}.} with $\qSh(m,s;j)$.
\end{remark}
For every $\mathbf{P}\in\N_0^{j\times m}$ and matrix composition  $\mathbf{a}\in\monoidComp^{m\times n}$ let $\mathbf{P}\mathbf{a}\in\monoidComp^{j\times n}$ denote the \DEF{action on the row space} of  $\mathbf{a}$, 
$${(\mathbf{P}\mathbf{a})}_{\iota,\nu}:=\underset{1\leq \mu\leq m}\bigstar \mathbf{a}_{\mu,\nu}^{\mathbf{P}_{\iota,\mu}},\quad(\iota,\nu)\leq (j,n).$$
Analogously, let $\mathbf{a}\mathbf{Q}^\top\in\monoidComp^{m\times k}$ denote the \DEF{action on the column space} via $\mathbf{Q}\in\N_0^{k\times n}$. 

\begin{example}
The surjection $q\in\qSh(2,2;3)$ with $q(1)=1$, $q(2)=3$, $q(3)=2$ and $q(4)=3$ defines the following action on rows,
$$\begin{bmatrix}\w{1}\\\w{1}\star\w{3}\\\w{1}\star\w{2}\star\w{4}\end{bmatrix}=
\begin{bmatrix}
1&0&0&0\\
0&0&1&0\\
0&1&0&1
\end{bmatrix}
\,
\begin{bmatrix}\w{1}\\\w{1}\star\w{4}\\\w{1}\star\w{3}\\\w{2}\end{bmatrix}\in\monoidComp_4^3.
$$
\end{example}

With this language we can bring the matrix (\ref{eq:c_quasishuffle}) of products over preimages from  \Cref{def:twodim_qshuffle} in a convenient form. 
Note that this viewpoint is also valid for the classical one-parameter setting.

\begin{lemma}\label{lem:relPandp}~
\begin{enumerate}
\item For all $\mathbf{a}\in\monoidComp^{m+s}$ and $p\in\qSh(m,s;j)$ with $\mathbf{P}=\varphi(p)$,
$$\begin{bmatrix}
\underset{u\in p^{-1}(1)}\bigstar\mathbf{a}_{u}\\
\vdots\\
\underset{u\in p^{-1}(j)}\bigstar\mathbf{a}_{u}
\end{bmatrix}=\mathbf{P}\mathbf{a}\in\mathfrak{M}^{j}.$$
\item For all $\mathbf{b}\in\monoidComp^{(m+s)\times(n+t)}$ and $q\in\qSh(n,t;k)$ with $\mathbf{Q}=\varphi(q)$,  
$${\left(\underset{\substack{u\in p^{-1}(x)\\v\in q^{-1}(y)}}\bigstar\mathbf{b}_{u,v}\right)}_{x,y}=\mathbf{P}\mathbf{b}{\mathbf{Q}}^\top\in\mathfrak{M}^{j\times k}.$$
\end{enumerate}
\end{lemma}

\begin{lemma}\label{def_2ParamQS}~
  \begin{enumerate}
\item For all compositions
$(\mathbf{a},\mathbf{b})\in\monoidComp^{m\times n}\times\monoidComp^{s\times t}$, 
$$\mathbf{a}\qShuffle\mathbf{b}
=\sum_{j,k\in\N}\;\; 
\sum_{\substack{\mathbf{P}\in\QSH(m,s;j)\\\mathbf{Q}\in\QSH(n,t;k)}}
\mathbf{P}\diag(\mathbf{a},\mathbf{b}) \mathbf{Q}^\top.
$$
      \item 
       For monomials  $\mathbf{a},\mathbf{b}\in\groundRing\langle\monoidComp^{1\times n}\rangle$ with $\deg(\mathbf{a})=m$ and $\deg(\mathbf{b})=s$, $$\mathbf{a}\qShuffle_1\mathbf{b}=\sum_{ j\in\N}\,\sum_{\mathbf{P}\in\QSH(m,s;j)}\mathbf{P}\begin{bmatrix}\mathbf{a}\\\mathbf{b}\end{bmatrix}$$
      \item
    For monomials  $\mathbf{c},\mathbf{d}\in\groundRing\langle\monoidComp^m\rangle$ with $\deg(\mathbf{c})=n$ and $\deg(\mathbf{d})=t$,
     $$\mathbf{c}\qShuffle_2\mathbf{d}=\sum_{k\in\N}\,\sum_{\mathbf{Q}\in\QSH(n,t;k)}\begin{bmatrix}\mathbf{c}&\mathbf{d}\end{bmatrix} \mathbf{Q}^\top$$
  \end{enumerate}
\end{lemma}

In \Cref{lemma_qsh_ass,Cor_sqsProp} we  show that  $(\groundRing\langle\compositionConnected\rangle,+,\qShuffle,0,\ec)$ is a commutative algebra. 
\Cref{theo_char_qs} brings the two-parameter quasi-shuffle into relation to the well-known one-parameter setting from \cite{EBRAHIMIFARD2017552,hoffman1999quasishuffle}. 
In particular one can use the recursive 
characterization  in the one-parameter setting for an efficient evaluation of \Cref{def_2ParamQS} relying not on surjections. 
As a preliminary consideration, we recall concatenation of surjections for one-parameter quasi-shuffles, encoded as matrices.

\begin{lemma}~\label{lemma_qSH_dec}
\begin{enumerate}
\item\label{lemma_qSH_dec_part1}
For all $\mathbf{Q}\in\QSH(m,s;j)$
exist unique block matrices
$\mathbf{Q}_1\in U(j,m)$ and 
$\mathbf{Q}_2\in U(j,s)$ such that 
$\mathbf{Q}=\begin{bmatrix}\mathbf{Q}_1&\mathbf{Q}_2\end{bmatrix}$ where 
$$U(j,\ell):=\left\{\begin{bmatrix}\e{\iota_1}&\ldots&\e{\iota_\ell}\end{bmatrix}\in\N_0^{j\times\ell}
\mid \iota_1<\ldots<\iota_\ell\right\}.$$
\item\label{lemma_qSH_dec_part2} $U(j,\ell)\cdot U(\ell,p)\subseteq U(j,p)$ where the product is taken entry-wise. 
\item\label{lemma_qSH_dec_part3}  
\begin{align*}
\QSH(m,s,u,k)&:=
\left\{\begin{bmatrix}
\e{\iota_1}&\cdots&\e{\iota_m}
&\e{\kappa_1}&\cdots&\e{\kappa_s}
&\e{\mu_1}&\cdots&\e{\mu_u}
\end{bmatrix}
\begin{array}{|l}
\iota_{1}<\cdots<\iota_{m}\\
\kappa_{1}<\cdots<\kappa_{s}\\
\mu_{1}<\cdots<\mu_{u}\\
\text{right invertible}
\end{array}\right\}\\
&=
\left\{\begin{bmatrix}\mathbf{F}_1&\mathbf{F}_2\mathbf{A}_1&
\mathbf{F}_2\mathbf{A}_2\end{bmatrix}\mid \mathbf{A}\in\QSH(s,u;j),\mathbf{F}\in\QSH(m,j;k)\right\}\\
&=
\left\{\begin{bmatrix}\mathbf{C}_1\mathbf{P}_1&\mathbf{C}_1\mathbf{P}_2&
\mathbf{C}_2\end{bmatrix}\mid
\mathbf{P}\in\QSH(m,s;j),\mathbf{C}\in\QSH(j,u;k)\right\}\\
&\subseteq\N_0^{k\times(m+s+u)}. 
\end{align*}
\end{enumerate}
\end{lemma}
The proof follows from the classical,  one-parameter setting.

\begin{corollary}\label{lemma_qsh_ass}
The two-parameter quasi-shuffle $\qShuffle$ is commutative and associative.
\end{corollary}
\begin{proof} Let  $\mathbf{a}\in\monoidComp^{m\times n}$ and 
$\mathbf{b}\in\monoidComp^{s\times t}$ be compositions. 
Commutativity follows from 
$$\mathbf{a}\qShuffle\mathbf{b}=
\sum_{{[\mathbf{P}_2\;\mathbf{P}_1]}}
\sum_{{[\mathbf{Q}_2\;\mathbf{Q}_1]}}
\mathbf{P}_2\mathbf{b}\mathbf{Q}^\top_2
\star
\mathbf{P}_1\mathbf{a}\mathbf{Q}^\top_1=\mathbf{b}\qShuffle\mathbf{a}.$$
For composition  
$\mathbf{c}\in\monoidComp^{u\times v}$ follows 
 \begin{align*}
  \mathbf{a}\qShuffle(\mathbf{b}\qShuffle\mathbf{c})
  &=
  \mathbf{a}\qShuffle\left(\sum_{\substack{\mathbf{A}\in\QSH(s,u;j)\\
  \mathbf{B}\in\QSH(t,v;x)}}\begin{bmatrix}\mathbf{A}_1&\mathbf{A}_2\end{bmatrix}
  \,\diag(\mathbf{b},\mathbf{c})\,\begin{bmatrix}\mathbf{B}_1^\top\\\mathbf{B}_2^\top\end{bmatrix}\right)\\
  &=
 \sum_{\substack{\mathbf{A}\\ \mathbf{B}}}\;
 \sum_{\substack{\mathbf{F}\in\QSH(m,j;k)\\ \mathbf{G}\in\QSH(n,x;y)}}
 \begin{bmatrix}\mathbf{F}_1&\mathbf{F}_2\end{bmatrix}
 \diag(\mathbf{a},\begin{bmatrix}\mathbf{A}_1&\mathbf{A}_2\end{bmatrix}
  \,\diag(\mathbf{b},\mathbf{c})\,
  \begin{bmatrix}\mathbf{B}_1^\top\\\mathbf{B}_2^\top\end{bmatrix})\,\begin{bmatrix}\mathbf{G}_1^\top\\\mathbf{G}_2^\top\end{bmatrix}\\
  &=
 \sum_{\substack{\mathbf{A}\\ \mathbf{B}}}\;
 \sum_{\substack{\mathbf{F}\\ \mathbf{G}}}\;
 \begin{bmatrix}\mathbf{F}_1&\mathbf{F}_2\end{bmatrix}\,
 \begin{bmatrix}\mathrm{I}_m&0&0\\
 0&\mathbf{A}_1&\mathbf{A}_2\end{bmatrix}
 \diag(\mathbf{a},\mathbf{b},\mathbf{c})
 \begin{bmatrix}
 \mathrm{I}_m&0\\
 0&\mathbf{B}_1^\top\\
 0&\mathbf{B}_2^\top
 \end{bmatrix}\,
 \begin{bmatrix}\mathbf{G}_1^\top\\\mathbf{G}_2^\top\end{bmatrix}\\
  &=
 \sum_{\substack{\mathbf{A}\\ \mathbf{B}}}\;
 \sum_{\substack{\mathbf{F}\\ \mathbf{G}}}\;
 \begin{bmatrix}\mathbf{F}_1&\mathbf{F}_2\mathbf{A}_1&\mathbf{F}_2\mathbf{A}_2\end{bmatrix}
 \diag(\mathbf{a},\mathbf{b},\mathbf{c})
 \begin{bmatrix}\mathbf{G}_1^\top\\{(\mathbf{G}_2\mathbf{B}_1)}^\top\\
 {(\mathbf{G}_2\mathbf{B}_2)}^\top\end{bmatrix}\\
  &=
 \sum_{\substack{\mathbf{C}\in\QSH(j,u;k)\\ \mathbf{D}\in\QSH(x,v;y)}}
 \sum_{\substack{\mathbf{P}\in\QSH(m,s;j)\\ \mathbf{Q}\in\QSH(n,t;x)}}
 \begin{bmatrix}\mathbf{C}_1\mathbf{P}_1&\mathbf{C}_1\mathbf{P}_2&\mathbf{C}_2\end{bmatrix}
 \diag(\mathbf{a},\mathbf{b},\mathbf{c})
 \begin{bmatrix}{(\mathbf{Q}_1\mathbf{D}_1)}^\top\\{(\mathbf{Q}_2\mathbf{D}_2)}^\top\\
 {\mathbf{Q}_2}^\top\end{bmatrix}\\
  &=
 \left(
 \sum_{\substack{\mathbf{P}\\\mathbf{Q}}}
 \begin{bmatrix}\mathbf{P}_1&\mathbf{P}_2
 \end{bmatrix}\,
 \diag(\mathbf{a},\mathbf{b})
 \begin{bmatrix}\mathbf{Q}_1^\top\\\mathbf{Q}_2^\top\end{bmatrix}\right)\qShuffle\mathbf{c}
  =(\mathbf{a}\qShuffle\mathbf{b})\qShuffle \mathbf{c}.
  \end{align*}
\end{proof}

With the following statement we verify that the algorithm in \Cref{sec:algo2dim} is sound.

\begin{lemma}\label{theo_char_qs}
Let $\mathbf{a}\in\monoidComp^{m\times n}$ and $\mathbf{b}\in\monoidComp^{s\times t}$ be compositions. 
\begin{enumerate}
\item\label{theo_char_qs_part1} 
With the blocks arising in the following quasi-shuffle of columns 
$$\begin{bmatrix}\mathbf{a}\\\monCompNeutrElem_{s\times n}\end{bmatrix}\qShuffle_2
\begin{bmatrix}\monCompNeutrElem_{m\times t}\\\mathbf{b}\end{bmatrix}=
\sum_{{[\mathbf{Q}_1\;\mathbf{Q}_2]}\in\QSH(n,t;k)}\; \begin{bmatrix}\mathbf{a}\mathbf{Q}_1^\top
\\
\mathbf{b}\mathbf{Q}_2^\top\end{bmatrix} \in \groundRing\langle\monoidComp^{m+s}\rangle,$$
we have
$$\mathbf{a}\qShuffle\mathbf{b}=
\sum_{{[\mathbf{Q}_1\;\mathbf{Q}_2]}} \mathbf{a}\mathbf{Q}_1^\top\qShuffle_1 \mathbf{b}\mathbf{Q}_2^\top.$$

\item\label{theo_char_qs_part2} 
With the blocks arising in the following quasi-shuffle of rows
$$\begin{bmatrix}\mathbf{a}&\vareps _{m\times t}\end{bmatrix}\qShuffle_1
\begin{bmatrix}\vareps _{s\times n}&\mathbf{b}\end{bmatrix}=
\sum_{{[\mathbf{P}_1\;\mathbf{P}_2]}\in\QSH(m,s;j)}\; \begin{bmatrix}\mathbf{P}_1\mathbf{a}
&
\mathbf{P}_2\mathbf{b}\end{bmatrix} \in \groundRing\langle\monoidComp^{1\times(n+t)}\rangle,$$
we have
$$\mathbf{a}\qShuffle\mathbf{b}=
\sum_{{[\mathbf{P}_1\;\mathbf{P}_2]}} \mathbf{P}_1\mathbf{a}
\qShuffle_2
\mathbf{P}_2\mathbf{b}.$$
\end{enumerate}
\end{lemma}
\begin{proof}
$$\mathbf{a}\qShuffle\mathbf{b}=
\sum_{{[\mathbf{P}_1\;\mathbf{P}_2]}}
\left(\sum_{\mathbf{Q}}
\begin{bmatrix}\mathbf{P}_1\mathbf{a}&
\mathbf{P}_2\mathbf{b}\end{bmatrix}\mathbf{Q}^\top\right)
=\sum_{{[\mathbf{Q}_1\;\mathbf{Q}_2]}}
\left(\sum_{\mathbf{P}}
\mathbf{P}\,\begin{bmatrix}\mathbf{a}\mathbf{Q}_1^\top\\\mathbf{b}\mathbf{Q}_2^\top\end{bmatrix}\right)
$$
\end{proof}
\begin{corollary}~\label{Cor_sqsProp}
The two-parameter quasi-shuffle is closed, that is for all matrix compositions   $\mathbf{a},\mathbf{b}\in\composition$, 
$\mathbf{a}\qShuffle\mathbf{b}\in\groundRing\langle\compositionConnected\rangle$. 
\end{corollary}

\begin{proof}
In every summand 
$$\begin{bmatrix}\mathbf{P}_1\mathbf{a}&\mathbf{P}_2\mathbf{b}\end{bmatrix}\text{ from }
\begin{bmatrix}\mathbf{a}&\monCompNeutrElem_{m\times t}\end{bmatrix}\qShuffle_2
\begin{bmatrix}\monCompNeutrElem_{s\times n}&\mathbf{b}\end{bmatrix}$$
there does not appear an $\monCompNeutrElem$-column, and therefore 
not in $\mathbf{P}_1\mathbf{a}$ or $\mathbf{P}_2\mathbf{b}$, and thus not in $\mathbf{P}_1\mathbf{a}\qShuffle_2\mathbf{P}_2\mathbf{b}$ from $\mathbf{a}\qShuffle\mathbf{b}$. 
An analogous argument yields with 
part \ref{theo_char_qs_part1}. of \Cref{theo_char_qs} that $\mathbf{a}\qShuffle\mathbf{b}$ contains no $\monCompNeutrElem$-rows. 
\end{proof}

\begin{corollary}\label{CorgradedProd}
The two-parameter quasi-shuffle is graded, that is  for all $\mathbf{a},\mathbf{b}\in\composition$ 
with $i=\weight(\diag(\mathbf{a},\mathbf{b}))$,
$$\mathbf{a}\qShuffle\mathbf{b}\in\bigoplus_{\substack{\mathbf{c}\in\composition\\\weight(\mathbf{c})=i}}
  \!\!\!\groundRing\mathbf{c}.$$
\end{corollary}
\begin{proof}
This follows with \Cref{lem:relPandp,lemma:idweightDiag}.
\end{proof}

\subsection{Bialgebra structure of matrix compositions}

We show \Cref{theoroem:bialgebra}, which leads to   \Cref{theoroem:hopfalgebra}.
The following case study  illustrates the central argument in the special case, where all compositions are connected.

\begin{lemma}\label{lemma:PdiagABQ}
For all $\mathbf{a},\mathbf{b}\in\compositionConnected$,  $\mathbf{P}
\in\QSH(\row(\mathbf{a}),\row(\mathbf{b});j)$ and 
$\mathbf{Q}
\in\QSH(\col(\mathbf{a}),\col(\mathbf{b});k)$, 
exactly one of the following holds:
\begin{enumerate}

\item 
$\mathbf{P}\diag(\mathbf{a},\mathbf{b})\mathbf{Q}^\top$ is connected, and hence
$$\Delta(\mathbf{P}\diag(\mathbf{a},\mathbf{b})\mathbf{Q}^\top)
=
\mathbf{P}\diag(\mathbf{a},\mathbf{b})\mathbf{Q}^\top\otimes \ec
+ \ec\otimes\mathbf{P}\diag(\mathbf{a},\mathbf{b})\mathbf{Q}^\top.$$
\item\label{lemma:PdiagABQ_part1}
$\mathbf{P}\diag(\mathbf{a},\mathbf{b})\mathbf{Q}^\top$ 
is not connected,
$j=\row(\mathbf{a})+\row(\mathbf{b})$,  $k=\col(\mathbf{a})+\col(\mathbf{b})$,  $\mathbf{P}=\mathrm{I}_{j}$, $\mathbf{Q}=\mathrm{I}_{k}$, and
$$\Delta(\mathbf{P}\diag(\mathbf{a},\mathbf{b})\mathbf{Q}^\top)
=
\diag(\mathbf{a},\mathbf{b})\otimes \ec 
+ \mathbf{a}\otimes\mathbf{b} 
+ \ec\otimes\diag(\mathbf{a},\mathbf{b}).$$

\item\label{lemma:PdiagABQ_part2}
$\mathbf{P}\diag(\mathbf{a},\mathbf{b})\mathbf{Q}^\top$ 
is not connected,
$\mathbf{P}=\begin{bmatrix}0&\mathrm{I}_{\row(\mathbf{b})}\\\mathrm{I}_{\row(\mathbf{a})}&0\end{bmatrix}$,
$\mathbf{Q}=\begin{bmatrix}0&\mathrm{I}_{\col(\mathbf{a})}\\\mathrm{I}_{\col(\mathbf{b})}&0\end{bmatrix}$, and  
$$\Delta(\mathbf{P}\diag(\mathbf{a},\mathbf{b})\mathbf{Q}^\top)
=
\diag(\mathbf{b},\mathbf{a})\otimes \ec 
+ \mathbf{b}\otimes\mathbf{a} 
+ \ec\otimes\diag(\mathbf{b},\mathbf{a}).$$

\end{enumerate}
\end{lemma}
\begin{example}\label{ex:bialgebraRelConnectedCase}In the setting of  \Cref{ex:algo2dimQSh} with   
$$(\mathbf{a},\mathbf{b})=(\begin{bmatrix} \w{1} \end{bmatrix},\begin{bmatrix}\w{2}\\ \w{3}\end{bmatrix})\in\monoidComp_3^{1\times 1}\times\monoidComp_3^{1\times 2},$$ 
the only summands in $\mathbf{a}\qShuffle\mathbf{b}$ which are not connected are $\diag(\mathbf{a},\mathbf{b})$ and $$\diag(\mathbf{b},\mathbf{a})
=\begin{bmatrix}
e_{3}&e_1&e_2
\end{bmatrix}
\diag(\mathbf{a},\mathbf{b})
{\begin{bmatrix}
e_2&e_1
\end{bmatrix}}^\top.$$
\end{example}
\begin{proof}[Proof of \Cref{lemma:PdiagABQ}]
Let $(m,n)=\size(\mathbf{a})$ and $(s,t)=\size(\mathbf{b})$. 
Let $1\le \mu_1<\cdots< \mu_m \le j$ and $1 \le \nu_1< \cdots< \nu_s \le j$ be such that
$$\mathbf{P}
=\begin{bmatrix}
\mathbf{P}_1&\mathbf{P}_2
\end{bmatrix}
=\begin{bmatrix}\e{\mu_1}&\cdots&\e{\mu_m}&\e{\nu_1}&\cdots&\e{\nu_s}\end{bmatrix}\in\N_0^{j\times (m+s)}
$$ 
according to \Cref{lemma_qSH_dec}, and similarly $\mathbf{Q}
=\begin{bmatrix}\mathbf{Q}_1&\mathbf{Q}_2\end{bmatrix}\in\N_0^{k\times (n+t)}
$ with increasing chains $\iota$ and $\kappa$.   
For $i\geq 2$ and $\omega\in\{\mu,\iota\}$ let $\overline\omega_i:=\omega_i-\omega_{i-1}-1$. 
In the composition
\begin{equation}\label{eq:proofBialgebraPrimitiveCase}
\mathbf{P}\diag(\mathbf{a},\mathbf{b})\mathbf{Q}^\top=\mathbf{P}_1\mathbf{a}\mathbf{Q}^\top_1\star\mathbf{P}_2\mathbf{b}\mathbf{Q}^\top_2\in\monoidComp^{j\times k},\end{equation}
consider the factor derived from the connected matrix composition $\mathbf{a}$, 
\begin{align*}
\mathbf{P}_1\mathbf{a}\mathbf{Q}_1^\top&=
\begin{bmatrix}
\monCompNeutrElem_{(\mu_1-1)\times(\iota_1-1)}
&\monCompNeutrElem_{(\mu_1-1)\times1}&
\monCompNeutrElem_{(\mu_1-1)\times\overline\iota_2}&
\cdots&
\monCompNeutrElem_{(\mu_1-1)\times1}&
\monCompNeutrElem_{(\mu_1-1)\times(k-\iota_s)}\\
\monCompNeutrElem_{1\times(\iota_1-1)}
&\mathbf{a}_{1,1}&
\monCompNeutrElem_{1\times\overline\iota_2}&\cdots&\mathbf{a}_{1,s}&
\monCompNeutrElem_{1\times(k-\iota_s)}
\\
\monCompNeutrElem_{\overline\mu_2\times(\iota_1-1)}
&\monCompNeutrElem_{\overline\mu_2\times1}
&\monCompNeutrElem_{\overline\mu_2\times\overline\iota_2}
&\cdots&
\monCompNeutrElem_{\overline\mu_2\times1}&
\monCompNeutrElem_{\overline\mu_2\times(k-\iota_s)}\\
\monCompNeutrElem_{1\times(\iota_1-1)}
&\mathbf{a}_{2,1}&
\monCompNeutrElem_{1\times\overline\iota_2}&\cdots&\mathbf{a}_{2,s}&
\monCompNeutrElem_{1\times(k-\iota_s)}
\\
\monCompNeutrElem_{\overline\mu_3\times(\iota_1-1)}
&\monCompNeutrElem_{\overline\mu_3\times1}
&\monCompNeutrElem_{\overline\mu_3\times\overline\iota_2}
&\cdots&
\monCompNeutrElem_{\overline\mu_3\times1}&
\monCompNeutrElem_{\overline\mu_3\times(k-\iota_s)}\\
\vdots&\vdots&\vdots&\ddots&\vdots&\vdots\\
\monCompNeutrElem_{1\times(\iota_1-1)}
&\mathbf{a}_{m,1}
&\monCompNeutrElem_{1\times\overline\iota_2}
&\cdots
&\mathbf{a}_{m,s}&\monCompNeutrElem_{1\times(k-\iota_s)}\\
\monCompNeutrElem_{(j-\mu_m)\times(\iota_1-1)}
&\monCompNeutrElem_{(j-\mu_m)\times1}&
\monCompNeutrElem_{(j-\mu_m)\times\overline\iota_2}
&\cdots
&\monCompNeutrElem_{(j-\mu_m)\times1}&\monCompNeutrElem_{(j-\mu_m)\times(k-\iota_s)}
\end{bmatrix}.
\end{align*}
It satisfies the following property:
if 
\begin{equation}\label{eq:factorA_nonPrimitivePDiagABQT}
\mathbf{P}_1\mathbf{a}\mathbf{Q}_1^\top=\diag(\mathbf{x}_1,\mathbf{y}_1)\end{equation} 
with non-empty $\mathbf{x}_1$ and $\mathbf{y}_1$, then $\mathbf{x}_1$ has all entries equal
to $\monCompNeutrElem$ or $\mathbf{y}_1$ does.

The same argument holds for the connected matrix composition $\mathbf{b}$ and 
\begin{equation}\label{eq:factorB_nonPrimitivePDiagABQT}
\mathbf{P}_2\mathbf{b}\mathbf{Q}_2^\top=\diag(\mathbf{x}_2,\mathbf{y}_2).\end{equation}
If the composition $\mathbf{P}\diag(\mathbf{a},\mathbf{b})\mathbf{Q}^\top=\diag(\mathbf{x},\mathbf{y})$ is not connected with non-empty compositions $\mathbf{x}$ and $\mathbf{y}$, then the factors from \Cref{eq:proofBialgebraPrimitiveCase} must have a block factorization, and thus 
$$\mathbf{P}\diag(\mathbf{a},\mathbf{b})\mathbf{Q}^\top=\diag(\mathbf{x}_1\star\mathbf{x}_2,\mathbf{y}_1\star\mathbf{y}_2)$$
with 
(\ref{eq:factorA_nonPrimitivePDiagABQT}), (\ref{eq:factorB_nonPrimitivePDiagABQT}) and where $\size(\mathbf{x})=\size(\mathbf{x}_1)=\size(\mathbf{x}_2)$. 
If $\mathbf{y}_1$ is $\monCompNeutrElem$-valued, then $\mathbf{x}_1=\mathbf{a}$,
$\mathbf{x}_2$ is  $\monCompNeutrElem$-valued, thus
$\mathbf{y}_2=\mathbf{b}$, and therefore we are in case 
\ref{lemma:PdiagABQ_part1}. 
If instead $\mathbf{x}_1$ is  $\monCompNeutrElem$-valued, then
$\mathbf{y}_1=\mathbf{a}$,
$\mathbf{y}_2$ is   $\monCompNeutrElem$-valued,  $\mathbf{x}_2=\mathbf{b}$, and 
hence we obtain case 
\ref{lemma:PdiagABQ_part2}. 
\end{proof}
Let $\tau:\groundRing\langle\compositionConnected\rangle\otimes\groundRing\langle\compositionConnected\rangle\rightarrow\groundRing\langle\compositionConnected\rangle\otimes\groundRing\langle\compositionConnected\rangle$ denote the \DEF{flip} homomorphism, uniquely determined by $\tau(\mathbf{a}\otimes\mathbf{b}):=\mathbf{b}\otimes\mathbf{a}$ for all  $\mathbf{a},\mathbf{b}\in\composition$. 

\begin{corollary}
For $\mathbf{a},\mathbf{b}\in\compositionConnected$, 
$$(\qShuffle\otimes\qShuffle)\circ(\id\otimes\,\tau\otimes\id)\circ(\Delta\otimes\Delta)(\mathbf{a}\otimes\mathbf{b})=\Delta\,\circ\qShuffle(\mathbf{a}\otimes\mathbf{b}).$$
\end{corollary}
\begin{proof}
With \Cref{lemma:PdiagABQ}, 
\begin{align*}
\Delta\,\circ\qShuffle(\mathbf{a}\otimes\mathbf{b})
&=\sum_{\mathbf{P},\mathbf{Q}}\Delta(\mathbf{P}\diag(\mathbf{a},\mathbf{b})\mathbf{Q}^\top)\\
&=\mathbf{a}\otimes\mathbf{b}+\mathbf{b}\otimes\mathbf{a}+
\sum_{\mathbf{P},\mathbf{Q}}\left(\mathbf{P}\diag(\mathbf{a},\mathbf{b})\mathbf{Q}^\top\otimes \ec+ \ec\otimes\mathbf{P}\diag(\mathbf{a},\mathbf{b})\mathbf{Q}^\top\right)\\
&=(\qShuffle\otimes\qShuffle)(
\mathbf{a}\otimes \ec\otimes \ec\otimes\mathbf{b} + 
 \ec\otimes\mathbf{b}\otimes\mathbf{a}\otimes \ec + 
 \mathbf{a}\otimes\mathbf{b}\otimes \ec\otimes \ec+
 \ec\otimes \ec\otimes\mathbf{a}\otimes\mathbf{b})%
\end{align*}
verifies the bialgebra relation in the case where $\mathbf{a}$ and $\mathbf{b}$ are connected. 
\end{proof}

We now generalize  \Cref{lemma:PdiagABQ} for arbitrary matrix compositions, visualized in \Cref{ex:bialgebraRelationPossibleSplittings} and in \Cref{fig:casestudy_primitive_decomp_qshuffle} with decomposition length $a=2$ and $b=1$.
For this, let $$\underline u_{\alpha}:=\sum_{1\leq r\leq \alpha}u_r$$
denote the cumulative sum\footnote{The empty sum yields $\underline u_{0}:=0$.} of $u\in\N^a$ at index $0\leq \alpha \leq a$.  
\index[general]{splittings0@$\domainAlphaBeta$}
For fixed 
$\mathbf{a},\mathbf{b}\in \composition$ with decompositions into connected compositions $\mathbf{a}=\diag(\mathbf{v}_1,\ldots,\mathbf{v}_a)$ and $\mathbf{b}=\diag(\mathbf{w}_1,\ldots,\mathbf{w}_b)$ let 
$$\domainAlphaBeta_{a,b}:=\{(\alpha,\beta)\in\N_0^2\mid 0_2\leq(\alpha,\beta)\leq(a,b)\land0\not=\alpha+\beta\not=a+b\}.$$
For every $(\alpha,\beta)\in\domainAlphaBeta_{a,b}$ we call 
 $$\mathbf{P}
 =\begin{bmatrix}\mathbf{P}_1&\mathbf{P}_2
 \end{bmatrix}
=\begin{bmatrix}\e{\mu_1}&\cdots&\e{\mu_{\row(\mathbf{a})}}&\e{\nu_1}&\cdots&\e{\nu_{\row(\mathbf{b})}}\end{bmatrix}\in\QSH(\row(\mathbf{a}),\row(\mathbf{b}),j)$$
\DEF{$(\alpha,\beta)$-decomposable}, if%
\footnote{Note that the notation does not distinguish between standard columns $e_i$ from $\N_0^{j_1}$ and $\N_0^{j_2}$.} 
\footnote{
    In the cases $\alpha\in\{0,a\}$ we have   $\mathbf{P}_{11}^{(0)}=\begin{bmatrix}0_{j_1}&\cdots&0_{j_1}\end{bmatrix},\mathbf{P}_{11}^{(a)}=\begin{bmatrix}
\e{\mu_{1}}&\cdots&\e{\mu_{\row(\mathbf{a})}}\end{bmatrix}\in\N_0^{j_1\times\row(\mathbf{a})}$  and  $\mathbf{P}_{21}^{(0)}=\begin{bmatrix}
\e{(\mu_{1}-j_1)}&\cdots&\e{(\mu_{\row(\mathbf{a})}-j_1)}\end{bmatrix},\mathbf{P}_{21}^{(a)}=\begin{bmatrix}0_{j_2}&\cdots&0_{j_2}\end{bmatrix}\in\N_0^{j_2\times\row(\mathbf{a})}$.} 
\begin{align}\label{eq1:lem518}
\mathbf{P}_1&=
\begin{bmatrix}\e{\mu_{1}}&\cdots&\e{\mu_{\underline{u}_{\alpha}}}&0_{j_1}&\cdots&0_{j_1}\\
0_{j_2}&\cdots&0_{j_2
}&
\e{(\mu_{(\underline{u}_{\alpha}+1)}-j_1)}&\cdots&\e{(\mu_{\row(\mathbf{a})}-j_1)}
\end{bmatrix}=:\begin{bmatrix}\mathbf{P}_{11}^{(\alpha)}\\
\mathbf{P}_{21}^{(\alpha)}\end{bmatrix}
\end{align}
and 
\begin{align}\label{eq2:lem518}
\mathbf{P}_2&=
\begin{bmatrix}\e{\nu_{1}}&\cdots&\e{\nu_{\underline{\omega}_{\beta}}}&0_{j_1}&\cdots&0_{j_1}\\
0_{j_2}&\cdots&0_{j_2}&
\e{(\nu_{(\underline{\omega}_{\beta}+1)}-j_1)}&\cdots&\e{(\nu_{\row(\mathbf{b})}-j_1)}
\end{bmatrix}=:\begin{bmatrix}
\mathbf{P}_{12}^{(\beta)}\\
\mathbf{P}_{22}^{(\beta)}
\end{bmatrix}
\end{align}
with suitable $j_1,j_2\in\N$, where  $(u_\alpha, v_\alpha)=\size(\mathbf{v}_\alpha)$ and $(\omega_\beta, w_\beta)=\size(\mathbf{w}_\beta)$ denote the block sizes in $\mathbf{a}$ and  $\mathbf{b}$, respectively. 
\begin{lemma}\label{lem:uniquenessAndInje}
If $\mathbf{P}$ is $(\alpha,\beta)$-decomposable, then its block decomposition
\begin{equation}\label{eq:blockOfP}
\mathbf{P}=\begin{bmatrix}\mathbf{P}_{11}^{(\alpha)}&\mathbf{P}_{12}^{(\beta)}\\\mathbf{P}_{21}^{(\alpha)}&\mathbf{P}_{22}^{(\beta)}\end{bmatrix}
\end{equation}
is uniquely determined. 
Furthermore, if $\mathbf{P}$ is also $(\alpha',\beta')$-decomposable with 
$\size(\mathbf{P}_{11}^{(\alpha)})=\size(\mathbf{P}_{11}^{(\alpha')})$, then $(\alpha,\beta)=(\alpha',\beta')\in\domainAlphaBeta_{a,b}$. 
\end{lemma}
\begin{proof}
Due to \Cref{lemma_qSH_dec},  $\mathbf{P}$ decomposes uniquely into $\mathbf{P}_1$ and $\mathbf{P}_2$.
Since $\mathbf{P}$ is right invertible, it contains no zero rows.
Therefore, with (\ref{eq1:lem518}) and (\ref{eq2:lem518}), $j_1$ and $j_2=j-j_1$ are unique. 
The second part follows since $\mu$ and $\nu$ are strictly increasing. 
\end{proof}

\begin{lemma}\label{lem:PdiagAbQChar}
Let $\mathbf{a}=\diag(\mathbf{v}_1,\ldots,\mathbf{v}_a)$ and $\mathbf{b}=\diag(\mathbf{w}_1,\ldots,\mathbf{w}_b)$ be decompositions into connected compositions. 
\index[general]{splittings1@$\splitt$}
For fixed $$(\mathbf{P},\mathbf{Q})\in\QSH(\row(\mathbf{a}),\row(\mathbf{b});j)\times\QSH(\col(\mathbf{a}),\col(\mathbf{b});k),$$
the set
\begin{equation*}\label{lem:bialg_Set1}
\{(\mathbf{x},\mathbf{y})\in\composition^2\mid\mathbf{x}\not=\ec\not=\mathbf{y}\land\diag(\mathbf{x},\mathbf{y})=\mathbf{P}\diag(\mathbf{a},\mathbf{b})\mathbf{Q}^\top\}
\end{equation*}
is in one-to-one correspondence to the set of \DEF{splittings}
\begin{equation*}\label{lem:bialg_Set2}
{\splitt}_{\mathbf{P},\mathbf{Q}}^{\mathbf{a},\mathbf{b}}:=\{(\alpha,\beta)\in\domainAlphaBeta_{a,b}\mid
\mathbf{P}=\begin{bmatrix}\mathbf{P}_{11}^{(\alpha)}&\mathbf{P}_{12}^{(\beta)}\\\mathbf{P}_{21}^{(\alpha)}&\mathbf{P}_{22}^{(\beta)}\end{bmatrix}
\land
\mathbf{Q}=\begin{bmatrix}
\mathbf{Q}_{11}^{(\alpha)}
&\mathbf{Q}_{12}^{(\beta)}\\
\mathbf{Q}_{21}^{(\alpha)}
&\mathbf{Q}_{22}^{(\beta)}
\end{bmatrix}
\},\end{equation*}
where $\mathbf{P}$ and $\mathbf{Q}$ are simultaneously $(\alpha,\beta)$-decomposable with  (\ref{eq1:lem518}), (\ref{eq2:lem518}), \begin{align}\label{eq3:lem518}
    \begin{bmatrix}\mathbf{Q}_{11}^{(\alpha)}\\\mathbf{Q}_{21}^{(\alpha)}\end{bmatrix}
&=
\begin{bmatrix}\e{\iota_{1}}&\cdots&\e{\iota_{\underline{v}_{\alpha}}}&0_{k_1}&\cdots&0_{k_1}\\
0_{k_2}&\cdots&0_{k_2
}&
\e{(\iota_{(\underline{v}_{\alpha}+1)}-k_1)}&\cdots&\e{(\iota_{\col(\mathbf{a})}-k_1)}\end{bmatrix}
\end{align}
and 
\begin{align}
\label{eq4:lem518}
\begin{bmatrix}\mathbf{Q}_{12}^{(\beta)}\\\mathbf{Q}_{22}^{(\beta)}\end{bmatrix}&=
\begin{bmatrix}\e{\kappa_{1}}&\cdots&\e{\kappa_{\underline{w}_{\beta}}}&0_{k_1}&\cdots&0_{k_1}\\
0_{k_2}&\cdots&0_{k_2}&
\e{(\kappa_{(\underline{w}_{\beta}+1)}-k_1)}&\cdots&\e{(\kappa_{\col(\mathbf{b})}-k_1)}\end{bmatrix}
\end{align} 
for suitable $k_1,k_2\in\N$ and increasing 
$(\iota,\kappa)\in\N^{\col(\mathbf{a})}\times\N^{\col(\mathbf{b})}$.  
\end{lemma}
\begin{example}\label{ex:bialgebraRelationPossibleSplittings}
~
\begin{enumerate}
    \item
    In  \Cref{ex:bialgebraRelConnectedCase} we have $\mathbf{a},\mathbf{b}\in\compositionConnected$, i.e.,  both $\mathbf{a}$ and $\mathbf{b}$ have a trivial decomposition of length $a=b=1$. If $\mathbf{P}=\begin{bmatrix}e_3&e_1&e_2\end{bmatrix}$ and $\mathbf{Q}=\begin{bmatrix}e_2&e_1\end{bmatrix}$, then   $\splitt_{\mathbf{P},\mathbf{Q}}^{\mathbf{a},\mathbf{b}}=\{(0,1)\}$
     where 
$$(\begin{bmatrix}\mathbf{P}_{11}^{(\alpha)}&\mathbf{P}_{12}^{(\beta)}\\\mathbf{P}_{21}^{(\alpha)}&\mathbf{P}_{22}^{(\beta)}\end{bmatrix},
\begin{bmatrix}\mathbf{Q}_{11}^{(\alpha)}&\mathbf{Q}_{12}^{(\beta)}\\\mathbf{Q}_{21}^{(\alpha)}&\mathbf{Q}_{22}^{(\beta)}\end{bmatrix})=
(\begin{bmatrix}
0_{j_1\times 1}&\mathrm{I}_2\\
\mathrm{I}_1&0_{j_2\times 2}\end{bmatrix},
\begin{bmatrix}
0_{k_1\times 1}&\mathrm{I}_1\\
\mathrm{I}_1&0_{k_2\times 1}
\end{bmatrix})$$
with $(j_1,j_2)=(2,1)$ and $(k_1,k_2)=(1,1)$, that is case \ref{lemma:PdiagABQ_part2}. of \Cref{lemma:PdiagABQ}.
The only other candidate $(\alpha,\beta)=(1,0)\in\domainAlphaBeta_{1,1}$ is not a splitting, since ${\underline{u}}_{1}=2$ and $\mu_{1}=3$, i.e.
(\ref{eq1:lem518}) is not satisfied.  
\item 
For a case with $\mathbf{a}\not\in\compositionConnected$, consider
\begin{equation}\label{eq:exSplittings}
(\mathbf{a},\mathbf{b})
=(\begin{bmatrix}
\w{1}&\w{2}&\varepsilon&\varepsilon\\
\varepsilon&\w{3}&\varepsilon&\varepsilon\\
\varepsilon&\varepsilon&\w{4}&\w{5}
\end{bmatrix},
\begin{bmatrix}
\varepsilon&\w{6}\\
\w{7}&\w{8}
\end{bmatrix})\in\monoidComp_8^{3\times 4}\times\monoidComp_8^{2\times 2},
\end{equation}
Then, with
$\mathbf{P}=\begin{bmatrix}e_1&e_2&e_5&e_3&e_4\end{bmatrix}$ and $\mathbf{Q}=\begin{bmatrix}e_1&e_2&e_5&e_6&e_3&e_4\end{bmatrix}$, 
$$\mathbf{P}\diag(\mathbf{a},\mathbf{b}){\mathbf{Q}}^\top=\begin{bmatrix}
\w{1}&\w{2}&\varepsilon&\varepsilon&\varepsilon&\varepsilon\\
\varepsilon&\w{3}&\varepsilon&\varepsilon&\varepsilon&\varepsilon\\
\varepsilon&\varepsilon&\varepsilon&\w{6}&\varepsilon&\varepsilon\\
\varepsilon&\varepsilon&\w{7}&\w{8}&\varepsilon&\varepsilon\\
\varepsilon&\varepsilon&\varepsilon&\varepsilon&\w{4}&\w{5}\end{bmatrix}$$
which has two splittings $\splitt_{\mathbf{P},\mathbf{Q}}^{\mathbf{a},\mathbf{b}}=\{(1,0),(1,1)\}$ via 
\begin{align*}
(\mathbf{P},\mathbf{Q})
&=(\begin{bmatrix}
\begin{bmatrix}e_1&e_2&0_2\end{bmatrix}
&\begin{bmatrix}0_2&0_2\end{bmatrix}\\
\begin{bmatrix}0_3&0_3&e_3\end{bmatrix}
&\begin{bmatrix}e_1&e_2\end{bmatrix}
\end{bmatrix},
\begin{bmatrix}
\begin{bmatrix}e_1&e_2&0_2&0_2\end{bmatrix}
&\begin{bmatrix}0_2&0_2\end{bmatrix}\\
\begin{bmatrix}0_4&0_4&e_3&e_4\end{bmatrix}
&\begin{bmatrix}e_1&e_2\end{bmatrix}
\end{bmatrix})\\
&=(\begin{bmatrix}
\begin{bmatrix}e_1&e_2&0_4\end{bmatrix}
&\begin{bmatrix}e_3&e_4\end{bmatrix}\\
\begin{bmatrix}0&0&1\end{bmatrix}
&\begin{bmatrix}0&0\end{bmatrix}
\end{bmatrix},
\begin{bmatrix}
\begin{bmatrix}e_1&e_2&0_4&0_4\end{bmatrix}
&\begin{bmatrix}e_3&e_4\end{bmatrix}\\
\begin{bmatrix}0_2&0_2&e_1&e_2\end{bmatrix}
&\begin{bmatrix}0_2&0_2\end{bmatrix}
\end{bmatrix}).
\end{align*}
\item With $(\mathbf{a},\mathbf{b})$ as in (\ref{eq:exSplittings}), $\mathbf{P}=\mathrm{I}_5$ and $\mathbf{Q}=\begin{bmatrix}
    e_1&e_2&e_3&e_4&e_2&e_5
\end{bmatrix}$,
$$\mathbf{P}
\diag(\mathbf{a},\mathbf{b})
{\mathbf{Q}}^\top=
\begin{bmatrix}
\w{1}&\w{2}&\varepsilon&\varepsilon&\varepsilon\\
\varepsilon&\w{3}&\varepsilon&\varepsilon&\varepsilon\\
\varepsilon&\varepsilon&\w{4}&\w{5}&\varepsilon\\
\varepsilon&\varepsilon&\varepsilon&\varepsilon&\w{6}\\
\varepsilon&\w{7}&\varepsilon&\varepsilon&\w{8}
\end{bmatrix}$$
is connected.     
The constellation $(0,1)\in\domainAlphaBeta_{2,1}$ is not a splitting with ${\mathbf{P}}_{1,1}=1$.
Also for the remaining cases $1\leq\alpha\leq2$, 
\begin{align*}\mathbf{Q}=
\begin{bmatrix}
\begin{bmatrix}e_1&e_2&0_2&0_2\end{bmatrix}
&\begin{bmatrix}e_2&0_2\end{bmatrix}\\
\begin{bmatrix}0_3&0_3&e_3&e_4\end{bmatrix}
&\begin{bmatrix}0_3&e_3\end{bmatrix}
\end{bmatrix}
=
\begin{bmatrix}
\begin{bmatrix}e_1&e_2&e_3&e_4\end{bmatrix}
&\begin{bmatrix}e_2&0_4\end{bmatrix}\\
\begin{bmatrix}0&0&0&0\end{bmatrix}
&\begin{bmatrix}0&1\end{bmatrix}
\end{bmatrix}
\end{align*}
is not $(\alpha,\beta)$-decomposable. This follows  
since with $\iota_{\underline{v}_1}=2$ and $\iota_{\underline{v}_2}=4$,   
$\mathbf{Q}_2$ is not of the form 
(\ref{eq4:lem518}). 
Therefore, $\splitt_{\mathbf{P},\mathbf{Q}}^{\mathbf{a},\mathbf{b}}$ is empty.
\end{enumerate}
\end{example}

\begin{figure}
     \centering
     \begin{subfigure}[b]{0.3\textwidth}
         \centering
         $\begin{bmatrix}
\w{1}&\w{2}&\varepsilon&\varepsilon&\varepsilon&\varepsilon\\
\varepsilon&\w{3}&\varepsilon&\varepsilon&\varepsilon&\varepsilon\\
\varepsilon&\varepsilon&\w{4}&\w{5}&\varepsilon&\varepsilon\\
\varepsilon&\varepsilon&\varepsilon&\varepsilon&\varepsilon&\w{6}\\
\varepsilon&\varepsilon&\varepsilon&\varepsilon&\w{7}&\w{8}\end{bmatrix}$
         \caption*{$\splitt_{\mathbf{P},\mathbf{Q}}^{\mathbf{a},\mathbf{b}}=\{(1,0),(2,0)\}$}
     \end{subfigure}
     \hfill
     \begin{subfigure}[b]{0.3\textwidth}
         \centering
         $\begin{bmatrix}
\w{1}&\w{2}&\varepsilon&\varepsilon&\varepsilon\\
\varepsilon&\w{3}&\varepsilon&\varepsilon&\varepsilon\\
\varepsilon&\varepsilon&\w{4}&\w{5}&\varepsilon\\
\varepsilon&\varepsilon&\varepsilon&\varepsilon&\w{6}\\
\varepsilon&\varepsilon&\varepsilon&\w{7}&\w{8}
\end{bmatrix}$
         \caption*{$\splitt_{\mathbf{P},\mathbf{Q}}^{\mathbf{a},\mathbf{b}}=\{(1,0)\}$}
     \end{subfigure}
     \hfill
     \begin{subfigure}[b]{0.3\textwidth}
         \centering
         $\begin{bmatrix}
\w{1}&\w{2}&\varepsilon&\varepsilon&\varepsilon\\
\varepsilon&\w{3}&\varepsilon&\varepsilon&\varepsilon\\
\varepsilon&\varepsilon&\w{4}&\w{5}&\varepsilon\\
\varepsilon&\varepsilon&\varepsilon&\varepsilon&\w{6}\\
\varepsilon&\w{7}&\varepsilon&\varepsilon&\w{8}
\end{bmatrix}$
         \caption*{$\splitt_{\mathbf{P},\mathbf{Q}}^{\mathbf{a},\mathbf{b}}=\emptyset$}
     \end{subfigure}
     \newline\\
     \begin{subfigure}[b]{0.3\textwidth}
         \centering
         $\begin{bmatrix}
\w{1}&\w{2}&\varepsilon&\varepsilon&\varepsilon&\varepsilon\\
\varepsilon&\w{3}&\varepsilon&\varepsilon&\varepsilon&\varepsilon\\
\varepsilon&\varepsilon&\varepsilon&\w{6}&\varepsilon&\varepsilon\\
\varepsilon&\varepsilon&\w{7}&\w{8}&\varepsilon&\varepsilon\\
\varepsilon&\varepsilon&\varepsilon&\varepsilon&\w{4}&\w{5}\end{bmatrix}$
         \caption*{$\splitt_{\mathbf{P},\mathbf{Q}}^{\mathbf{a},\mathbf{b}}=\{(1,0),(1,1)\}$}
     \end{subfigure}
     \hfill
     \begin{subfigure}[b]{0.3\textwidth}
         \centering
         $\begin{bmatrix}
\w{1}&\w{2}&\varepsilon&\varepsilon&\varepsilon\\
\varepsilon&\w{3}&\varepsilon&\varepsilon&\varepsilon\\
\varepsilon&\varepsilon&\w{6}&\varepsilon&\varepsilon\\
\varepsilon&\w{7}&\w{8}&\varepsilon&\varepsilon\\
\varepsilon&\varepsilon&\varepsilon&\w{4}&\w{5}\end{bmatrix}$
         \caption*{$\splitt_{\mathbf{P},\mathbf{Q}}^{\mathbf{a},\mathbf{b}}=\{(1,1)\}$}
     \end{subfigure}
     \hfill
     \begin{subfigure}[b]{0.3\textwidth}
         \centering
         $\begin{bmatrix}
\varepsilon&\w{6}&\varepsilon&\varepsilon&\varepsilon&\varepsilon\\
\w{7}&\w{8}&\varepsilon&\varepsilon&\varepsilon&\varepsilon\\
\varepsilon&\varepsilon&\w{1}&\w{2}&\varepsilon&\varepsilon\\
\varepsilon&\varepsilon&\varepsilon&\w{3}&\varepsilon&\varepsilon\\
\varepsilon&\varepsilon&\varepsilon&\varepsilon&\w{4}&\w{5}\end{bmatrix}$
         \caption*{$\splitt_{\mathbf{P},\mathbf{Q}}^{\mathbf{a},\mathbf{b}}=\{(0,1),(1,1)\}$}
     \end{subfigure}
        \caption{
        For some terms appearing
        in the quasi-shuffle of $\mathbf a$ and $\mathbf b$ from 
        \Cref{ex:bialgebraRelationPossibleSplittings}.2
        (which lead to right invertible matrices $\mathbf P$ and $\mathbf Q$),
        all splittings due to the Lemma are given.}
\label{fig:casestudy_primitive_decomp_qshuffle}
\end{figure}\textbf{}

\begin{proof}[Proof of \Cref{lem:bialg_Set2}]
If $(\alpha,\beta)\in \splitt_{\mathbf{P},\mathbf{Q}}^{\mathbf{a},\mathbf{b}}$, then $\mathbf{P}$ and $\mathbf{Q}$ are $(\alpha,\beta)$-decomposable with $\mathbf{P}_{21}^{(\alpha)}\mathbf{a}{\mathbf{Q}_{11}^{(\alpha)}}^\top=
\mathbf{P}_{22}^{(\beta)}\mathbf{b}{\mathbf{Q}_{12}^{(\beta)}}^\top=\monCompNeutrElem_{j_2\times k_1}
$ and $
\mathbf{P}_{11}^{(\alpha)}\mathbf{a}{\mathbf{Q}_{21}^{(\alpha)}}^\top
=
\mathbf{P}_{12}^{(\beta)}\mathbf{b}{\mathbf{Q}_{22}^{(\beta)}}^\top
=\monCompNeutrElem_{j_1\times k_2}$, hence
\begin{align}\label{lemEq:PdiagAbQChar}
\mathbf{P}\diag(\mathbf{a},\mathbf{b})\mathbf{Q}^\top&=
\begin{bmatrix}
\mathbf{P}_{11}^{(\alpha)}\mathbf{a}{\mathbf{Q}_{11}^{(\alpha)}}^\top
\star
\mathbf{P}_{12}^{(\beta)}\mathbf{b}{\mathbf{Q}_{12}^{(\beta)}}^\top
&
\mathbf{P}_{11}^{(\alpha)}\mathbf{a}{\mathbf{Q}_{21}^{(\alpha)}}^\top
\star
\mathbf{P}_{12}^{(\beta)}\mathbf{b}{\mathbf{Q}_{22}^{(\beta)}}^\top
\\
\mathbf{P}_{21}^{(\alpha)}\mathbf{a}{\mathbf{Q}_{11}^{(\alpha)}}^\top
\star
\mathbf{P}_{22}^{(\beta)}\mathbf{b}{\mathbf{Q}_{12}^{(\beta)}}^\top
&\mathbf{P}_{21}^{(\alpha)}\mathbf{a}{\mathbf{Q}_{21}^{(\alpha)}}^\top
\star
\mathbf{P}_{22}^{(\beta)}\mathbf{b}{\mathbf{Q}_{22}^{(\beta)}}^\top
\end{bmatrix}\\
&=\diag(
\mathbf{P}_{11}^{(\alpha)}\mathbf{a}{\mathbf{Q}_{11}^{(\alpha)}}^\top
\star
\mathbf{P}_{12}^{(\beta)}\mathbf{b}{\mathbf{Q}_{12}^{(\beta)}}^\top
,\mathbf{P}_{21}^{(\alpha)}\mathbf{a}{\mathbf{Q}_{21}^{(\alpha)}}^\top
\star
\mathbf{P}_{22}^{(\beta)}\mathbf{b}{\mathbf{Q}_{22}^{(\beta)}}^\top)\nonumber
\end{align}
is not connected. 
Different splittings in $\splitt_{\mathbf{P},\mathbf{Q}}^{\mathbf{a},\mathbf{b}}$ yield different block decompositions (\ref{lemEq:PdiagAbQChar}) due to \Cref{lem:uniquenessAndInje}. 

Conversely, if 
\begin{align*}\mathbf{P}\diag(\mathbf{a},\mathbf{b})\mathbf{Q}^\top
&=\underset{1\leq s\leq a}\bigstar
\begin{bmatrix}\e{\mu_{(\underline{u}_{(s-1)}+1)}}&\cdots&\e{\mu_{\underline{u}_{s}}}\end{bmatrix}\mathbf{v}_s
\begin{bmatrix}\e{\iota_{(\underline{v}_s+1)}}&\cdots&\e{\iota_{\underline{v}_{(s+1)}}}\end{bmatrix}^\top\\
&\;\;\;\;\;\;\;\;\;\;\;
\star
\underset{1\leq t\leq b}\bigstar
\begin{bmatrix}\e{\nu_{(\underline{\omega}_t+1)}}&\cdots&\e{\nu_{\underline{\omega}_{(t+1)}}}\end{bmatrix}\mathbf{w}_t
\begin{bmatrix}\e{\kappa_{(\underline{w}_t+1)}}&\cdots&\e{\kappa_{\underline{w}_{(t+1)}}}\end{bmatrix}^\top\\
&=\diag(\mathbf{x},\mathbf{y})
\end{align*}
with $\size(\mathbf{x})=(j_1,k_1)$, then analogous to \Cref{lemma:PdiagABQ}, all $s$-indexed factors 
$$\begin{bmatrix}\e{\mu_{(\underline{u}_{(s-1)}+1)}}&\cdots&\e{\mu_{\underline{u}_{s}}}\end{bmatrix}\mathbf{v}_s
\begin{bmatrix}\e{\iota_{(\underline{v}_s+1)}}&\cdots&\e{\iota_{\underline{v}_{(s+1)}}}\end{bmatrix}^\top
=:\diag(\mathbf{x}_1^{(s)},\mathbf{y}_1^{(s)})$$
decompose with $\size(\mathbf{x}_1^{(s)})=(j_1,k_1)$.
Furthermore,  ${\mathbf{x}}_1^{(s)}$ has all entries equal
to $\monCompNeutrElem$ or ${\mathbf{y}}_1^{(s)}$ does.
With strictly increasing $\mu$ and $\iota$, there is a unique  $0\leq\alpha\leq a$ such that $\mathbf{x}_1^{(1)},\ldots,\mathbf{x}_1^{(\alpha)}$ have entries different from $\monCompNeutrElem$, and $\mathbf{x}_1^{(\alpha+1)},\ldots,\mathbf{x}_1^{(a)}$ have all entries equal to $\monCompNeutrElem$.
Therefore, $\mathbf{P}_1$ and $\mathbf{Q}_1$ are of shape (\ref{eq1:lem518}) and (\ref{eq3:lem518}), respectively. 
Analogously one obtains $\mathbf{P}_2$ and $\mathbf{Q}_2$ of shape
(\ref{eq2:lem518}) and (\ref{eq4:lem518})
via factors derived from $\mathbf{w}_t$ and uniquely determined  $0\leq \beta\leq b$. 
Note $(\alpha,\beta)\in\domainAlphaBeta_{a,b}$ since $\mathbf{x}\not=\ec\not=\mathbf{y}$. 

If $\mathbf{P}\diag(\mathbf{a},\mathbf{b})\mathbf{Q}^\top=\diag(\mathbf{x}',\mathbf{y}')$ has a second decomposition with resulting splitting $(\alpha',\beta')\in\splitt_{a,b}$, then without loss of generality $\row(\mathbf{x})<\row(\mathbf{x}')$ and $\col(\mathbf{x})<\col(\mathbf{x}')$. 
With $\mathbf{v}_s,\mathbf{w}_t\in\compositionConnected$ this implies $\alpha<\alpha'$ or $\beta<\beta'$. 
\end{proof}

\begin{remark}\label{rem:quasishuffle_correspndence_bialgebra}
Let $\mathbf{a},\mathbf{b},\mathbf{v}$ and $\mathbf{w}$ be as in \Cref{lem:PdiagAbQChar}. 
For all $(\alpha,\beta)\in\domainAlphaBeta_{a,b}$ there is a one-to-one correspondence between 
\begin{enumerate}
\item 
the set of all  $$(\mathbf{P},\mathbf{Q})\in\bigcup_{j,k\in\N}\QSH(\row(\mathbf{a}),\row(\mathbf{b});j)\times\QSH(\col(\mathbf{a}),
\col(\mathbf{b});k)$$ which are simultaneously $(\alpha,\beta)$-decomposable, and 
\item 
the set of $4$-tuples 
$$(\begin{bmatrix}\mathbf{P}_{11}^{(\alpha)}&\mathbf{P}_{12}^{(\beta)}\end{bmatrix}
,
\begin{bmatrix}\mathbf{P}_{21}^{(\alpha)}&\mathbf{P}_{22}^{(\beta)}\end{bmatrix},
\begin{bmatrix}\mathbf{Q}_{11}^{(\alpha)}&\mathbf{Q}_{12}^{(\beta)}\end{bmatrix}
,
\begin{bmatrix}\mathbf{Q}_{21}^{(\alpha)}&\mathbf{Q}_{22}^{(\beta)}\end{bmatrix})$$
from 
$$
\bigcup_{j_1,j_2,k_1,k_2\in\N}\N^{j_1\times\row(\diag(\mathbf{a},\mathbf{b}))}
\times
\N^{j_2\times\row(\diag(\mathbf{a},\mathbf{b}))}
\times
\N^{k_1\times\col(\diag(\mathbf{a},\mathbf{b}))}
\times
\N^{k_2\times\col(\diag(\mathbf{a},\mathbf{b}))}$$
with right invertible components of the form (\ref{eq1:lem518}), (\ref{eq2:lem518}), (\ref{eq3:lem518}) and (\ref{eq4:lem518}).
\end{enumerate}
\end{remark}

We now verify the bialgebra property in order to conclude \Cref{theoroem:hopfalgebra}.
\begin{theorem}\label{theoroem:bialgebra}
$(\groundRing\langle\compositionConnected\rangle,\qShuffle,\eta,\Delta,\eps)$ is a graded, connected  bialgebra.
\end{theorem}
\begin{proof}
For fixed $\mathbf{a},\mathbf{b}\in\composition$ let $\mathbf{v}\in\compositionConnected^{a}$ and $ \mathbf{w}\in\compositionConnected^{b}$ denote the decompositions into connected compositions with block sizes 
 $(u_{\alpha},v_{\alpha})=\size({\mathbf{v}}_{\alpha})$ and $({\omega}_\beta, w_\beta)=\size({\mathbf{w}}_{\beta})$, respectively.  
Then, with \Cref{lem:PdiagAbQChar}, 
\begingroup
\allowdisplaybreaks
\begin{align*}
\Delta\,\circ&\qShuffle(\mathbf{a}\otimes\mathbf{b})=
\sum_{\mathbf{P},\mathbf{Q}}\Delta(\mathbf{P}\diag(\mathbf{a},\mathbf{b})\mathbf{Q}^\top)\\
&=\sum_{\mathbf{P},\mathbf{Q}}\left(\mathbf{P}\diag(\mathbf{a},\mathbf{b})\mathbf{Q}^\top\otimes \ec + \ec\otimes \mathbf{P}\diag(\mathbf{a},\mathbf{b})\mathbf{Q}^\top+\sum_{\substack{\mathbf{x},\mathbf{y}\in\composition\setminus\{\ec\}\\\diag(\mathbf{x},\mathbf{y})=\mathbf{P}\diag(\mathbf{a},\mathbf{b})\mathbf{Q}^\top}}\mathbf{x}\otimes\mathbf{y}\right)\\
&\overset{(\ref{lemEq:PdiagAbQChar})}=\left(\sum_{\mathbf{P},\mathbf{Q}}\mathbf{P}\diag(\mathbf{a},\mathbf{b})\mathbf{Q}^\top\right)\otimes \ec + \ec\otimes \left(\sum_{\mathbf{P},\mathbf{Q}}\mathbf{P}\diag(\mathbf{a},\mathbf{b})\mathbf{Q}^\top\right)\\
&\;\;\;+\sum_{\mathbf{P},\mathbf{Q}}\sum_{(\alpha,\beta)\in \splitt_{\mathbf{P},\mathbf{Q}}^{\mathbf{a},\mathbf{b}}}
\mathbf{P}_{11}^{(\alpha)}\mathbf{a}{\mathbf{Q}_{11}^{(\alpha)}}^\top
\star
\mathbf{P}_{12}^{(\beta)}\mathbf{b}{\mathbf{Q}_{12}^{(\beta)}}^\top
\otimes
\mathbf{P}_{21}^{(\alpha)}\mathbf{a}{\mathbf{Q}_{21}^{(\alpha)}}^\top
\star
\mathbf{P}_{22}^{(\beta)}\mathbf{b}{\mathbf{Q}_{22}^{(\beta)}}^\top\\
&=\mathbf{a}\qShuffle\mathbf{b}\otimes \ec + \ec\otimes \mathbf{a}\qShuffle\mathbf{b}\\
&\;\;\;+\sum_{\mathbf{P},\mathbf{Q}}\sum_{\substack{(\alpha,\beta)\in\domainAlphaBeta_{a,b}\\
\text{such that }\mathbf{P},\mathbf{Q}\\
\text{ $(\alpha,\beta)$-decomposable}
}}
\mathbf{P}_{11}^{(\alpha)}\mathbf{a}{\mathbf{Q}_{11}^{(\alpha)}}^\top
\star
\mathbf{P}_{12}^{(\beta)}\mathbf{b}{\mathbf{Q}_{12}^{(\beta)}}^\top
\otimes
\mathbf{P}_{21}^{(\alpha)}\mathbf{a}{\mathbf{Q}_{21}^{(\alpha)}}^\top
\star
\mathbf{P}_{22}^{(\beta)}\mathbf{b}{\mathbf{Q}_{22}^{(\beta)}}^\top\\
&=\mathbf{a}\qShuffle\mathbf{b}\otimes \ec + \ec\otimes \mathbf{a}\qShuffle\mathbf{b}\\
&\;\;\;+\sum_{(\alpha,\beta)\in\domainAlphaBeta_{a,b}}
\sum_{\substack{
\mathbf{P},\mathbf{Q}\\
\text{ $(\alpha,\beta)$-decomposable}
}}
\mathbf{P}_{11}^{(\alpha)}\mathbf{a}{\mathbf{Q}_{11}^{(\alpha)}}^\top
\star
\mathbf{P}_{12}^{(\beta)}\mathbf{b}{\mathbf{Q}_{12}^{(\beta)}}^\top
\otimes
\mathbf{P}_{21}^{(\alpha)}\mathbf{a}{\mathbf{Q}_{21}^{(\alpha)}}^\top
\star
\mathbf{P}_{22}^{(\beta)}\mathbf{b}{\mathbf{Q}_{22}^{(\beta)}}^\top\\
&\overset{\ref{rem:quasishuffle_correspndence_bialgebra}}=\mathbf{a}\qShuffle\mathbf{b}\otimes \ec + \ec\otimes \mathbf{a}\qShuffle\mathbf{b}\\
&\;\;\;+\sum_{(\alpha,\beta)}
\sum_{\substack{
{[\mathbf{P}_{11}^{(\alpha)}\;\mathbf{P}_{12}^{(\beta)}]}\\
{[\mathbf{Q}_{11}^{(\alpha)}\;\mathbf{Q}_{12}^{(\beta)}]}}}
\sum_{\substack{
{[\mathbf{P}_{21}^{(\alpha)}\;\mathbf{P}_{22}^{(\beta)}]}\\
{[\mathbf{Q}_{21}^{(\alpha)}\;\mathbf{Q}_{22}^{(\beta)}]}}}
\mathbf{P}_{11}^{(\alpha)}\mathbf{a}{\mathbf{Q}_{11}^{(\alpha)}}^\top
\star
\mathbf{P}_{12}^{(\beta)}\mathbf{b}{\mathbf{Q}_{12}^{(\beta)}}^\top
\otimes
\mathbf{P}_{21}^{(\alpha)}\mathbf{a}{\mathbf{Q}_{21}^{(\alpha)}}^\top
\star
\mathbf{P}_{22}^{(\beta)}\mathbf{b}{\mathbf{Q}_{22}^{(\beta)}}^\top\\
&\overset{(\ref{eq2:proof_bialgebra},\ref{eq3:proof_bialgebra})}=(\qShuffle\otimes\qShuffle)\circ(\id\otimes\,\tau\otimes\id)\circ(\Delta\otimes\Delta)(\mathbf{a}\otimes\mathbf{b}),
\end{align*}
\endgroup
where the 
summation over all tuples $$(\begin{bmatrix}\mathbf{P}_{11}^{(\alpha)}&\mathbf{P}_{12}^{(\beta)}\end{bmatrix},\begin{bmatrix}\mathbf{Q}_{11}^{(\alpha)}&\mathbf{Q}_{12}^{(\beta)}\end{bmatrix})$$ 
yields
\begin{align}
\sum_{\substack{
{[\mathbf{P}_{11}^{(\alpha)}\;\mathbf{P}_{12}^{(\beta)}]}\\
{[\mathbf{Q}_{11}^{(\alpha)}\;\mathbf{Q}_{12}^{(\beta)}]}}}
\mathbf{P}_{11}^{(\alpha)}\mathbf{a}{\mathbf{Q}_{11}^{(\alpha)}}^\top
\star
\mathbf{P}_{12}^{(\beta)}\mathbf{b}{\mathbf{Q}_{12}^{(\beta)}}^\top
&=
\sum_{\substack{
\widetilde{\mathbf{P}}_1\in\QSH(\underline{u}_{\alpha},\underline{\omega}_{\beta};j_1)\\
\widetilde{\mathbf{Q}}_1\in\QSH(\underline{v}_{\alpha},\underline{w}_{\beta};k_1)}}
\widetilde{\mathbf{P}}_1\diag(\mathbf{v}_1,\ldots,\mathbf{v}_\alpha,\mathbf{w}_1,\ldots,\mathbf{w}_\beta){\widetilde{\mathbf{Q}}_1}^\top\nonumber\\
&=
\diag(\mathbf{v}_1,\ldots,\mathbf{v}_{\alpha})
\qShuffle
\diag(\mathbf{w}_1,\ldots,\mathbf{w}_{\beta})\label{eq2:proof_bialgebra},
\end{align}
and in the same manner,
\begin{equation}\label{eq3:proof_bialgebra}
\sum_{\substack{
{[\mathbf{P}_{21}^{(\alpha)}\;\mathbf{P}_{22}^{(\beta)}]}\\
{[\mathbf{Q}_{21}^{(\alpha)}\;\mathbf{Q}_{22}^{(\beta)}]}}}
\mathbf{P}_{21}^{(\alpha)}\mathbf{a}{\mathbf{Q}_{21}^{(\alpha)}}^\top
\star
\mathbf{P}_{22}^{(\beta)}\mathbf{b}{\mathbf{Q}_{22}^{(\beta)}}^\top
=
\diag(\mathbf{v}_{\alpha+1},\ldots,\mathbf{v}_a)
\qShuffle
\diag(\mathbf{w}_{\beta+1},\ldots,\mathbf{w}_b)
\end{equation}
for the second component in the tensor product. 
\end{proof}

\section{Conclusions and outlook}

In \Cref{def:ss} we introduce the two-parameter sums signature of a multi-dimensional function $Z:\N^2\rightarrow\groundRing^d$.
It stores all polynomial warping invariants (\Cref{theorem:invariants}, \Cref{sec:polynomialInvariant})
as a linear functional on the linear span of 
matrix compositions.

These invariants are compatible with a Hopf algebra on matrix compositions:
they satisfy a ``quasi-shuffle identity'',
which means that the sums signature is a character with respect to the quasi-shuffle (\Cref{lem:quasiShuffleRel}). 
Compatibility with a concatenation-type coproduct is encoded in a (weak) form
of Chen's identity, \Cref{thm:chen}.

The underlying Hopf algebra of matrix compositions is,
in the case of a one-dimensional ambient
space ($d=1$), isomorphic
to a sub-Hopf algebra of formal power series.
These are akin to quasisymmetric functions,
where the underlying poset of the natural
numbers has been replaced by the poset $\N^2$.
We therefore call them two-parameter quasisymmetric functions.

There remain several directions for future work.
\begin{itemize}

  \item
    The recent extension of Chen's iterated \emph{integrals} in \cite{GLNO2022} has properties similar to our two-parameter sums signature. 
    \emph{In which sense does the discrete setting converge to the continuous?}
    \emph{How are shuffles and quasi-shuffles related to each other? Is there a continuous version of diagonal concatenation that leads to Chen's identity with respect to diagonal deconcatenation?}

    \item
    In this article we restrict to $p=2$ parameters.
    Many results remain true for $p>2$, 
    when adjusting the row and column operations to tensor operations over arbitrary axes.
    This quickly leads to notational clutter. 
    \emph{Can the language of restriction species (and decomposition spaces)
    or $B_\infty$-structures
    simplify the proofs of the bialgebra properties?}

    \item 
    For two parameters, there are several operations that can take on the role of concatenation (and thus deconcatenation).
    We have presented a ``diagonal'' concatenation, 
    with a resulting (weak) notion of Chen's identity. 

    Mutatis mutandis, one could instead concatenate along the anti-diagonal.
    \emph{In which sense are those concepts equivalent (if at all)?}
    Alternatively,
    concatenation along just one axis (as in \cite{GLNO2022}) yields a modified Chen's relation when restricted to certain matrix compositions. 
    \emph{Is there a combination of several Chen-like formulas, which results in an efficient dynamic programming method?}

    \item
    Regarding applications, it will be interesting
    to investigate how the sums that we propose as feature
    for, say, images, compare to the integral features used in \cite{ZLT22}.
    We believe that, as in the one-parameter case, sums (instead of integrals)
    might provide a richer set of features.

    \item
    We introduced two-parameter quasisymmetric functions as a natural generalization of the (classical, one-parameter) quasisymmetric functions. 
    We presented a basis which corresponds to the monomial basis in the one-parameter case.
    \emph{Are there relevant analogons of other bases, for example the fundamental basis?}
    \emph{Closely related, what is a suitable notion of refinement for matrix compositions?}

    \item
    We provided an iterative evaluation procedure of signature coefficients for certain matrix compositions $\closureDiag$.
    Similar methods (with linear complexity) would be possible for matrix compositions based on chaining on the anti-diagonal. 
    \emph{What is the general complexity for arbitrary matrix compositions?}

    \item
    The quasi-shuffle that we introduce
    can be considered as two interacting (one-parameter) quasi-shuffles
    of rows and columns, \Cref{theo_char_qs}.
    It is therefore plausible, and in fact true,
    that our product possesses a tridendriform structure
    (in fact, there is one for the row-point-of-view
    and one for the column-point-of-view).
    Moreover, these two tridendriform structures
    combine into an ennea structure \cite{leroux2004ennea}, with $3^2=9$ operations.
    
    There is already an ennea structure on
    one-parameter quasi-shuffles, by pinning the last \emph{and} the first
    letter (compare \cite[Example 1.8]{aguiar2004quadri}).
    ``Tensoring'' this gives a (maybe interesting) structure with $9^2 = 81$ operations.

    \item The Hopf algebra from   \Cref{theoroem:hopfalgebra} guaranteed that the algebra of  matrix compositions if free. 
   \emph{Is there a free generating set akin to Lyndon words?}
\end{itemize}

\section*{Statements and Declarations}

\subsection*{Competing Interests}
The authors have been supported by the trilateral French-Japanese-German research project \emph{EnhanceD Data stream Analysis with the Signature Method} (EDDA) of the French National Research Agency (ANR), together with the Japan Science and Technology Agency (JST), and the Deutsche Forschungsgemeinschaft (DFG). 
\\\\
Many discussions took place at the event \emph{New interfaces of Stochastic Analysis and Rough Paths}, hosted by the Banff International Research Station (BIRS workshop '22, Banff, Canada), the \emph{International Conference on Scientific Computation and Differential Equations} (SciCADE '22) hosted by the University of Iceland, (Reykjavik, Iceland), during the virtual event \emph{AlCoVE: an Algebraic Combinatorics Virtual Expedition '22}, and at the workshop \emph{Stochastic and Rough Analysis '22} hosted by the FU, HU, TU and WIAS Berlin, together with the University of Potsdam at the Harnack-Haus (Berlin, Germany). 

\subsection*{Acknowledgements}

For various discussions, we thank 
Harald Oberhauser, 
Darrick Lee, 
Samy Tindel,
Fabian Harang, 
Georges Oppenheim, 
Amy Pang, 
Anthony Lazzeroni, 
Weixin Yang,
Danyu Yang and  
Terry Lyons. 
\\\\
L.S. would like to thank 
the Stochastic Analysis Research Group in the Mathematical Institute, University of Oxford (Oxford, UK) for the invitation in June '22.


\newcommand{\etalchar}[1]{$^{#1}$}

\printindex[general]

\end{document}